\documentclass[3p,times]{elsarticle}

\usepackage{amssymb}
\usepackage{amsthm}
\usepackage{mathtools}
\allowdisplaybreaks

\usepackage[figuresright]{rotating}

\theoremstyle{thmstyleone}%
\newtheorem{theorem}{Theorem}[section]
\newtheorem{lemma}[theorem]{Lemma}
\newtheorem{corollary}[theorem]{Corollary}
\theoremstyle{thmstylethree}%
\newtheorem{definition}{Definition}%

\newcommand{\V}[1]{{\bf #1}} 
\newcommand{\del}{\mathrm{d}}
\newcommand{\ellipse}{\mathbb{E}}
\newcommand{\inum}{\mathrm{i}}
\newcommand{\vt}{\varphi} 
\newcommand{\rhoana}{\rho_\mathrm{A}}
\newcommand{\minfre}[1]{\check{#1}} 
\newcommand{\intge}[1]{\mathbb{Z}_{\ge #1}} 
\newcommand{\bd}[1]{\mathrm{bd}(#1)}
\newcommand{\rhp}{\mathbb{C}_{> 0}} 
\newcommand{\trans}{\top}
\newcommand{\basis}[1]{\tilde{#1}} 
\newcommand{\eset}{\mathcal{E}} 
\newcommand{\pset}{\mathcal{P}} 
\newcommand{\rset}{\mathcal{R}} 
\newcommand{\vtset}{\mathcal{F}} 
\newcommand{\aset}{\mathbb{A}} 
\newcommand{\imset}{\mathbb{I}} 

\DeclareMathOperator{\RE}{Re}
\DeclareMathOperator{\IM}{Im}
\DeclareMathOperator{\sn}{sn}
\DeclareMathOperator{\cn}{cn}
\DeclareMathOperator{\dn}{dn}
\DeclareMathOperator{\cs}{cs}
\DeclareMathOperator{\cd}{cd}
\DeclareMathOperator{\dc}{dc}
\DeclareMathOperator{\sd}{sd}
\DeclareMathOperator{\am}{am}
\DeclareMathOperator{\arcdn}{arcdn}
\DeclareMathOperator{\enum}{e}
\DeclareMathOperator{\dist}{dist}
\DeclareMathOperator{\EK}{K}
\DeclareMathOperator{\EE}{E}
\DeclareMathOperator{\TV}{TV}
\DeclareMathOperator*{\mint}{\int \cdots \int}
\DeclareMathOperator{\expm1}{expm1}

\begin{document}

\begin{frontmatter}


\title{Exponential sum approximations of finite completely monotonic functions}
\author{Yohei M. Koyama}
\ead{ym.koyama@riken.jp}
\address{RIKEN, 6-7-1 Minatojima-minamimachi, Chuo-ku, Kobe, Hyogo 650-0047, Japan}

\begin{abstract}
Bernstein's theorem (also called Hausdorff--Bernstein--Widder theorem) enables the integral representation of a completely monotonic function. We introduce a finite completely monotonic function, which is a completely monotonic function with a finite positive integral interval of the integral representation. We consider the exponential sum approximation of a finite completely monotonic function based on the Gaussian quadrature with a variable transformation. If the variable transformation is analytic on an open Bernstein ellipse, the maximum absolute error decreases at least geometrically with respect to the number of exponential functions. The maximization of the decreasing rate of the error bound can be achieved by using a variable transformation represented by Jacobi's delta amplitude function (also called dn function). The error curve is expanded by introducing basis functions, which are eigenfunctions of a fourth order differential operator, satisfy orthogonality conditions, and have the interlacing property of zeros by Kellogg's theorem.
\end{abstract}


\end{frontmatter}


\section{Introduction}

For a given $X \in \mathbb{R}$, let $\intge{X} \coloneqq \{ n \in \mathbb{Z} \mid n \ge X \}$. Let $n \in \intge{0}$. If a function $f(x)$ is infinitely differentiable and satisfies $(-1)^n f^{(n)}(x) \ge 0$ for $0 < x < \infty$, $f(x)$ is called completely monotonic in $0 < x < \infty$. If $f(x)$ is completely monotonic in $0 < x < \infty$ and satisfies $\lim_{x \downarrow 0} f(x) = f(0) < \infty$, $f(x)$ is called completely monotonic in $0 \le x < \infty$. For example, $f(x) = x^{-\eta}$ with $0 < \eta < \infty$ is not completely monotonic in $0 \le x < \infty$ but  completely monotonic in $0 < x < \infty$. $f(x) = \enum^{-ax}$ with $0 \le a < \infty$ and $f(x) = 1 / (x + 1)$ are completely monotonic in $0 \le x < \infty$. Bernstein's theorem (also called Hausdorff--Bernstein--Widder theorem \cite{Hausdorff1921,Hausdorff1921a,Bernstein1929,Widder1931}) states that a necessary and sufficient condition that $f(x)$ should be completely monotonic in $0 < x < \infty$ is that $f(x) = \int_0^\infty \enum^{-xt} \del W(t)$, where $W(t)$ is nondecreasing and the integral converges for $0 < x < \infty$ \cite[Theorem IV.12b]{Widder1941}. If $f(x)$ is completely monotonic in $0 \le x < \infty$, then $f(x) = \int_0^\infty \enum^{-xt} \del W(t)$, where $W(t)$ is bounded and nondecreasing for $0 \le t < \infty$ \cite[Theorem IV.19b]{Widder1941}. $W(t)$ is called the determining function. For example, by the Euler integral formula, the integral representation of $x^{-\eta}$ is given as \cite[Anhang. 2.]{Doetsch1937}
\begin{equation}
\label{eq:f_inv}
x^{-\eta} = \frac{1}{\Gamma(\eta)} \int_0^\infty \enum^{-x t} t^{\eta - 1} \del t, 0 < x < \infty, 0 < \eta < \infty,
\end{equation}
whose determining function is given as $W(t) = \int_0^t \frac{ s^{\eta - 1} }{\Gamma(\eta)} \del s = \frac{ t^\eta }{ \Gamma(\eta + 1) }$, $0 \le t < \infty$.

The integral representation by Bernstein's theorem implies a completely monotonic function can be approximated by the sum of exponential functions. We describe some known results of the approximation. We denote a set of exponential sum at most $M \in \intge{1}$ terms as
\begin{equation}
\eset_M \coloneqq \left\{ f \middle\vert f(x) = \sum_{\nu = 1}^M c_\nu \enum^{- t_\nu x}, t_\nu \in \mathbb{R}, c_\nu \in \mathbb{R}, \nu = 1, \ldots, M, x \in \mathbb{R} \right\},
\end{equation}
and the best uniform approximation error of a function $F(x)$ on $x \in \mathbb{X}$ from a set of functions $\mathcal{F}$ as $\dist(F, \mathcal{F}, \mathbb{X}) \coloneqq \inf\limits_{G \in \mathcal{F}} \sup\limits_{x \in \mathbb{X}} \lvert F(x) - G(x) \rvert$. Kammler \cite{Kammler1976} obtained the conditions for the existence of the unique best exponential sum approximation of a completely monotonic function in $0 \le x < \infty$. We extracted the results relevant to the current study in our convenient form as follows:
\begin{lemma}[Kammler {\cite[a part of Lemma 4]{Kammler1976}}]
\label{lem:kammler_zero}
Let $f(x)$ be a completely monotonic function in $0 \le x < \infty$. Let $M \in \intge{1}$ and $f_M \in \eset_M$. If $f(x) - f_M(x)$ has more than $2M$ nonnegative zeros, then $f_M = f$. If $f(x) - f_M(x)$ has exactly $2M$ nonnegative zeros $0 \le x_1 < \cdots < x_{2M} < \infty$, then $f_M(x) = \sum_{\nu = 1}^M c_\nu \enum^{-t_\nu x}$ satisfies $0 < t_1 < \cdots < t_M < \infty$, $c_\nu > 0$, $\nu = 1, \ldots, M$, and $f(x) - f_M(x) > 0$ for $x_{2M} < x < \infty$.
\end{lemma}
\begin{theorem}[Kammler {\cite[a part of Lemma 5 and Theorem1]{Kammler1976}}]
\label{the:kammler_best}
Let $0 \le a < b \le  \infty$ and $M \in \intge{1}$. Let $W(t)$ be a bounded nondecreasing function with at least $M + 1$ points of increase for $a \le t \le b$. Let $f(x) \coloneqq \int_a^b \enum^{-xt} \del W(t)$ for $0 \le x < \infty$. Let $f(x)$  satisfy $\lim\limits_{x \to \infty} f(x) = 0$, i.e., $a > 0$ or $a = 0$ and $\lim\limits_{t \downarrow 0} W(t) = W(0)$. Then there exists a unique best exponential sum $f_M(x) = \sum_{\nu = 1}^M c_\nu \enum^{-t_\nu x}$, which satisfies
\begin{subequations}
\begin{align}
&a < t_1 < \cdots < t_M < b, c_\nu > 0, \nu = 1, \ldots, M, \\
\label{eq:equioscillation}
&f(x_i) - f_M(x_i) = (-1)^i \dist(f, \eset_M, [0, \infty)), i = 0, \ldots, 2M, 0 \eqqcolon x_0 < \cdots < x_{2M} < \infty.
\end{align}
\end{subequations}
\end{theorem}

Kammler showed $L^2$ approximation error of a completely monotonic function $f(x) = \int_a^b \enum^{-xt} \del W(t)$, $0 < a < b < \infty$, decreases at least geometrically as follows:
\begin{theorem}[Kammler \cite{Kammler1981}]
\label{the:KammlerL2}
Let $0 < a < b < \infty$, $M \in \intge{1}$, and $\lVert F \rVert_2 \coloneqq \sqrt{\int_0^\infty \lvert F(x) \rvert^2 \del x}$. Let $W(t)$ be a bounded nondecreasing function for $a \le t \le b$. For given $a \le t_1 < \cdots < t_M \le b$, one has
\begin{equation}
\min_{(c_1, \ldots, c_M) \in \mathbb{R}^M} \left\lVert \int_a^b \enum^{-x t} \del W(t) - \sum_{\nu = 1}^M c_\nu \enum^{- t_\nu x} \right\rVert_2 \le \int_a^b \frac{1}{\sqrt{2t}} \prod_{\nu = 1}^M \left\lvert \frac{ t - t_\nu }{ t + t_\nu } \right\rvert \del W(t) \le \frac{\int_a^b \del W(t)}{\sqrt{2 a}} \left( \frac{b - a}{b + a} \right)^M.
\end{equation}
\end{theorem}
Braess and Saff \cite[Theorem VI.3.4]{Braess1986} and Braess and Hackbusch \cite{Braess2009} showed the best uniform approximation error for $0 < a \le x \le b < \infty$ decreases at least geometrically as follows:
\begin{theorem}[Braess and Hackbusch {\cite[Lemma 2.1, Theorem 3.1, and Theorem 3.3]{Braess2009}}]
\label{the:braess}
Let $f(x)$ be completely monotonic in $0 \le x < \infty$. Let $\EK(k)$ be the complete elliptic integral of the first kind with the modulus $k$. For $M \in \intge{1}$, one has 
\begin{equation}
\dist(f, \eset_M, [a, b]) \le 8 f(0) \rho^{-2M}, 0 < a < b < \infty, \rho \coloneqq \enum^{\pi \frac{ \EK(a/b) }{ \EK(\sqrt{1 - (a/b)^2}) } }.
\end{equation}
\end{theorem}

In this study, we consider exponential sum approximations of a completely monotonic function $f(x)$ in $0 \le x < \infty$, whose integral representation is given as $f(x) = \int_a^b \enum^{-xt} \del W(t)$ with $0 < a < b < \infty$. If the integral interval of the Laplace transform is finite, the transform is called the finite Laplace transform \cite{Bertero1985}\cite[Sec.~4.8]{Villiers2016}. In a similar way, we call $f(x)$ the finite completely monotonic function. By Bernstein's theorem, $W(t)$ is a bounded nondecreasing function for $a \le t \le b$. By Lemma \ref{lem:fl_entire}, $f(z) = \int_a^b \enum^{-zt} \del W(t)$, $z \in \mathbb{C}$, is an entire function and satisfies $f^{(n)}(z) = (-1)^n \int_a^b \enum^{-zt} t^n \del W(t)$ for $n \in \intge{1}$. Thus a finite completely monotonic function can be analytically continuable to $\mathbb{C}$ (i.e., entire function) \cite{Varga1968} and satisfies $(-1)^n f^{(n)}(x) = \int_a^b \enum^{-xt} t^n \del W(t) \ge 0$ for $-\infty < x < \infty$, $n \in \intge{0}$ (i.e., completely monotonic in $-\infty < x < \infty$).

In Section \ref{sec:gauss}, we investigate the exponential sum approximation of a finite completely monotonic function based on the Gaussian quadrature with a variable transformation. For a variable transformation analytic on an open Bernstein ellipse, we obtain the maximum absolute error bound (Theorem \ref{the:Evt}) and the expansion of error by introducing basis functions associated with the variable transformation (Theorem \ref{the:f_gauss_basis}). In Section \ref{sec:max_hatrho}, we construct a variable transformation that maximizes the decreasing rate of the obtained error bound (Theorem \ref{the:Phi}). We show the corresponding basis functions are eigenfunctions of a fourth order differential operator (Theorem \ref{the:L}), satisfy orthogonality conditions (Theorem \ref{the:basis_Phi_orthogonal}), and have interlacing property of zeros (Theorem \ref{the:basis_Phi_interlacing}). In Section \ref{sec:numerical}, we describe the numerical implementation and conducts the numerical experiments by using finite completely monotonic functions associated with the inverse power function \eqref{eq:f_inv}. We obtain a bound (Theorem \ref{the:Eh_ineq}), which can be used to select a spacing parameter required for the initialization of the Remez algorithm based on a Gaussian quadrature to compute the best exponential sum approximation of a finite completely monotonic function.

\section{Exponential sum approximation based on the Gaussian quadrature}
\label{sec:gauss}

\subsection{Variable transformation}
\label{sec:vt}
Let $0 < a < b < \infty$ and $0 \le x < \infty$. We consider the exponential sum approximation of a finite completely monotonic function $f(x) = \int_a^b \enum^{-xt} \del W(t)$ based on the Gaussian quadrature. In \ref{app:gauss}, we summarized properties and nomenclature of Gaussian quadrature used in the current study. In general, the approximation of an integral by the Gaussian quadrature is changed by applying a nonlinear change of variable of the integral and it affects the error of the approximation. By the change of variable $t = b \tau$, we have $f(x) = \int_{a/b}^1 \enum^{-x b \tau} \del W(b \tau)$. We consider the variable transformations to incorporate the nonlinear effect and  to convert the integral interval $[a/b, 1]$ into the standard integral interval $[-1, 1]$ as follows:
\begin{definition}
For a given $0 < r < 1$, $\vtset_r$ represents a function space, whose element $\vt_r(u)$ defined on $[-1, 1]$ satisfies
\begin{subequations}
\label{eq:vt_increase_eq}
\begin{align}
\label{eq:vt_eq}
&\text{$\vt_r(-1) = r$ and $\vt_r(1) = 1$}, \\
\label{eq:vt_increase}
&\text{$\vt_r(u)$ is continuous and strictly increasing on $[-1, 1]$}.
\end{align}
\end{subequations}
\end{definition}
Then a finite Laplace--Stieltjes transform can be expressed as follows:
\begin{lemma}
Let $0 < a < b < \infty$. Let $W(t)$ be of bounded variation in $a \le t \le b$. Let $\vt_{a/b} \in \vtset_{a/b}$. Then one has
\begin{equation}
\label{eq:f_m1_1}
f(z) \coloneqq \int_a^b \enum^{ - z t } \del W(t) = \int_{-1}^1 \enum^{- b z \vt_{a / b}(u)} \del W(b \vt_{a / b}(u)), z \in \mathbb{C}.
\end{equation}
\end{lemma}
\begin{proof}
The statement follows from Theorem \ref{the:stieltjes_change} and \eqref{eq:vt_increase_eq}.
\end{proof}
The existence of the Gaussian quadrature with the variable transformation is guaranteed under the following conditions:
\begin{lemma}
\label{lem:vt_gauss}
Let $0 < a < b < \infty$ and $M \in \intge{1}$. Let $W(t)$ be a bounded nondecreasing function with at least $M + 1$ points of increase for $a \le t \le b$. Let $f(z) \coloneqq \int_a^b \enum^{ - z t } \del W(t)$ for $z \in \mathbb{C}$. Let $\vt_{a/b} \in \vtset_{a/b}$. Then there exists $u_\nu$ and $c_\nu$, $\nu = 1, \ldots, M$, satisfying
\begin{align}
\label{eq:vt_gauss_uc_ineq}
&\text{$-1 < u_1 < \cdots < u_M < 1$ and $c_\nu > 0, \nu = 1, \ldots, M$}, \\
\label{eq:vt_gauss_exact}
&\int_{-1}^1 p(u) \del W(b \vt_{a / b}(u)) = \sum_{\nu = 1}^M c_\nu p(u_\nu), p \in \pset_{2M-1}, \\
\label{eq:vt_gauss_t_ineq}
&a < b \vt_{a / b}(u_1) < \cdots < b \vt_{a / b}(u_M) < b, \\
\label{eq:f0_c}
&0 < f(0) = \int_a^b \del W(t) = \int_{-1}^1\del W(b \vt_{a / b}(u)) = \sum_{\nu = 1}^M c_\nu < \infty, \\
\label{eq:Evtx_exp}
&\left\lvert f(x) - \sum_{\nu = 1}^M c_\nu \enum^{- b \vt_{a / b}(u_\nu) x} \right\rvert \le \lvert \enum^{-ax} - \enum^{-bx} \rvert f(0), x \in \mathbb{R}.
\end{align}
\end{lemma}
\begin{proof}
Under the conditions, $W(b \vt_{a / b}(u))$ is a bounded nondecreasing function with at least $M + 1$ points of increase for $-1 \le u \le 1$. Then by Theorem \ref{the:gauss}, there exists $u_\nu$ and $c_\nu$, $\nu = 1, \ldots, M$, satisfying \eqref{eq:vt_gauss_uc_ineq} and \eqref{eq:vt_gauss_exact}. \eqref{eq:vt_gauss_t_ineq} follows from \eqref{eq:vt_increase_eq} and \eqref{eq:vt_gauss_uc_ineq}. \eqref{eq:f0_c} follows from \eqref{eq:f_m1_1}, \eqref{eq:vt_gauss_exact} with $p(u) = 1$, Theorem \ref{the:stieltjes_pos}, and $\int_a^b \del W(t) = W(b) - W(a) < \infty$. By using \eqref{eq:gauss_error_dist}, $\dist(\enum^{- b \vt_{a / b}(u) x}, \mathcal{P}_0, [-1, 1]) = \lvert \enum^{-ax} - \enum^{-bx} \rvert / 2$ for a given $x \in \mathbb{R}$, and \eqref{eq:f0_c}, we have \eqref{eq:Evtx_exp}.
\end{proof}

\subsection{Maximum absolute error bound}
\label{sec:error}
We consider error bounds of Gaussian quadrature for \eqref{eq:f_m1_1}, whose integrand is $\enum^{-bx \vt_{a/b}(u)}$ for $-1 \le u \le 1$. In \ref{app:gauss}, we summarized error bounds of Gaussian quadrature used in the current study. The Achieser--Stenger bound \eqref{eq:achieser_stenger} requires the analyticity of the integrand on an interior of an ellipse
\begin{equation}
\ellipse_\rho \coloneqq \left\{ z \in \mathbb{C} \middle\vert \frac{ (\RE z)^2 }{((\rho + \rho^{-1}) / 2)^2} + \frac{ (\IM z)^2 }{((\rho - \rho^{-1}) / 2)^2} < 1 \right\}
\end{equation}
for some $1 < \rho < \infty$, which is called the open Bernstein ellipse. The boundary, denoted as $\bd{\ellipse_\rho}$, is called the Bernstein ellipse, whose foci are $(\pm1, 0)$ and $\rho$ corresponds to the sum of the semimajor axis $(\rho + \rho^{-1}) / 2$ and the semiminor axis $(\rho - \rho^{-1}) / 2$ \cite[p.~56]{Trefethen2013}. For convenience, we set $\ellipse_\infty \coloneqq \mathbb{C}$.
We summarize the basic properties of an analytic function $F(u)$ on $\ellipse_R$ for some $1 < R \le \infty$ as follows:
\begin{theorem}[Laurent series, Fourier series, and Chebyshev series; {\cite[Theorem 8.1]{Trefethen2013}}]
\label{the:laurent}
Let $F(u)$ be an analytic function on $\ellipse_R$ for a given $1 < R \le \infty$. Let us introduce
\begin{align}
&\text{$\aset_R \coloneqq \{ z \in \mathbb{C} \mid R^{-1} < \lvert z \rvert < R \}$, where $\infty^{-1} = 0$}, \\
&\imset_Y \coloneqq \{ z \in \mathbb{C} \mid \lvert \IM z \rvert < Y \}, 0 < Y \le \infty.
\end{align}
Let $a_n \coloneqq \frac{1}{\pi} \int_0^\pi F(\cos\theta) \cos(n \theta) \del \theta$ for $n \in \intge{0}$. Let $T_n(z)$ be the $n$-degree Chebyshev polynomial of the first kind for $z \in \mathbb{C}$ and $n \in \intge{0}$. Then one has
\begin{align}
&\text{$F((s + s^{-1}) / 2)$ is analytic on $\aset_R$ and $F(\cos\theta)$ is analytic on $\imset_{\log(R)}$}, \\
&F((s + s^{-1}) / 2) = a_0 + \sum_{n = 1}^\infty a_n (s^n + s^{-n}), s \in \aset_R, \\
&F(\cos\theta) = a_0 + 2 \sum_{n = 1}^\infty a_n \cos(n \theta), \theta \in \imset_{\log(R)}, \\
&F(u) = a_0 + 2 \sum_{n = 1}^\infty a_n T_n(u), u \in \ellipse_R, \\
&\lvert a_n \rvert \le \rho^{-n} \sup\limits_{u \in \ellipse_\rho} \lvert F(u) \rvert < \infty, 1 < \rho < R, n \in \intge{0}, \\
&\lvert a_n \rvert \le R^{-n} \sup\limits_{u \in \ellipse_R} \lvert F(u) \rvert < \infty, n \in \intge{0}, \text{ if $R \ne \infty$ and $\sup\limits_{u \in \ellipse_R} \lvert F(u) \rvert < \infty$}.
\end{align}
\end{theorem}
To represent the largest open Bernstein ellipse to satisfy the analyticity of a function $F(u)$, we introduce an extended real number $\rhoana[F]$ as follows:
\begin{definition}
Let $F: [-1, 1] \to \mathbb{R}$. Then $\rhoana[F]$ is defined as
\begin{equation}
\rhoana[F] \coloneqq
\begin{cases}
1, & \text{if there exists $v \in [-1, 1]$ such that $F$ is not analytic at $v$}, \\
\rho, & \text{if there exists $\rho \in (1, \infty)$ and $v \in \bd{\ellipse_\rho}$ such that $F$ is} \\
& \text{analytically continuable to $\ellipse_\rho$ and is not analytic at $v$}, \\
\infty, & \text{if $F$ is analytically continuable to $\mathbb{C}$ (i.e., entire function)}.
\end{cases}
\end{equation}
\end{definition}
Let $0 < x < \infty$. To apply the Achieser--Stenger bound, the integrand $\enum^{-bx \vt_{a/b}(u)}$ in \eqref{eq:f_m1_1} is necessary to satisfy the inequality $\sup\limits_{u \in \ellipse_\rho} \lvert \RE \enum^{-bx \vt_{a/b}(u)} \rvert < \infty$. By using $\sup\limits_{u \in \ellipse_\rho} \lvert \RE \enum^{-bx \vt_{a/b}(u)} \rvert \le \enum^{-bx \inf\limits_{u \in \ellipse_\rho} \RE \vt_{a/b}(u)}$, the inequality is satisfied if $\inf\limits_{u \in \ellipse_\rho} \RE \vt_{a/b}(u) > -\infty$. Then we investigate the properties of $\inf\limits_{u \in \ellipse_\rho} \RE \vt_{a/b}(u)$ as follows:
\begin{lemma}
\label{lem:infRe}
Let $0 < r < 1$, $\vt_r \in \vtset_r$, and $\rhoana[\vt_r] > 1$. Then one has
\begin{align}
\label{eq:infRe_minRe}
&\inf_{u \in \ellipse_\rho} \RE \vt_r(u) = \min_{u \in \bd{\ellipse_\rho}} \RE \vt_r(u), 1 < \rho < \rhoana[\vt_r], \\
\label{eq:infRe_decrease}
&\text{$\inf_{u \in \ellipse_\rho} \RE \vt_r(u)$ is continuous and strictly decreasing for $1 < \rho < \rhoana[\vt_r]$}, \\
\label{eq:lim_infRe}
&\lim_{\rho \downarrow 1} \inf_{u \in \ellipse_\rho} \RE \vt_r(u) = r, \lim_{\rho \uparrow \rhoana[\vt_r]} \inf_{u \in \ellipse_\rho} \RE \vt_r(u) = \inf_{u \in \ellipse_{\rhoana[\vt_r]}} \RE \vt_r(u), \\
\label{eq:infRe_rhoana_r}
&\inf_{u \in \ellipse_{\rhoana[\vt_r]}} \RE \vt_r(u) < \inf_{u \in \ellipse_\rho} \RE \vt_r(u) < r, 1 < \rho < \rhoana[\vt_r], \\
\label{eq:infRe_minf}
&\text{$\inf_{u \in \ellipse_{\rhoana[\vt_r]}} \RE \vt_r(u) = \inf_{u \in \mathbb{C}} \RE \vt_r(u) = -\infty$ if $\rhoana[\vt_r] = \infty$}, \\
\label{eq:infRe_gt_minf}
&\text{$1 < \rhoana[\vt_r] < \infty$ if $\inf_{u \in \ellipse_{\rhoana[\vt_r]}} \RE \vt_r(u) > -\infty$}.
\end{align}
\end{lemma}
\begin{proof}
Under the conditions, $\vt_r(u)$ is a nonconstant analytic function on $\ellipse_{\rhoana[\vt_r]}$. Then $\RE \vt_r(u)$ is a nonconstant harmonic function on $\ellipse_{\rhoana[\vt_r]}$. By the minimum principle of the nonconstant harmonic function and the continuity of the harmonic function on $\ellipse_{\rhoana[\vt_r]}$, we have \eqref{eq:infRe_minRe} and \eqref{eq:infRe_decrease}. By using \eqref{eq:infRe_minRe}, the continuity on $\ellipse_{\rhoana[\vt_r]}$, and \eqref{eq:vt_increase_eq}, we have
\begin{equation}
\lim_{\rho \downarrow 1} \inf_{u \in \ellipse_\rho} \RE \vt_r(u) = \lim_{\rho \downarrow 1} \min_{u \in \bd{\ellipse_\rho}} \RE \vt_r(u) = \min_{u \in [-1, 1]} \RE \vt_r(u) = r,
\end{equation}
which leads to the first equality in \eqref{eq:lim_infRe}. The second equality in \eqref{eq:lim_infRe} is obvious. \eqref{eq:infRe_rhoana_r} follows from \eqref{eq:infRe_decrease} and \eqref{eq:lim_infRe}.

If $\rhoana[\vt_r] = \infty$, $\vt_r(u)$ is a nonconstant entire function. Then by little Picard theorem, we have $\{ \RE \vt_r(u) \mid u \in \mathbb{C} \} = \mathbb{R}$, which leads to \eqref{eq:infRe_minf}. \eqref{eq:infRe_gt_minf} follows from the contraposition of \eqref{eq:infRe_minf} and the condition $\rhoana[\vt_r] > 1$.
\end{proof}
Then we introduce an interval $I_{\vt_{a/b}}$ to satisfy $\inf\limits_{u \in \ellipse_\rho} \RE \vt_{a/b}(u) > -\infty$ for $\rho \in I_{\vt_{a/b}}$ as follows:
\begin{lemma}
\label{lem:Ivt}
Let $0 < r < 1$, $\vt_r \in \vtset_r$, and $\rhoana[\vt_r] > 1$. Let $I_{\vt_r}$ be an interval
\begin{equation}
\label{eq:Ivt}
I_{\vt_r} \coloneqq
\begin{cases}
(1, \rhoana[\vt_r]) & \text{if $\inf\limits_{u \in \ellipse_{\rhoana[\vt_r]}} \RE \vt_r(u) = - \infty$}, \\
(1, \rhoana[\vt_r]] & \text{if $\inf\limits_{u \in \ellipse_{\rhoana[\vt_r]}} \RE \vt_r(u) > - \infty$}.
\end{cases}
\end{equation}
Then one has
\begin{align}
\label{eq:Ivt_subset}
&(1, \rhoana[\vt_r]) \subseteq I_{\vt_r} \subseteq (1, \infty), \\
\label{eq:infRe_ineq}
&-\infty < \inf_{u \in \ellipse_\rho} \RE \vt_r(u) < r, \rho \in I_{\vt_r}.
\end{align}
\end{lemma}
\begin{proof}
Let $\inf\limits_{u \in \ellipse_{\rhoana[\vt_r]}} \RE \vt_r(u) = - \infty$. Then $I_{\vt_r} = (1, \rhoana[\vt_r])$ satisfies \eqref{eq:Ivt_subset} by $1 < \rhoana[\vt_r] \le \infty$. \eqref{eq:infRe_ineq} follows from \eqref{eq:infRe_rhoana_r}.

Let $\inf\limits_{u \in \ellipse_{\rhoana[\vt_r]}} \RE \vt_r(u) > - \infty$. $I_{\vt_r} = (1, \rhoana[\vt_r]]$ satisfies \eqref{eq:Ivt_subset} by $\rhoana[\vt_r] > 1$ and \eqref{eq:infRe_gt_minf}. By \eqref{eq:infRe_decrease} and \eqref{eq:lim_infRe}, we have $-\infty < \inf\limits_{u \in \ellipse_{\rhoana[\vt_r]}} \RE \vt_r(u) \le \inf\limits_{u \in \ellipse_\rho} \RE \vt_r(u) < r$ for $1 < \rho \le \rhoana[\vt_r]$, which leads to \eqref{eq:infRe_ineq}.
\end{proof}
Then we have $x$ and $\rho$ dependent error bound as follows:
\begin{lemma}
\label{lem:Evtx}
Let the conditions of Lemma \ref{lem:vt_gauss} be satisfied. Let $\rhoana[\vt_{a/b}] > 1$. Then one has
\begin{align}
\label{eq:Evtx_stenger}
\left\lvert f(x) - \sum_{ \nu = 1 }^M c_\nu \enum^{ - b \vt_{a/b}(u_\nu) x } \right\rvert < \frac{16}{\pi} \rho^{- 2M} \enum^{ - b x \inf\limits_{u \in \ellipse_\rho} \RE \vt_{a / b}(u) } f(0), \rho \in I_{\vt_{a/b}}, 0 < x < \infty.
\end{align}
\end{lemma}
\begin{proof}
Since the composition of analytic functions is analytic, $\enum^{ -bx \vt_{a/b}(u) }$ is analytic on $\ellipse_{\rhoana[\vt_{a/b}]}$ for a given $0 < x < \infty$. By \eqref{eq:vt_increase_eq}, $\enum^{ -bx \vt_{a/b}(u) }$ is real on $[-1, 1]$ for a given $0 < x < \infty$. By using \eqref{eq:vt_increase_eq} and \eqref{eq:infRe_ineq}, we have
\begin{equation}
0 < \enum^{ -ax } = \max_{u \in [-1, 1]} \enum^{ -bx \vt_{a/b}(u) } \le \sup_{u \in \ellipse_\rho} \lvert \RE \enum^{ -bx \vt_{a/b}(u) } \rvert \le \enum^{ - bx \inf\limits_{u \in \ellipse_\rho} \RE \vt_{a/b}(u) } < \infty, \rho \in I_{\vt_{a/b}}, 0 < x < \infty.
\end{equation}
Then by applying Theorem \ref{the:achieser_stenger} to \eqref{eq:f_m1_1} and using \eqref{eq:f0_c}, we have \eqref{eq:Evtx_stenger}.
\end{proof}

To obtain the best possible bound of the maximum absolute error based on Lemma \ref{lem:Evtx}, we need to evaluate
\begin{equation}
\label{eq:sup_inf_E}
\sup_{0 < x < \infty} \inf_{ \rho \in I_{\vt_{a/b}} } \rho^{- 2M} \enum^{ - b x \inf\limits_{u \in \ellipse_\rho} \RE \vt_{a / b}(u) } = \sup_{0 < \lambda < \infty} \inf_{ \rho \in I_{\vt_{a/b}} } \rho^{- 2M} \enum^{ - \lambda \inf\limits_{u \in \ellipse_\rho} \RE \vt_{a / b}(u) }.
\end{equation}
If $\inf\limits_{u \in \ellipse_{\rhoana[\vt_{a/b}]}} \RE \vt_{a/b}(u) \ge 0$, we can evaluate \eqref{eq:sup_inf_E} as follows:
\begin{lemma}
\label{lem:sup_inf_E_rhoana}
Let $0 < r < 1$, $\vt_r \in \vtset_r$, $\rhoana[\vt_r] > 1$, and $\inf\limits_{u \in \ellipse_{\rhoana[\vt_r]}} \RE \vt_r(u) \ge 0$. Then one has
\begin{equation}
\label{eq:sup_inf_E_rhoana}
\sup_{ 0 < \lambda < \infty } \inf_{ \rho \in I_{\vt_r} } \rho^{- \alpha} \enum^{ - \lambda \inf\limits_{u \in \ellipse_\rho} \RE \vt_r(u) } = \rhoana[\vt_r]^{- \alpha}, 0 \le \alpha < \infty.
\end{equation}
\end{lemma}
\begin{proof}
Let $0 \le \alpha < \infty$. By using Lemma \ref{lem:Ivt} and $\ellipse_\rho \subseteq \ellipse_{\rhoana[\vt_r]}$, we have $0 \le \inf\limits_{u \in \ellipse_{\rhoana[\vt_r]}} \RE \vt_r(u) \le \inf\limits_{u \in \ellipse_\rho} \RE \vt_r(u) < r$ for $\rho \in I_{\vt_r} = (1, \rhoana[\vt_r]] \subseteq (1, \infty)$, which indicates $\inf\limits_{ \rho \in I_{\vt_r} } \rho^{- \alpha} \enum^{ - \lambda \inf\limits_{u \in \ellipse_\rho} \RE \vt_r(u) }$ is nonincreasing for $0 < \lambda < \infty$. Then we have
\begin{equation}
\sup_{ 0 < \lambda < \infty } \inf_{ \rho \in I_{\vt_r} } \rho^{- \alpha} \enum^{ - \lambda \inf\limits_{u \in \ellipse_\rho} \RE \vt_r(u) } = \lim_{ \lambda \downarrow 0 } \inf_{ \rho \in I_{\vt_r} } \rho^{- \alpha} \enum^{ - \lambda \inf\limits_{u \in \ellipse_\rho} \RE \vt_r(u) } = \inf_{ \rho \in I_{\vt_r} } \rho^{- \alpha} = \inf_{ \rho \in (1, \rhoana[\vt_r]] } \rho^{- \alpha} = \rhoana[\vt_r]^{- \alpha}, 0 \le \alpha < \infty.
\end{equation}
\end{proof}

Next, we consider the case for $\inf\limits_{u \in \ellipse_{\rhoana[\vt_{a/b}]}} \RE \vt_{a/b}(u) < 0$. For this purpose, we introduce a function $\minfre{\vt}_r(X)$, the strong duality of optimization problems, and a real number $\rho_0[\vt_r]$ as follows:
\begin{lemma}
\label{lem:minfre}
Let $0 < r < 1$, $\vt_r \in \vtset_r$, and $\rhoana[\vt_r] > 1$. Let $\minfre{\vt}_r(X)$ be a function
\begin{equation}
\label{eq:minfre}
\minfre{\vt}_r(X) \coloneqq - \inf_{u \in \ellipse_{\enum^X}} \RE \vt_r(u) = - \min_{u \in \bd{\ellipse_{\enum^X}}} \RE \vt_r(u), 0 < X < \log(\rhoana[\vt_r]).
\end{equation}
Then $\minfre{\vt}_r(X)$ is continuous, strictly increasing, and convex for $0 < X < \log(\rhoana[\vt_r])$, and satisfies $\lim\limits_{X \downarrow 0} \minfre{\vt}_r(X) = - r$ and $\lim\limits_{X \uparrow \log(\rhoana[\vt_r])} \minfre{\vt}_r(X) = - \inf\limits_{u \in \ellipse_{\rhoana[\vt_r]}} \RE \vt_r(u)$.
\end{lemma}
\begin{proof}
The statements except for the convexity follow from Lemma \ref{lem:infRe}. Let $G(\rho) \coloneqq \max\limits_{ \lvert s \rvert = \rho } \left\lvert \enum^{ - \vt_r( (s + s^{-1}) / 2 ) } \right\rvert = \enum^{ \minfre{\vt}_r(\log(\rho)) }$, $1 < \rho < \rhoana[\vt_r]$. By Theorem \ref{the:laurent}, $\enum^{ - \vt_r((s + s^{-1}) / 2) }$ is analytic on $\aset_{\rhoana[\vt_r]}$. Then by Hadamard's three-circles theorem, $\log( G(\rho) )$ is convex with respect to $\log(\rho)$ for $1 < \rho < \rhoana[\vt_r]$, which indicates $\minfre{\vt}_r(X)$ is convex for $0 < X < \log(\rhoana[\vt_r])$.
\end{proof}
\begin{theorem}[{\cite[Sec.~5.2.3 Strong duality and Slater's constraint qualification]{Boyd2004}}]
\label{the:strong_duality}
Let $f_0(x)$ and $f_1(x)$ be convex functions for $a < x < b$. If there exists $c \in (a, b)$ satisfying $f_1(c) < 0$ (called Slater's condition), one has
\begin{equation}
\sup_{ 0 \le \lambda < \infty } \inf_{ a < x < b } \left( f_0(x) + \lambda f_1(x) \right) = \inf_{ a < x < b, f_1(x) \le 0 } f_0(x).
\end{equation}
\end{theorem}
\begin{lemma}
\label{lem:rho0}
Let $0 < r < 1$, $\vt_r \in \vtset_r$, $\rhoana[\vt_r] > 1$, and $\inf\limits_{u \in \ellipse_{\rhoana[\vt_r]}} \RE \vt_r(u) < 0$. Then there exists a unique real number $\rho_0[\vt_r]$ satisfying
\begin{align}
&\text{$1 < \rho_0[\vt_r] < \rhoana[\vt_r]$ and $\inf_{u \in \ellipse_{\rho_0[\vt_r]}} \RE \vt_r(u) = 0$ }, \\
&\begin{cases}
\inf\limits_{u \in \ellipse_\rho} \RE \vt_r(u) > 0, & 1 < \rho < \rho_0[\vt_r], \\
\inf\limits_{u \in \ellipse_\rho} \RE \vt_r(u) < 0, & \rho_0[\vt_r] < \rho \le \rhoana[\vt_r].
\end{cases}
\end{align}
\end{lemma}
\begin{proof}
By Lemma \ref{lem:infRe} and $\inf\limits_{u \in \ellipse_{\rhoana[\vt_r]}} \RE \vt_r(u) < 0$, $\inf\limits_{u \in \ellipse_\rho} \RE \vt_r(u)$ is continuous and strictly decreasing for $1 < \rho < \rhoana[\vt_r]$ and satisfies $\lim\limits_{\rho \downarrow 1} \inf\limits_{u \in \ellipse_\rho} \RE \vt_r(u) > 0$ and $\lim\limits_{\rho \uparrow \rhoana[\vt_r]} \inf\limits_{u \in \ellipse_\rho} \RE \vt_r(u) < 0$. Thus $\inf\limits_{u \in \ellipse_\rho} \RE \vt_r(u)$ has a unique zero in $1 < \rho < \rhoana[\vt_r]$. Then the statements follow by denoting the zero as $\rho_0[\vt_r]$.
\end{proof}
Then if $\inf\limits_{u \in \ellipse_{\rhoana[\vt_r]}} \RE \vt_r(u) < 0$, we can evaluate \eqref{eq:sup_inf_E} as follows:
\begin{lemma}
\label{lem:sup_inf_E_rho0}
Let $0 < r < 1$, $\vt_r \in \vtset_r$, $\rhoana[\vt_r] > 1$, and $\inf\limits_{u \in \ellipse_{\rhoana[\vt_r]}} \RE \vt_r(u) < 0$. Then one has
\begin{equation}
\label{eq:sup_inf_E_rho0}
\sup_{ 0 < \lambda < \infty } \inf_{ \rho \in I_{\vt_r} } \rho^{-\alpha} \enum^{ - \lambda \inf\limits_{u \in \ellipse_\rho} \RE \vt_r(u) } = \rho_0[\vt_r]^{-\alpha}, 0 \le \alpha < \infty.
\end{equation}
\end{lemma}
\begin{proof}
By using Lemma \ref{lem:Ivt}, Lemma \ref{lem:minfre}, Theorem \ref{the:strong_duality}, and Lemma \ref{lem:rho0}, we have
\begin{align}
&\sup_{0 < \lambda < \infty} \inf_{\rho \in I_{\vt_r}} \left( -\alpha \log(\rho) - \lambda \inf_{u \in \ellipse_\rho} \RE \vt_r(u) \right) = \sup_{0 \le \lambda < \infty} \inf_{ 0 < X < \log(\rhoana[\vt_r]) } \left( -\alpha X + \lambda \minfre{\vt}_r(X) \right) \\ \nonumber
&= \inf_{ 0 < X < \log(\rhoana[\vt_r]), \minfre{\vt}_r(X) \le 0 } (-\alpha X) = \inf_{ 0 < X \le \log(\rho_0[\vt_r]) } (-\alpha X) = -\alpha \log(\rho_0[\vt_r]), 0 \le \alpha < \infty,
\end{align}
which leads to \eqref{eq:sup_inf_E_rho0}.
\end{proof}
To unify the description, we introduce a real number $\hat{\rho}[\vt_r]$ as follows:
\begin{lemma}
\label{lem:hatrho}
Let $0 < r < 1$, $\vt_r \in \vtset_r$, and $\rhoana[\vt_r] > 1$. Let us introduce
\begin{equation}
\label{eq:hatrho}
\hat{\rho}[\vt_r] \coloneqq
\begin{cases}
\rhoana[\vt_r] & \text{if $\inf\limits_{u \in \ellipse_{\rhoana[\vt_r]}} \RE \vt_r(u) \ge 0$}, \\
\rho_0[\vt_r] & \text{if $\inf\limits_{u \in \ellipse_{\rhoana[\vt_r]}} \RE \vt_r(u) < 0$}.
\end{cases}
\end{equation}
Then one has
\begin{align}
\label{eq:hatrho_in}
&\hat{\rho}[\vt_r] \in I_{\vt_r} \subseteq (1, \infty), \\
\label{eq:infRe_hatrho_ineq}
&0 \le \inf_{u \in \ellipse_{\hat{\rho}[\vt_r]}} \RE \vt_r(u) < r.
\end{align}
\end{lemma}
\begin{proof}
The statements follow from Lemma \ref{lem:Ivt} and Lemma \ref{lem:rho0}.
\end{proof}
By using $\hat{\rho}[\vt_r]$, we can evaluate \eqref{eq:sup_inf_E} as follows:
\begin{lemma}
\label{lem:sup_inf_E_hatrho}
Let $0 < r < 1$, $\vt_r \in \vtset_r$, and $\rhoana[\vt_r] > 1$. Then one has
\begin{equation}
\sup_{ 0 < \lambda < \infty } \inf_{ \rho \in I_{\vt_r} } \rho^{- \alpha} \enum^{ - \lambda \inf\limits_{u \in \ellipse_\rho} \RE \vt_r(u) } = \hat{\rho}[\vt_r]^{- \alpha}, 0 \le \alpha < \infty. 
\end{equation}
\end{lemma}
\begin{proof}
The statement follows from Lemma \ref{lem:sup_inf_E_rhoana}, Lemma \ref{lem:sup_inf_E_rho0}, and Lemma \ref{lem:hatrho}.
\end{proof}
With these preparations, we have the maximum absolute error bound as follows:
\begin{theorem}
\label{the:Evt}
Let the conditions of Lemma \ref{lem:vt_gauss} be satisfied. Let $\rhoana[\vt_{a/b}] > 1$. Then one has
\begin{equation}
\label{eq:Evt_stenger}
0 < \max_{0 \le x < \infty} \left\lvert f(x) - \sum_{\nu = 1}^M c_\nu \enum^{ - b \vt_{a/b}(u_\nu) x } \right\rvert < \frac{16}{\pi} \hat{\rho}[\vt_{a/b}]^{- 2M} f(0),
\end{equation}
whose upper bound gives the best possible maximum absolute error bound obtained from \eqref{eq:Evtx_stenger}.
\end{theorem}
\begin{proof}
Let $E_M(x) \coloneqq f(x) - \sum_{\nu = 1}^M c_\nu \enum^{ - b \vt_{a/b}(u_\nu) x }$. Since $W(t)$ has at least $M+1$ points of increase, we have $f \notin \eset_M$. By Lemma \ref{lem:kammler_zero}, (i) $E_M(x)$ has at most $2M$ zeros in $0 \le x < \infty$. By \eqref{eq:Evtx_exp}, we have (ii) $E_M(0) = \lim\limits_{x \to \infty} E_M(x) = 0$. Since $E_M(x)$ is continuous for $0 \le x < \infty$, by (i) and (ii), the maximum value of $\lvert E_M(x) \rvert$ is attained in $0 < x < \infty$ and satisfies $\max\limits_{0 \le x < \infty} \lvert E_M(x) \rvert = \max\limits_{0 < x < \infty} \lvert E_M(x) \rvert > 0$. Since $\hat{\rho}[\vt_{a/b}] \in I_{\vt_{a/b}}$ by \eqref{eq:hatrho_in}, we can substitute $\rho = \hat{\rho}[\vt_{a/b}]$ for \eqref{eq:Evtx_stenger}. Then by using \eqref{eq:infRe_hatrho_ineq}, we have the upper bound in \eqref{eq:Evt_stenger}. By Lemma \ref{lem:sup_inf_E_hatrho}, the upper bound gives the best possible maximum absolute error bound obtained from \eqref{eq:Evtx_stenger}.
\end{proof}
Theorem \ref{the:Evt} indicates the maximum absolute error decreases at least geometrically with respect to $M$ and the decreasing rate is determined by $\hat{\rho}[\vt_{a/b}]^2$. In \ref{app:hatrho}, we derived expressions of $\hat{\rho}[\vt_{a/b}]$ for some variable transformations. Since the larger value of $\hat{\rho}[\vt_{a/b}]$ gives the better bound, we find a variable transformation $\vt_{a/b}(u)$ maximizing $\hat{\rho}[\vt_{a/b}]$ in Section \ref{sec:dn}.

\subsection{Basis functions associated with the variable transformation}
\label{sec:basis}
By using the analyticity of the integrand $\enum^{- b x \vt_{a / b}(u)}$ in \eqref{eq:f_m1_1} on $\ellipse_{\rhoana[\vt_r]}$, we applied the Achieser--Stenger bound of the Gaussian quadrature in Section \ref{sec:error}. Hunter obtained the error expansion of a Gaussian quadrature by expanding an analytic integrand on an open Bernstein ellipse based on Chebyshev polynomials \cite[Theorem 3 and Lemma 6]{Hunter1995}. If we expand the integrand $\enum^{- b x \vt_{a / b}(u)}$ by Chebyshev polynomials with respect to $u$, the expansion coefficients depend on $bx$. For the expansion, we introduce functions corresponding to the expansion coefficients as follows:
\begin{lemma}
\label{lem:basis}
Let $0 < r < 1$, $\vt_r \in \vtset_r$, and $\rhoana[\vt_r] > 1$. Let us introduce
\begin{align}
\label{eq:basis}
&\basis{\vt}_{r, n}(z) \coloneqq \frac{1}{\pi} \int_0^\pi \enum^{- z \vt_r(\cos\theta)} \cos(n \theta) \del \theta, z \in \mathbb{C}, n \in \intge{0}, \\
&I_{\vt_r, \theta} \coloneqq
\begin{cases}
(1, \rhoana[\vt_r]) & \text{if $\inf\limits_{u \in \ellipse_{\rhoana[\vt_r]}} \RE (\enum^{\inum \theta} \vt_r(u)) = - \infty$}, \\
(1, \rhoana[\vt_r]] & \text{if $\inf\limits_{u \in \ellipse_{\rhoana[\vt_r]}} \RE (\enum^{\inum \theta}  \vt_r(u)) > - \infty$},
\end{cases}
\theta \in \mathbb{R}.
\end{align}
Let $z \in \mathbb{C}$ and $w \in \mathbb{C}$. If not  specified, we assume $n \in \intge{0}$. Then one has
\begin{align}
\label{eq:exp_vt_laurent}
&\enum^{- z \vt_r\left( \frac{ s + s^{-1} }{ 2 } \right)} = \basis{\vt}_{r, 0}(z) + \sum_{n = 1}^\infty \basis{\vt}_{r, n}(z) (s^n + s^{-n}), s \in \aset_{\rhoana[\vt_r]}, \\
\label{eq:exp_vt_fourier}
&\enum^{- z \vt_r(\cos\theta)} = \basis{\vt}_{r, 0}(z) + 2 \sum_{n = 1}^\infty \basis{\vt}_{r, n}(z) \cos(n \theta), \theta \in \imset_{\log(\rhoana[\vt_r])}, \\
\label{eq:exp_vt_chebyshev}
&\enum^{- z \vt_r(u)} = \basis{\vt}_{r, 0}(z) + 2 \sum_{n = 1}^\infty \basis{\vt}_{r, n}(z) T_n(u), u \in \ellipse_{\rhoana[\vt_r]}, \\
\label{eq:basis_0}
&\basis{\vt}_{r, 0}(0) = 1, \basis{\vt}_{r, n}(0) = 0, n \in \intge{1}, \\
\label{eq:basisz_ineq_rho}
&\lvert \basis{\vt}_{r, n}(R \enum^{\inum \theta}) \rvert \le \rho^{-n} \enum^{ - R \inf\limits_{ u \in \ellipse_\rho } \RE(\enum^{\inum \theta} \vt_r(u)) } < \infty, \rho \in I_{\vt_r, \theta}, 0 < R < \infty, \theta \in \mathbb{R}, \\
\label{eq:basis_ineq_rho}
&\lvert \basis{\vt}_{r, n}(x) \rvert \le \rho^{-n} \enum^{ - x \inf\limits_{ u \in \ellipse_\rho } \RE \vt_r(u) } < \infty, \rho \in I_{\vt_r}, 0 < x < \infty, \\
\label{eq:basis_ineq_hatrho}
&\lvert \basis{\vt}_{r, n}(x) \rvert \le \hat{\rho}[\vt_r]^{-n}, 0 < x < \infty, \\
\label{eq:basis_add}
&\basis{\vt}_{r, n}(z + w) = \basis{\vt}_{r, 0}(z) \basis{\vt}_{r, n}(w) + \sum_{m = 1}^\infty \basis{\vt}_{r, m}(z) (\basis{\vt}_{r, n + m}(w) + \basis{\vt}_{r, \lvert n - m \rvert}(w)).
\end{align}
\end{lemma}
\begin{proof}
Since $F(u) \coloneqq \enum^{- z \vt_r(u) }$ is analytic on $\ellipse_{\rhoana[\vt_r]}$ for a given $z \in \mathbb{C}$, \eqref{eq:exp_vt_laurent}, \eqref{eq:exp_vt_fourier}, \eqref{eq:exp_vt_chebyshev}, and  \eqref{eq:basisz_ineq_rho} follow from Theorem \ref{the:laurent}. \eqref{eq:basis_0} follows from $\int_0^\pi \del \theta = \pi$ and $\int_0^\pi \cos(n \theta) \del \theta = 0$ for $n \in \intge{1}$. By substituting $\theta = 0$ for \eqref{eq:basisz_ineq_rho} and using $I_{\vt_r, 0} = I_{\vt_r}$, we have \eqref{eq:basis_ineq_rho}. By substituting $\rho = \hat{\rho}[\vt_r]$ for \eqref{eq:basis_ineq_rho} and using \eqref{eq:infRe_hatrho_ineq}, we have \eqref{eq:basis_ineq_hatrho}.

By substituting \eqref{eq:exp_vt_fourier} for $\basis{\vt}_{r, n}(z + w) = \frac{1}{\pi} \int_0^\pi \enum^{- z \vt_r(\cos\theta)} \enum^{- w \vt_r(\cos\theta)} \cos(n \theta) \del \theta$ and using $2 \cos(m\theta) \cos(n\theta) = \cos((m+n)\theta) + \cos(\lvert m-n \rvert \theta)$, we have \eqref{eq:basis_add}.
\end{proof}

Then a finite Laplace--Stieltjes transform can be expanded as follows:
\begin{lemma}
\label{lem:f_basis}
Let $0 < a < b < \infty$ and $z \in \mathbb{C}$. Let $W(t)$ be of bounded variation in $a \le t \le b$. Let $f(z) \coloneqq \int_a^b \enum^{ - z t } \del W(t)$. Let $\vt_{a/b} \in \vtset_{a/b}$ and $\rhoana[\vt_{a/b}] > 1$. Let $\TV[F(X)]_{X_1}^{X_2}$ be the total variation of a function $F(X)$ on $[X_1, X_2]$, where $-\infty < X_1 < X_2 < \infty$. Let $\sigma_n \coloneqq \int_{-1}^1 T_n(u) \del W(b \vt_{a/b}(u))$, $n \in \intge{1}$. Then one has
\begin{align}
\label{eq:f_basis}
&f(z) = f(0) \basis{\vt}_{a/b, 0}(bz) + 2 \sum_{n = 1}^\infty \sigma_n \basis{\vt}_{a/b, n}(bz), \\
\label{eq:int_chebyshev_ineq}
&\lvert \sigma_n \rvert \le \TV[W(t)]_a^b, n \in \intge{1}.
\end{align}
\end{lemma}
\begin{proof}
\eqref{eq:f_basis} follows from \eqref{eq:f_m1_1}, \eqref{eq:exp_vt_chebyshev}, and Theorem \ref{the:stieltjes_series}. \eqref{eq:int_chebyshev_ineq} follows from Theorem \ref{the:stieltjes_abs_ineq}, $\lvert T_n(u) \rvert \le 1$ on $[-1, 1]$, and $\TV[W(b \vt_{a/b}(u))]_{-1}^1 = \TV[W(t)]_a^b$.
\end{proof}
\begin{corollary}
\label{cor:f_basis_error}
Let the conditions of Lemma \ref{lem:f_basis} be satisfied. Let us introduce
\begin{equation}
E_N(z) \coloneqq
\begin{cases}
f(z) - f(0) \basis{\vt}_{a/b, 0}(bz), & N = 0, \\
f(z) - f(0) \basis{\vt}_{a/b, 0}(bz) - 2 \sum\limits_{n = 1}^N \sigma_n \basis{\vt}_{a/b, n}(bz), & N \in \intge{1},
\end{cases}
z \in \mathbb{C}.
\end{equation}
For $N \in \intge{0}$, $0 < R < \infty$, and $\theta \in \mathbb{R}$, one has
\begin{align}
\label{eq:f0_basis_error}
&E_N(0) = 0,\\
\label{eq:fz_basis_error_ineq}
&\lvert E_N(R \enum^{\inum \theta}) \rvert \le \frac{ 2 \rho^{- N - 1} }{ 1 - \rho^{-1} } \enum^{ - R \inf\limits_{ u \in \ellipse_\rho } \RE(\enum^{\inum \theta} \vt_{a/b}(u)) } \TV[W(t)]_a^b, \rho \in I_{\vt_{a/b}, \theta}, \\
\label{eq:fx_basis_error_ineq}
&\lvert E_N(x) \rvert \le \frac{ 2 \hat{\rho}[\vt_{a/b}]^{- N - 1} }{ 1 - \hat{\rho}[\vt_{a/b}]^{-1} } \TV[W(t)]_a^b, 0 < x < \infty.
\end{align}
\end{corollary}
\begin{proof}
\eqref{eq:f0_basis_error} follows from \eqref{eq:basis_0}. \eqref{eq:fz_basis_error_ineq} follows from \eqref{eq:f_basis}, \eqref{eq:int_chebyshev_ineq}, and \eqref{eq:basisz_ineq_rho}. \eqref{eq:fx_basis_error_ineq} follows from \eqref{eq:f_basis}, \eqref{eq:int_chebyshev_ineq}, and \eqref{eq:basis_ineq_hatrho}.
\end{proof}
By Lemma \ref{lem:f_basis}, a finite Laplace--Stieltjes transform can be expanded based on the functions $\basis{\vt}_{r, n}(x)$. Then we call $\basis{\vt}_{r, n}(x)$ the basis functions associated with $\vt_r(u)$. By using the basis functions, we can expand the error of the Gaussian quadrature as follows:
\begin{theorem}
\label{the:f_gauss_basis}
Let the conditions of Lemma \ref{lem:vt_gauss} be satisfied. Let $\rhoana[\vt_{a/b}] > 1$. Let $\varepsilon_{M, n} \coloneqq \int_{-1}^1 T_n(u) \del W(b \vt_{a/b}(u)) - \sum_{\nu = 1}^M c_\nu T_n(u_\nu)$, $n \in \intge{2M}$. Then one has
\begin{align}
\label{eq:f_gauss_basis}
&f(z) - \sum_{\nu = 1}^M c_\nu \enum^{- b \vt_{a/b}(u_\nu) z} = 2 \sum_{n = 2M}^\infty \varepsilon_{M, n} \basis{\vt}_{a/b, n}(bz), z \in \mathbb{C}, \\
\label{eq:f_gauss_basis_eps_ineq}
&\lvert \varepsilon_{M, n} \rvert \le 2 f(0), n \in \intge{2M}.
\end{align}
\end{theorem}
\begin{proof}
\eqref{eq:f_gauss_basis} follows from \eqref{eq:f_basis}, \eqref{eq:exp_vt_chebyshev}, and \eqref{eq:vt_gauss_exact}. \eqref{eq:f_gauss_basis_eps_ineq} follows from $\lvert T_n(u) \rvert \le 1$ on $[-1, 1]$, \eqref{eq:vt_gauss_uc_ineq}, and \eqref{eq:f0_c}.
\end{proof}
By the exactness of the Gaussian quadrature for the polynomial integrand with at most $2M-1$ degree, the basis expansion \eqref{eq:f_gauss_basis} starts from $n = 2M$. Let $0 < x < \infty$. Since the inequality $\lvert \varepsilon_{M, n} \basis{\vt}_{a/b, n}(bx) \rvert \le 2 f(0) \hat{\rho}[\vt_{a/b}]^{-n}$ holds for $n \in \intge{2M}$ by \eqref{eq:basis_ineq_hatrho} and \eqref{eq:f_gauss_basis_eps_ineq}, the leading terms of the right-hand side of \eqref{eq:f_gauss_basis} are expected to have the larger contributions to the error of the Gaussian quadrature $f(x) - \sum_{\nu = 1}^M c_\nu \enum^{- b \vt_{a/b}(u_\nu) x}$ if $\hat{\rho}[\vt_{a/b}] \gg 1$. We can investigate some properties of the basis functions only from the conditions \eqref{eq:vt_increase_eq} as follows:
\begin{lemma}
\label{lem:basis_chebyshev}
Let $0 < r < 1$ and $\vt_r \in \vtset_r$. For $r \le \tau \le 1$, let us introduce
\begin{align}
\label{eq:basis_W0}
&W_{\vt_r, 0}(\tau) \coloneqq \frac{ \arccos(- \vt_r^{-1}(\tau)) }{ \pi } = \frac{ \pi - \arccos(\vt_r^{-1}(\tau)) }{\pi}, \\
\label{eq:basis_Wge1}
&W_{\vt_r, n}(\tau) \coloneqq \frac{ (-1)^n }{ n \pi } \sin(n \arccos(-\vt_r^{-1}(\tau))) = - \frac{ 1 }{n \pi} \sin( n \arccos( \vt_r^{-1}(\tau) ) ) \\ \nonumber
&= \frac{ (-1)^n }{ n \pi } \sin(n \pi W_{\vt_r, 0}(\tau)) = - \frac{ 1 }{n \pi} \sin( n \pi (1 - W_{\vt_r, 0}(\tau)) ), n \in \intge{1},
\end{align}
where $\arccos(u)$ on $[-1, 1]$ is the principal value of the inverse cosine function. Let $\det( a_{i, j} )_{i, j = n_1}^{n_2}$ be the determinant of a matrix whose element is $a_{i, j}$ for $i = n_1, \ldots, n_2$ and $j = n_1, \ldots, n_2$. For $z \in \mathbb{C}$, $m \in \intge 0$, and $0 \le x < \infty$, one has
\begin{align}
\label{eq:basis_fl}
&\basis{\vt}_{r, n}(z) = \int_r^1 \enum^{-z \tau} \del W_{\vt_r, n}(\tau), n \in \intge{0}, \\
\label{eq:basis_W0_increase}
&\text{$W_{\vt_r, 0}(\tau)$ is a continuous strictly increasing function for $r \le \tau \le 1$} \\ \nonumber
&\text{and satisfies $W_{\vt_r, 0}(r) = 0$, $W_{\vt_r, 0}(1) = 1$, and $\TV[W_{\vt_r, 0}(\tau)]_r^1 = 1$}, \\
\label{eq:basis_Wge1_bv}
&\text{$W_{\vt_r, n}(\tau)$ is a real continuous function for $r \le \tau \le 1$ and satisfies} \\ \nonumber
&\text{$W_{\vt_r, n}(r) = W_{\vt_r, n}(1) = 0$ and $\TV[W_{\vt_r, n}(\tau)]_r^1 = 2 / \pi$ for $n \in \intge{1}$}, \\
\label{eq:basis_0_dWn}
&\basis{\vt}_{r, 0}(0) = \int_r^1 \del W_{\vt_r, 0}(\tau) = 1, \basis{\vt}_{r, n}(0) = \int_r^1 \del W_{\vt_r, n}(\tau) = 0, n \in \intge{1}, \\
\label{eq:basis_entire}
&\text{$\basis{\vt}_{r, n}(z)$ is an entire function and real on the real axis, $n \in \intge{0}$,} \\
\label{eq:basis_deriv_W}
&\basis{\vt}_{r, n}^{(m)}(z) = (-1)^m \int_r^1 \tau^m \enum^{-z \tau} \del W_{\vt_r, n}(\tau), n \in \intge{0}, \\
\label{eq:basis_deriv_cos}
&\basis{\vt}_{r, n}^{(m)}(z) = \frac{(-1)^m}{\pi} \int_0^\pi \vt_r(\cos\theta)^m \enum^{- z \vt_r(\cos\theta)} \cos(n \theta) \del \theta, n \in \intge{0}, \\
\label{eq:basis_deriv_ineq}
&r^m \enum^{-x} \le (-1)^m \basis{\vt}_{r, 0}^{(m)}(x) \le \enum^{- r x}, \lvert \basis{\vt}_{r, n}^{(m)}(x) \rvert \le \frac{ 2 }{ \pi } \enum^{- r x}, n \in \intge{1}, \\
\label{eq:basis_sin}
&\basis{\vt}_{r, n}(z) = (-1)^n \frac{z}{n \pi} \int_r^1 \enum^{ -z \tau } \sin( n \pi W_{\vt_r, 0}(\tau) ) \del \tau = - \frac{z}{n \pi} \int_r^1 \enum^{ -z \tau } \sin( n \pi (1 - W_{\vt_r, 0}(\tau)) ) \del \tau, n \in \intge{1}, \\
\label{eq:basis_ineq1}
&\lvert \basis{\vt}_{r, n}(x) \rvert \le \frac{ \enum^{ -r x } - \enum^{ - x } }{ n \pi } \le \frac{ 1 - r }{ n \pi } x \enum^{ -r x }, n \in \intge{1}, \\
\label{eq:basis0_det_ineq}
&\det( (-1)^j \basis{\vt}_{r, j}(x_i) )_{ i, j = 0 }^n > 0, -\infty < x_0 < \cdots < x_n < \infty, n \in \intge{0}, \\
\label{eq:basis_det_ineq}
&\det( (-1)^j \basis{\vt}_{r, j}(x_i) )_{ i, j = 1 }^n > 0, 0 < x_1 < \cdots < x_n < \infty, n \in \intge{1}.
\end{align}
\end{lemma}
\begin{proof}
By the change of variable $\theta$ to $\pi - \theta$, \eqref{eq:basis} is expressed as
\begin{equation}
\label{eq:basis_zm}
\basis{\vt}_{r, n}(z) = \frac{1}{\pi} \int_0^\pi \enum^{- z \vt_r(- \cos\theta)} (-1)^n \cos(n \theta) \del \theta, z \in \mathbb{C}, n \in \intge{0}.
\end{equation}
By introducing
\begin{equation}
\label{eq:basis_beta}
\beta_n(\theta) \coloneqq \int_0^\theta \frac{(-1)^n}{\pi} \cos(n \eta) \del \eta =
\begin{cases}
\frac{\theta}{\pi}, & n = 0, \\
\frac{ (-1)^n \sin(n \theta) }{ n \pi }, & n \in \intge{1},
\end{cases}
0 \le \theta \le \pi,
\end{equation}
and using Theorem \ref{the:stieltjes_riemann}, \eqref{eq:basis_zm} is expressed as
\begin{equation}
\label{eq:basis_int_beta_theta}
\basis{\vt}_{r, n}(z) = \int_0^\pi \enum^{- z \vt_r(-\cos\theta)} \del \beta_n(\theta), z \in \mathbb{C}, n \in \intge{0}.
\end{equation}
If a function defined in an interval has subintervals such that the function is bounded and monotonic for each interval, the total variation in the overall interval is the sum of them for each interval \cite[Theorem 6.1.12 and Corollary 6.1.13]{Kannan1996}. By using the property, we have
\begin{subequations}
\label{eq:TV_beta}
\begin{align}
&\TV[\beta_0(\theta)]_0^\pi = [\theta/\pi]_0^\pi = 1, \\
&\TV[\beta_n(\theta)]_0^\pi = \sum_{m = 1}^{2n} \left\lvert \left[ \frac{ (-1)^n \sin(n \theta) }{ n \pi } \right]_{ \frac{ (m-1) \pi }{ 2n } }^{ \frac{ m \pi }{ 2n } } \right\rvert = \frac{ 2 }{ \pi }, n \in \intge{1}.
\end{align}
\end{subequations}
By \eqref{eq:vt_increase_eq}, we see that
\begin{align}
\label{eq:basis_theta_t_increase}
&\text{$\arccos(-\vt_r^{-1}(\tau))$ is continuous and strictly increasing for $r \le \tau \le 1$} \\ \nonumber
&\text{and satisfies $\arccos(-\vt_r^{-1}(r)) = 0$ and $\arccos(-\vt_r^{-1}(1)) = \pi$.}
\end{align}
Then by the change of variable $\theta = \arccos(-\vt_r^{-1}(\tau))$ with Theorem \ref{the:stieltjes_change}, \eqref{eq:basis_int_beta_theta} is expressed as
\begin{equation}
\label{eq:basis_int_beta_t}
\basis{\vt}_{r, n}(z) = \int_r^1 \enum^{- z \tau} \del \beta_n(\arccos(-\vt_r^{-1}(\tau))), z \in \mathbb{C}, n \in \intge{0}.
\end{equation}

Let $n \in \intge{0}$. By introducing $W_{\vt_r, n}(\tau) \coloneqq \beta_n(\arccos(-\vt_r^{-1}(\tau)))$  and using \eqref{eq:basis_beta} and \eqref{eq:basis_int_beta_t}, we have \eqref{eq:basis_W0}, \eqref{eq:basis_Wge1}, and \eqref{eq:basis_fl}. \eqref{eq:basis_W0_increase} and \eqref{eq:basis_Wge1_bv} follow from \eqref{eq:basis_W0}, \eqref{eq:basis_Wge1}, \eqref{eq:basis_theta_t_increase}, $\TV[W_{\vt_r, n}(\tau)]_r^1 = \TV[\beta_n(\theta)]_0^\pi$, and \eqref{eq:TV_beta}. \eqref{eq:basis_0_dWn} follows from \eqref{eq:basis} and \eqref{eq:basis_fl}. \eqref{eq:basis_entire} and \eqref{eq:basis_deriv_W} follow from \eqref{eq:basis_fl}, \eqref{eq:basis_W0_increase}, \eqref{eq:basis_Wge1_bv}, and Lemma \ref{lem:fl_entire}. By introducing $F(\theta, z) \coloneqq \enum^{- z \vt_r(\cos\theta)} \cos(n \theta)$, \eqref{eq:basis} is expressed as $\basis{\vt}_{r, n}(z) = \frac{1}{\pi} \int_0^\pi F(\theta, z) \del \theta$. For any $m \in \intge{0}$, $\frac{ \partial^m F(\theta, z) }{ \partial z^m }  = (-1)^m \vt_r(\cos\theta)^m \enum^{- z \vt_r(\cos\theta)} \cos(n \theta)$ is continuous on $[0, \pi] \times \mathbb{C}$ by \eqref{eq:vt_increase_eq}. Then by applying Leibniz integral rule $m$ times to $\basis{\vt}_{r, n}(z) = \frac{1}{\pi} \int_0^\pi F(\theta, z) \del \theta$, we have \eqref{eq:basis_deriv_cos}.

Inequalities for $(-1)^m \basis{\vt}_{r, 0}^{(m)}(x)$ in \eqref{eq:basis_deriv_ineq} follow from \eqref{eq:basis_deriv_cos} and \eqref{eq:vt_increase_eq}. The inequality for $n \in \intge{1}$ in \eqref{eq:basis_deriv_ineq} follows from \eqref{eq:basis_Wge1_bv}, \eqref{eq:basis_deriv_W}, and Theorem \ref{the:stieltjes_abs_ineq}. \eqref{eq:basis_sin} follows from \eqref{eq:basis_fl}, \eqref{eq:basis_Wge1_bv}, Theorem \ref{the:stieltjes_parts}, Theorem \ref{the:stieltjes_riemann}, and \eqref{eq:basis_Wge1}. \eqref{eq:basis_ineq1} follows from \eqref{eq:basis_sin} and $\enum^{-r x} - \enum^{-x} = x \int_r^1 \enum^{-x \tau} \del \tau \le (1 - r) x \enum^{-r x}$ for $0 \le x < \infty$.

Let $n \in \intge{0}$. By applying the basic composition formula \cite[Sec.~1.2]{Karlin1968} to $(-1)^n \basis{\vt}_{r, n}(z) = \frac{1}{\pi} \int_0^\pi \enum^{- z \vt_r(\cos\theta)} \cos(n (\pi - \theta)) \del \theta$, we have
\begin{equation}
\label{eq:det_basis0_int}
\det( (-1)^j \basis{\vt}_{r, j}(x_i) )_{ i, j = 0 }^n = \frac{ 1 }{ \pi^{n + 1} } \mint_{0 \le \theta_0 \le \cdots \le \theta_n \le \pi} \det( \enum^{- x_i \vt_r(\cos\theta_j)} )_{ i, j = 0 }^n \det( \cos(i (\pi - \theta_j)) )_{ i, j = 0 }^n \del \theta_0 \cdots \del \theta_n, (x_0, \ldots, x_n) \in \mathbb{R}^{n + 1}.
\end{equation}
By using the equality \cite[Lemma 3]{Goodman1948}
\begin{equation}
\det( \cos( i \theta_j ) )_{ i, j = 0}^n = 2^\frac{ n (n - 1) }{ 2 } \prod_{ 0 \le i < j \le n } ( \cos\theta_j - \cos\theta_i )
\end{equation}
for $(\theta_0, \ldots, \theta_n) \in \mathbb{R}^{n + 1}$, $n \in \intge{1}$, we have
\begin{equation}
\label{eq:det_cos_mtheta_ineq}
\det( \cos(i (\pi - \theta_j)) )_{ i, j = 0 }^n > 0, 0 \le \theta_0 < \cdots < \theta_n \le \pi.
\end{equation}
By using the inequality \cite[Sec.~1.2]{Karlin1968}
\begin{equation}
\label{eq:det_exp_ineq}
\det( \enum^{ x_i t_j } )_{ i, j = 1 }^n > 0, -\infty < x_1 < \cdots < x_n < \infty, -\infty < t_1 < \cdots < t_n < \infty,
\end{equation}
and \eqref{eq:vt_increase_eq}, we have
\begin{equation}
\label{eq:det_exp_xmvt_ineq}
\det( \enum^{x_i (- \vt_r(\cos\theta_j))} )_{ i, j = 0 }^n > 0, -\infty < x_1 < \cdots < x_n < \infty, 0 \le \theta_0 < \cdots < \theta_n \le \pi,
\end{equation}
Then by using \eqref{eq:det_basis0_int}, \eqref{eq:det_cos_mtheta_ineq}, and \eqref{eq:det_exp_xmvt_ineq}, we have \eqref{eq:basis0_det_ineq}.

Let $n \in \intge{1}$. By applying the change of variable $\tau = -t$ to \eqref{eq:basis_sin}, we have $(-1)^n \basis{\vt}_{r, n}(z) = \frac{z}{n \pi} \int_{-1}^{-r} \enum^{ z t } \sin( n \pi W_{\vt_r, 0}(-t) ) \del t$. Then by using the basic composition formula, we have
\begin{equation}
\label{eq:det_basis_int}
\det( (-1)^j \basis{\vt}_{r, j}(x_i) )_{ i, j = 1 }^n = \frac{ \prod_{i = 1}^n x_i }{ n! \pi^n } \mint_{-1 \le t_1 \le \cdots \le t_n \le - r} \det( \enum^{ x_i t_j } )_{ i, j = 1 }^n \det( \sin( i \pi W_{\vt_r, 0}(-t_j) ) )_{ i, j = 1 }^n \del t_1 \cdots \del t_n, (x_1, \ldots, x_n) \in \mathbb{R}^n.
\end{equation}
By using the equality \cite[Lemma 3]{Goodman1948}
\begin{equation}
\det( \sin( i \theta_j ) )_{ i, j = 1}^n = 2^\frac{ n (n - 1) }{ 2 } \prod_{l = 1}^n \sin\theta_l \prod_{ 1 \le i < j \le n } ( \cos\theta_j - \cos\theta_i )
\end{equation}
for $(\theta_1, \ldots, \theta_n) \in \mathbb{R}^n$, $n \in \intge{2}$, and \eqref{eq:basis_W0_increase}, we have
\begin{equation}
\label{eq:det_sin_W0_ineq}
\det( \sin( i \pi W_{\vt_r, 0}(-t_j) ) )_{ i, j = 1 }^n > 0, -1 < t_1 < \cdots < t_n < -r.
\end{equation}
Then by using \eqref{eq:det_basis_int}, \eqref{eq:det_sin_W0_ineq}, and \eqref{eq:det_exp_ineq}, we have \eqref{eq:basis_det_ineq}.
\end{proof}
Examples of $W_{\vt_r, 0}(\tau)$ are found in \eqref{eq:W0_Phi} and \eqref{eq:W0_P1}. $W_{\vt_r, n}(\tau)$ for $n \in \intge{1}$ can be expressed by $W_{\vt_r, 0}(\tau)$ based on the last two equalities in \eqref{eq:basis_Wge1}. $W_{\vt_r, 0}(\tau)$ also appears in a kernel function \eqref{eq:K_vt}. By \eqref{eq:basis_fl} and \eqref{eq:basis_W0_increase}, we see that $\basis{\vt}_{r, 0}(x)$ is a finite completely monotonic function. Let $n \in \intge{1}$. Since $\basis{\vt}_{r, 1}(x), \ldots, \basis{\vt}_{r, n}(x)$ are real continuous functions on $\mathbb{R}$ by \eqref{eq:basis_entire} and satisfy the inequality \eqref{eq:basis_det_ineq}, they form a Chebyshev system on $(0, \infty)$ and $\basis{\vt}_{r, n}(x)$ has at most $n - 1$ zeros on $(0, \infty)$. The basis function $\basis{\vt}_{\pset_1, r, n}(x)$ defined in \eqref{eq:basis_P1} has no zeros on $(0, \infty)$ as shown in \eqref{eq:basis_P1_ineq}. The basis function $\basis{\Phi}_{r, n}(x)$ defined in \eqref{eq:basis_Phi_cos} has the maximum number of zeros, i.e., $n - 1$ zeros on $(0, \infty)$ as shown in Theorem \ref{the:basis_Phi_interlacing}.

We see later in \eqref{eq:basis_Phi_integral_eq} that a bivariate function $\basis{\vt}_{r, 0}(z + y) - \basis{\vt}_{r, 0}(z) \basis{\vt}_{r, 0}(y)$ on $\mathbb{C} \times [0, \infty)$ serves as the kernel function of an integral equation. We can investigate some properties of the function only from the conditions \eqref{eq:vt_increase_eq} as follows:
\begin{definition}[{\cite[Chap.~2, Definition~1.1]{Karlin1968}}]
Let $-\infty \le a < b \le \infty$ and $-\infty \le c < d \le \infty$. Let $K(x, y)$ be a real function on $(a, b) \times (c, d)$. If the inequality $\det(K(x_i, y_j))_{i, j = 1}^n > 0$ holds for any $a < x_1 < \cdots < x_n < b$, $c < y_1 < \cdots < y_n < d$, $n \in \intge{1}$, we say $K(x, y)$ is strictly totally positive on $(a, b) \times (c, d)$.
\end{definition}
\begin{lemma}
Let the conditions of Lemma \ref{lem:basis_chebyshev} be satisfied. Let us introduce
\begin{align}
\label{eq:K_vt}
&K_{\vt_r}(\sigma, \tau) \coloneqq W_{\vt_r, 0}(\min(\sigma, \tau)) (1 - W_{\vt_r, 0}(\max(\sigma, \tau))), r \le \sigma \le 1, r \le \tau \le 1, \\
\label{eq:V_vt}
&V_{\vt_r} \coloneqq \int_r^1 \tau^2 \del W_{\vt_r, 0}(\tau) - \left( \int_r^1 \tau \del W_{\vt_r, 0}(\tau) \right)^2.
\end{align}
For $z \in \mathbb{C}$, $w \in \mathbb{C}$, $0 \le x < \infty$, and $0 \le y < \infty$, one has
\begin{align}
\label{eq:basis_kernel_int}
&\basis{\vt}_{r, 0}(z + w) - \basis{\vt}_{r, 0}(z) \basis{\vt}_{r, 0}(w) = z w \int_r^1 \int_r^1 \enum^{- z \sigma} \enum^{ - w \tau } K_{\vt_r}(\sigma, \tau) \del \sigma \del \tau, \\
\label{eq:K_vt_ineq}
&0 \le K_{\vt_r}(\sigma, \tau) \le 1 / 4, r \le \sigma \le 1, r \le \tau \le 1, \\
\label{eq:V_vt_eq}
&V_{\vt_r} = \basis{\vt}_{r, 0}''(0) - \basis{\vt}_{r, 0}'(0)^2  = \int_r^1 \int_r^1 K_{\vt_r}(\sigma, \tau) \del \sigma \del \tau, \\
\label{eq:V_vt_ineq}
&0 < V_{\vt_r} \le (1 - r)^2 / 4, \\
\label{eq:basis_kernel_ineq}
&V_{\vt_r} x y \enum^{ - ( x + y) } \le \basis{\vt}_{r, 0}(x + y) - \basis{\vt}_{r, 0}(x) \basis{\vt}_{r, 0}(y) \le V_{\vt_r} x y \enum^{ - r ( x + y) }, \\
\label{eq:basis_kernel_ineq2}
&\basis{\vt}_{r, 0}(x + y) - \basis{\vt}_{r, 0}(x) \basis{\vt}_{r, 0}(y) \le (\enum^{-r x} - \enum^{-x}) (\enum^{-r y} - \enum^{-y}) / 4, \\
\label{eq:det_basis_kernel_ineq}
&\det( \basis{\vt}_{r, 0}(x_i + y_j) - \basis{\vt}_{r, 0}(x_i) \basis{\vt}_{r, 0}(y_j) )_{ i, j = 1 }^n > 0, 0 < x_1 < \cdots < x_n < \infty, 0 < y_1 < \cdots < y_n < \infty, n \in \intge{1}.
\end{align}
\end{lemma}
\begin{proof}
By using \eqref{eq:basis_fl} and \eqref{eq:basis_0_dWn}, we have
\begin{equation}
\label{eq:basis_kernel_int_W0}
\basis{\vt}_{r, 0}(z + w) - \basis{\vt}_{r, 0}(z) \basis{\vt}_{r, 0}(w) = \frac{1}{2} \int_r^1 \int_r^1 (\enum^{- z \sigma} - \enum^{- z \tau}) (\enum^{- w \sigma} - \enum^{- w \tau}) \del W_{\vt_r, 0}(\sigma) \del W_{\vt_r, 0}(\tau).
\end{equation}
By introducing $
1_{\sigma, \tau}(t) \coloneqq
\begin{cases}
1, & t \in [\min(\sigma, \tau), \max(\sigma, \tau)], \\
0, & t \notin [\min(\sigma, \tau), \max(\sigma, \tau)],
\end{cases}
$ $r \le \sigma \le 1$, $r \le \tau \le 1$, $r \le t \le 1$, we have
\begin{equation}
\label{eq:iint_exp}
(\enum^{- z \sigma} - \enum^{- z \tau}) (\enum^{- w \sigma} - \enum^{- w \tau}) = \left( z \int_\sigma^\tau \enum^{- z s} \del s \right) \left( w \int_\sigma^\tau \enum^{- w t} \del t \right) = z w \int_r^1 \enum^{- z s} 1_{\sigma, \tau}(s) \del s \int_r^1 \enum^{- w t} 1_{\sigma, \tau}(t) \del t.
\end{equation}
By substituting \eqref{eq:iint_exp} for \eqref{eq:basis_kernel_int_W0} and using
\begin{align}
&\int_r^1 \int_r^1 1_{\sigma, \tau}(s) 1_{\sigma, \tau}(t) \del W_{\vt_r, 0}(\sigma) \del W_{\vt_r, 0}(\tau) \\ \nonumber
&= \int_r^{\min(s, t)} \del W_{\vt_r, 0}(\sigma) \int_{\max(s, t)}^1 \del W_{\vt_r, 0}(\tau) + \int_{\max(s, t)}^1 \del W_{\vt_r, 0}(\sigma) \int_r^{\min(s, t)} \del W_{\vt_r, 0}(\tau) \\ \nonumber
&= 2 W_{\vt_r, 0}(\min(s, t)) (1 - W_{\vt_r, 0}(\max(s, t))) = 2 K_{\vt_r}(s, t), r \le s \le1, r \le t \le 1,
\end{align}
we have \eqref{eq:basis_kernel_int}.

Let $r \le \sigma \le 1$ and $r \le \tau \le 1$. $K_{\vt_r}(\sigma, \tau) \ge 0$ in \eqref{eq:K_vt_ineq} follows from $0 \le W_{\vt_r, 0}(\tau) \le 1$ by \eqref{eq:basis_W0_increase}. Since $W_{\vt_r, 0}(\tau)$ is a strictly increasing function by \eqref{eq:basis_W0_increase}, we have $K_{\vt_r}(\sigma, \tau) \le W_{\vt_r, 0}(\max(\sigma, \tau)) (1 - W_{\vt_r, 0}(\max(\sigma, \tau)))$. Then by using $0 \le W_{\vt_r, 0}(\tau) \le 1$ and $s (1 - s) \le 1/4$ for $0 \le s \le 1$, we have $K_{\vt_r}(\sigma, \tau) \le 1 / 4$ in \eqref{eq:K_vt_ineq}.

The first equality in \eqref{eq:V_vt_eq} follows from \eqref{eq:basis_deriv_W} and \eqref{eq:V_vt}. By comparing each coefficient of $zw$ term of Taylor series at $(0, 0)$ for both sides in \eqref{eq:basis_kernel_int}, we have the second equality in \eqref{eq:V_vt_eq}. By \eqref{eq:basis_W0_increase}, $W_{\vt_r, 0}(\tau)$ can be interpreted as a cumulative distribution function on $[r, 1]$. Then $V_{\vt_r}$ defined in \eqref{eq:V_vt} represents the variance of the distribution. Since $W_{\vt_r, 0}(\tau)$ has infinitely many points of increase on $[r, 1]$ by \eqref{eq:basis_W0_increase}, we have $V_{\vt_r} > 0$ in \eqref{eq:V_vt_ineq}. $V_{\vt_r} \le (1 - r)^2 / 4$ in \eqref{eq:V_vt_ineq} follows from Popoviciu's inequality on variances or \eqref{eq:K_vt_ineq} and \eqref{eq:V_vt_eq}. \eqref{eq:basis_kernel_ineq} follows from \eqref{eq:basis_kernel_int}, \eqref{eq:K_vt_ineq}, and \eqref{eq:V_vt_eq}. \eqref{eq:basis_kernel_ineq2} follows from \eqref{eq:basis_kernel_int}, \eqref{eq:K_vt_ineq}, and $x \int_r^1 \enum^{-x \sigma} \del \sigma = \enum^{-r x} - \enum^{-x}$.

By introducing $I(x, y) \coloneqq \int_r^1 \int_r^1 \enum^{- x \sigma} \enum^{ - y \tau } K_{\vt_r}(\sigma, \tau) \del \sigma \del \tau$ on $\mathbb{R} \times \mathbb{R}$ and using \eqref{eq:basis_kernel_int}, we have $\basis{\vt}_{r, 0}(x + y) - \basis{\vt}_{r, 0}(x) \basis{\vt}_{r, 0}(y) = x y I(x, y)$. By the change of variables $\sigma = -s$ and $\tau = -t$, we have $I(x, y) = \int_{-1}^{-r} \int_{-1}^{-r} \enum^{x s} \enum^{y t} K(s, t) \del s \del t$, where $K(s, t) \coloneqq K_{\vt_r}(- s, - t)$ on $[-1, -r] \times [-1, -r]$. By introducing $\phi(t) \coloneqq (1 - W_{\vt_r, 0}(- t))$ and $\psi(t) \coloneqq W_{\vt_r, 0}(- t)$ and using \eqref{eq:K_vt}, we have $K(s, t) = \phi(\min(s, t)) \psi(\max(s, t))$ on $[-1, -r] \times [-1, -r]$. By \eqref{eq:basis_W0_increase}, $\phi(t)$ and  $\psi(t)$ are continuous, satisfy $\phi(t) \psi(t) > 0$, and $\phi(t) / \psi(t)$ is strictly increasing for $-1 < t < -r$. Then we have \cite[Chap.~3, Corollary~3.1]{Karlin1968}
\begin{align}
\label{eq:det_K_ineq}
&\det( K(s_i, t_j) )_{i, j = 1}^n
\begin{cases}
> 0, & n = 1, \\
> 0, & \max(s_1, t_1) < \min(s_2, t_2), \ldots, \\
& \max(s_{n - 1}, t_{n - 1}) < \min(s_n, t_n), n \in \intge{2}, \\
= 0, & \text{otherwise},
\end{cases} \\ \nonumber
&-1 < s_1 < \cdots < s_n < -r, -1 < t_1 < \cdots < t_n < -r, n \in \intge{1}.
\end{align}
By applying the basic composition formula to $J(x, t) \coloneqq \int_{-1}^{-r} \enum^{x s} K(s, t) \del s$ and using \eqref{eq:det_exp_ineq} and \eqref{eq:det_K_ineq}, (i) $J(x, t)$ is strictly totally positive on $\mathbb{R} \times (-1, -r)$. By applying the basic composition formula to $I(x, y) = \int_{-1}^{-r} \enum^{y t} J(x, t) \del t$ and using \eqref{eq:det_exp_ineq} and (i), (ii) $I(x, y)$ is strictly totally positive on $\mathbb{R} \times \mathbb{R}$. Then by using 
\begin{equation}
\det( \basis{\vt}_{r, 0}(x_i + y_j) - \basis{\vt}_{r, 0}(x_i) \basis{\vt}_{r, 0}(y_j) )_{i, j = 1}^n = \det( I(x_i, y_j) )_{i, j = 1}^n \prod_{i = 1}^n x_i \prod_{j = 1}^n y_j
\end{equation}
and (ii), we have \eqref{eq:det_basis_kernel_ineq}.
\end{proof}
Examples of $V_{\vt_r}$ are found in \eqref{eq:V_Phi} and \eqref{eq:V_P1}. By \eqref{eq:det_basis_kernel_ineq}, $\basis{\vt}_{r, 0}(x + y) - \basis{\vt}_{r, 0}(x) \basis{\vt}_{r, 0}(y)$ is strictly totally positive on $(0, \infty) \times (0, \infty)$.

\section{A variable transformation $\vt_r(u)$ to maximize $\hat{\rho}[\vt_r]$}
\label{sec:max_hatrho}
\subsection{Construction of the variable transformation}
\label{sec:dn}
In this section, we find a function $\vt_r(u)$, which maximizes $\hat{\rho}[\vt_r]$ under the conditions $\vt_r \in \vtset_r$ and $\rhoana[\vt_r] > 1$ for a given $0 < r < 1$. For this purpose, we focus on the image $\vt_r(\ellipse_{\hat{\rho}[\vt_r]})$. Since $\ellipse_{\hat{\rho}[\vt_r]}$ is an open set and $\vt_r(u)$ is a nonconstant analytic function on $\ellipse_{\hat{\rho}[\vt_r]}$, $\vt_r(\ellipse_{\hat{\rho}[\vt_r]})$ is an open set by the open mapping theorem. Then by $\inf\limits_{u \in \ellipse_{\hat{\rho}[\vt_r]}} \RE \vt_r(u) \ge 0$ in \eqref{eq:infRe_hatrho_ineq}, $\vt_r(\ellipse_{\hat{\rho}[\vt_r]})$ is a subset of the right-half plane, i.e., 
\begin{equation}
\label{eq:vt_rhp}
\vt_r(\ellipse_{\hat{\rho}[\vt_r]}) \subseteq \rhp \coloneqq \{ \tau \in \mathbb{C} \mid \RE \tau > 0 \}.
\end{equation}
For the inclusion relationship of images of two analytic functions defined on the open unit disk $\mathbb{D} \coloneqq \{ z \in \mathbb{C} \mid \lvert z \rvert < 1 \}$, the following result is known:
\begin{theorem}[{\cite[Sec.~V.9.~Subordination]{Nehari1975}}]
\label{the:subordination}
Let $f(z)$ and $F(z)$ be analytic functions on $\mathbb{D}$. Let $F(z)$ be a one-to-one mapping of $\mathbb{D}$ onto $F(\mathbb{D})$. If $f(0) = F(0)$ and $f(\mathbb{D}) \subseteq F(\mathbb{D})$, then there exists an analytic function $\omega(z)$ on $\mathbb{D}$ satisfying $f(z) = F(\omega(z))$, $\lvert \omega(z) \rvert < \lvert z \rvert$ if $f(\mathbb{D}) \subsetneq F(\mathbb{D})$, and $\lvert \omega(z) \rvert = \lvert z \rvert$ if $f(\mathbb{D}) = F(\mathbb{D})$.
\end{theorem}
\begin{corollary}
\label{cor:subordination_pos}
Let the conditions of Theorem \ref{the:subordination} be satisfied. If $f(z)$ and $F(z)$ are real on $(-1, 1)$, one has $f(\bar{z}) = \overline{f(z)}$, $F(\bar{z}) = \overline{F(z)}$, and $\omega(\bar{z}) = \overline{\omega(z)}$ on $\mathbb{D}$.
\end{corollary}
\begin{proof}
Let $z \in \mathbb{D}$. $f(\bar{z}) = \overline{f(z)}$ and $F(\bar{z}) = \overline{F(z)}$ follow from the Schwarz reflection principle. Then we have $f(\bar{z}) = F(\omega(\bar{z})) = \overline{f(z)} = \overline{F(\omega(z))} = F(\overline{\omega(z)})$. Since $F(z)$ is one-to-one, $F(\omega(\bar{z})) = F(\overline{\omega(z)})$ indicates $\omega(\bar{z}) = \overline{\omega(z)}$.
\end{proof}
Since our region is not the open unit disk but an open Bernstein ellipse, we introduce a function to convert them based on the Riemann mapping theorem as follows:
\begin{theorem}[{\cite[Sec.~6.1.~The Riemann mapping theorem: Theorem 1 and Exercises 1]{Ahlfors1979}}]
\label{the:riemann}
Given any simply connected region $\Omega$ which is not the whole plane, and a point $z_0 \in \Omega$, there exists a unique analytic function $f(z)$ in $\Omega$, normalized by the conditions $f(z_0) = 0$, $f'(z_0) > 0$, such that $f(z)$ defines a one-to-one mapping of $\Omega$ onto $\mathbb{D}$. If $z_0$ is real and $\Omega$ is symmetric with respect to the real axis, $f$ satisfies the symmetry relation $f(\bar{z}) = \overline{f(z)}$.
\end{theorem}
\begin{lemma}[{\cite[Exercises IV.7.4]{Conway1973}}]
\label{lem:dfne0}
Let $G$ be a region. If $f: G \to \mathbb{C}$ is analytic and one-to-one, then $f'(z) \ne 0$ on $G$.
\end{lemma}
\begin{corollary}
\label{cor:dfpos}
Let $f(z)$ be analytic and one-to-one on $\mathbb{D}$. If $f(z)$ is real on $(-1, 1)$ and satisfies $f'(c) > 0$ for a point $c \in (-1, 1)$, one has $f'(z) > 0$ on $(-1, 1)$.
\end{corollary}
\begin{lemma}
\label{lem:Erho}
For a given $1 < \rho < \infty$, there exists a function $E_\rho(z)$ on $\mathbb{D}$ satisfying
\begin{subequations}
\label{eq:Erho}
\begin{align}
\label{eq:Erho_map_D}
&\text{$E_\rho(z)$ is a one-to-one analytic mapping  of $\mathbb{D}$ onto $\ellipse_\rho$}, \\
\label{eq:Erho_increase}
&\text{$E_\rho(0) = -1$, $E_\rho(\bar{z}) = \overline{E_\rho(z)}$ on $\mathbb{D}$, $E_\rho'(x) > 0$ for $-1 < x < 1$}, \\
\label{eq:Erho_map_real}
&\text{$E_\rho(x)$ is a one-to-one mapping of $(-1, 1)$ onto $\left( - \frac{ \rho + \rho^{-1} }{ 2 }, \frac{ \rho + \rho^{-1} }{ 2 } \right)$}.
\end{align}
\end{subequations}
For a given $1 < \rho < \infty$, there is a unique real number $x_\rho$ satisfying
\begin{equation}
\label{eq:Erho_xrho}
\text{$0 < x_\rho < 1$ and $E_\rho(x_\rho) = 1$}.
\end{equation}
Furthermore, one has
\begin{align}
\label{eq:Erho_ineq}
&\text{$E_\rho(x) < E_R(x)$ for $0 < x < 1$ if $1 < \rho < R < \infty$}, \\
\label{eq:xrho_decrease}
&\text{$x_\rho$ is strictly decreasing for $1 < \rho < \infty$}.
\end{align}
\end{lemma}
\begin{proof}
Let $1 < \rho < \infty$. By Theorem \ref{the:riemann}, there exists a unique one-to-one analytic mapping $F_\rho(u)$ of $\ellipse_\rho$ onto $\mathbb{D}$ satisfying $F_\rho(-1) = 0$, $F_\rho'(-1) > 0$, and $F_\rho(\bar{u}) = \overline{F_\rho(u)}$ on $\ellipse_\rho$. If $F_\rho(u)$ is real, we have $F_\rho(\bar{u}) = \overline{F_\rho(u)} = F_\rho(u) $. Since $F_\rho(u)$ is one-to-one on $\ellipse_\rho$, we have $\bar{u} = u$, i.e., $u$ is real. Thus $F(u)$ is real iff $u$ is real and $F(u)$ is a one-to-one mapping of $(-(\rho + \rho^{-1}) / 2, (\rho + \rho^{-1}) / 2)$ onto $(-1, 1)$. By introducing $E_\rho(z) \coloneqq F_\rho^{-1}(z)$ on $\mathbb{D}$ and using Corollary \ref{cor:dfpos}, we have \eqref{eq:Erho}. The existence of a unique real number $x_\rho$ satisfying \eqref{eq:Erho_xrho} follows from \eqref{eq:Erho_increase} amd \eqref{eq:Erho_map_real} 

Let $1 < \rho < R < \infty$. We have $E_\rho(\mathbb{D}) = \ellipse_\rho \subsetneq \ellipse_R = E_R(\mathbb{D})$ and $E_\rho(0) = E_R(0) = -1$. By Theorem \ref{the:subordination}, there exists an analytic function $\omega(z)$ satisfying $E_\rho(z) = E_R(\omega(z))$ and $\lvert \omega(z) \rvert < \lvert z \rvert$ on $\mathbb{D}$. Since $\omega(x)$ is real for $-1 < x < 1$ by Corollary \ref{cor:subordination_pos} and $E_R'(x) > 0$ for $-1 < x < 1$, we have $E_\rho(x) = E_R(\omega(x)) < E_R(x)$ for $0 < x < 1$, which leads to \eqref{eq:Erho_ineq}. By using $E_R(x_R) = 1 = E_\rho(x_\rho) < E_R(x_\rho)$ and $E_R'(x) > 0$ for $-1 < x < 1$, we have $x_R < x_\rho$, which leads to \eqref{eq:xrho_decrease}.
\end{proof}

Based on these preparations, we obtain the characterization of the variable transformation $\vt_r(u)$, which maximizes $\hat{\rho}[\vt_r]$ as follows:
\begin{lemma}
\label{lem:max_hatrho}
Let $0 < r < 1$. Let $\vt_r \in \vtset_r$ and $\rhoana[\vt_r] > 1$. Let $\Phi_r \in \vtset_r$ and $\rhoana[\Phi_r] > 1$. If $\Phi_r(u)$ is a one-to-one mapping of $\ellipse_{\hat{\rho}[\Phi_r]}$ onto $\rhp$, one has
\begin{align}
\label{eq:hatrho_ineq}
&\text{$\hat{\rho}[\vt_r] < \hat{\rho}[\Phi_r]$ if $\vt_r(\ellipse_{\hat{\rho}[\vt_r]}) \subsetneq \rhp$}, \\
\label{eq:hatrho_eq}
&\text{$\hat{\rho}[\vt_r] = \hat{\rho}[\Phi_r]$ and $\vt_r(u) = \Phi_r(u)$ on $\ellipse_{\hat{\rho}[\Phi_r]}$ if $\vt_r(\ellipse_{\hat{\rho}[\vt_r]}) = \rhp$}.
\end{align}
\end{lemma}
\begin{proof}
We use the notations in Lemma \ref{lem:Erho}. Let $f(z) \coloneqq \vt_r(E_{\hat{\rho}[\vt_r]}(z))$ and $F(z) \coloneqq \Phi_r(E_{\hat{\rho}[\Phi_r]}(z))$ on $\mathbb{D}$. By using the analyticity of $\vt_r(u)$ (resp. $\Phi_r(u)$) on $\ellipse_{\hat{\rho}[\vt_r]}$ (resp. $\ellipse_{\hat{\rho}[\Phi_r]}$), \eqref{eq:vt_increase_eq}, Lemma \ref{lem:ellipse_real}, and Lemma \ref{lem:Erho}, $f(z)$ and $F(z)$ are analytic on $\mathbb{D}$, real on $(-1, 1)$, and satisfy
\begin{equation}
\label{eq:hatrho_fxhatrho}
f(0) = F(0) = r, f(x_{\hat{\rho}[\vt_r]}) = F(x_{\hat{\rho}[\Phi_r]}) = 1.
\end{equation}
Since $E_{\hat{\rho}[\Phi_r]}(z)$ is a one-to-one mapping of $\mathbb{D}$ onto $\ellipse_{\hat{\rho}[\Phi_r]}$ and $\Phi_r(u)$ is a one-to-one mapping of $\ellipse_{\hat{\rho}[\Phi_r]}$ onto $\rhp$, $F(z)$ is a one-to-one mapping of $\mathbb{D}$ onto $\rhp$. By \eqref{eq:Erho_map_D} and \eqref{eq:vt_rhp}, we have $f(\mathbb{D}) = \vt_r(\ellipse_{\hat{\rho}[\vt_r]}) \subseteq \rhp = F(\mathbb{D})$. By Theorem \ref{the:subordination}, there exist an analytic function $\omega(z)$ satisfying $f(z) = F(\omega(z))$ and  $\lvert \omega(z) \rvert \le \lvert z \rvert$ on $\mathbb{D}$. By Corollary \ref{cor:subordination_pos}, we have
\begin{equation}
\label{eq:hatrho_omega_real}
\text{$\omega(z)$ is real on $(-1, 1)$}.
\end{equation}
By using \eqref{eq:vt_increase}, \eqref{eq:Erho_increase}, and \eqref{eq:Erho_xrho}, we have
\begin{equation}
\label{eq:hatrho_df_pos}
f'(x) > 0, 0 < x < x_{\hat{\rho}[\vt_r]},
\end{equation}
and $F'(x) > 0, 0 < x < x_{\hat{\rho}[\Phi_r]}$. By Corollary \ref{cor:dfpos}, we have
\begin{equation}
\label{eq:hatrho_dF_pos}
F'(x) > 0, -1 < x < 1.
\end{equation}

First, we consider the case for $\vt_r(\ellipse_{\hat{\rho}[\vt_r]}) \subsetneq \rhp$, i.e., $f(\mathbb{D}) \subsetneq \rhp = F(\mathbb{D})$. By using $\lvert \omega(z) \rvert < \lvert z \rvert$ by Theorem \ref{the:subordination}, \eqref{eq:hatrho_omega_real}, and \eqref{eq:hatrho_dF_pos}, we have
\begin{equation}
\label{eq:max_hatrho_f_ineq}
\text{$f(x) = F(\omega(x)) < F(x)$ for $0 < x < 1$ if $\vt_r(\ellipse_{\hat{\rho}[\vt_r]}) \subsetneq \rhp$}.
\end{equation}
By substituting $x = x_{\hat{\rho}[\vt_r]}$ for \eqref{eq:max_hatrho_f_ineq} and using \eqref{eq:hatrho_fxhatrho}, we have $F(x_{\hat{\rho}[\Phi_r]}) < F(x_{\hat{\rho}[\vt_r]})$. By using \eqref{eq:hatrho_dF_pos} and \eqref{eq:xrho_decrease}, we have \eqref{eq:hatrho_ineq}.

Next, we consider the case for $\vt_r(\ellipse_{\hat{\rho}[\vt_r]}) = \rhp$, i.e., $f(\mathbb{D}) = \rhp = F(\mathbb{D})$. By Theorem \ref{the:subordination}, we have $\lvert \omega(z) \rvert = \lvert z \rvert$. Since $\omega(z)$ is analytic on $\mathbb{D}$ and real on $(-1, 1)$, we see that $\omega(z) = z$ or $\omega(z) = -z$. To hold \eqref{eq:hatrho_df_pos} and \eqref{eq:hatrho_dF_pos} by $f(x) = F(\omega(x))$ for $0 < x < x_{\hat{\rho}[\vt_r]}$, we can conclude that $\omega(z) = z$, i.e., $f(z) = F(z)$. By substituting $z = x_{\hat{\rho}[\vt_r]}$ for $f(z) = F(z)$ and using \eqref{eq:hatrho_fxhatrho} and \eqref{eq:hatrho_dF_pos}, we have $x_{\hat{\rho}[\vt_r]} = x_{\hat{\rho}[\Phi_r]}$. By \eqref{eq:xrho_decrease}, we have $\hat{\rho}[\vt_r] = \hat{\rho}[\Phi_r]$ and $\vt_r(E_{\hat{\rho}[\Phi_r]}(z)) = \vt_r(E_{\hat{\rho}[\vt_r]}(z)) = f(z) = F(z) = \Phi_r(E_{\hat{\rho}[\Phi_r]}(z))$ on $\mathbb{D}$, which leads to \eqref{eq:hatrho_eq}.
\end{proof}

By Lemma \ref{lem:max_hatrho}, if there exists a function $\Phi_r(u)$, then it give the maximum value of $\hat{\rho}[\vt_r]$ by $\hat{\rho}[\Phi_r]$ and the function is uniquely determined. Although Lemma \ref{lem:max_hatrho} do not guarantee the existence of the function, we can construct it based on Jacobi's delta amplitude function $\dn(v, k)$ with the modulus $k \coloneqq \sqrt{1 - r^2}$ for a given $0 < r < 1$ as shown below. $\dn(v, k)$ is the doubly periodic meromorphic function with simple poles and simple zeros on $\mathbb{C}$, real on the real axis, and satisfies \cite[Sec.~22]{DLMF} \cite{MilneThomson1972}
\begin{align}
\label{eq:dnperiod}
&\dn(v + 2 \EK(k) m + \inum 4 \EK(r) n, k) = \dn(v, k), m \in \mathbb{Z}, n \in \mathbb{Z}, v \in \mathbb{C}, \\
\label{eq:dneven}
&\dn(-v, k) = \dn(v, k), v \in \mathbb{C}, \\
\label{eq:dnK}
&\dn(v + \EK(k), k) = r / \dn(v, k), v \in \mathbb{C}, \\
\label{eq:dni2Kp}
&\dn(v + \inum 2 \EK(r), k) = -\dn(v, k), v \in \mathbb{C}, \\
\label{eq:ddn}
&\del \dn(v, k) / \del v = -k^2 \sn(v, k) \cn(v, k), v \in \mathbb{C},
\end{align}
\begin{subequations}
\label{eq:dnvalue}
\begin{align}
&\dn(0, k) = 1, \dn(\EK(k) / 2, k) = \sqrt{r}, \dn(\EK(k), k) = r, \\
&\left. \frac{ \del \dn(v, k) }{ \del v }\right\vert_{v = 0} = \left. \frac{ \del \dn(v, k) }{ \del v }\right\vert_{v = \EK(k)} = 0, \frac{ \del \dn(x, k) }{ \del x } < 0, 0 < x < \EK(k), \\
&\text{$v = \EK(k) + \inum \EK(r)$ is simple zero and $v = \inum \EK(r)$ is simple pole of $\dn(v, k)$},
\end{align}
\end{subequations}
\begin{equation}
\label{eq:dn2}
\sn(v, k)^2 + \cn(v, k)^2 = k^2 \sn(v, k)^2 + \dn(v, k)^2 = 1, r^2 \sn(v, k)^2 + \cn(v, k)^2 = k^2 \cn(v, k)^2 + r^2 = \dn(v, k)^2, v \in \mathbb{C},
\end{equation}
where $\sn(v, k)$ (resp. $\cn(v, k)$) is Jacobi's sine (resp. cosine) amplitude function with the modulus $k$. For real intervals $I_1$ and $I_2$, we use notations $I_1 + \inum I_2 \coloneqq \{ z \in \mathbb{C} \mid \RE z \in I_1, \IM z \in I_2 \}$ and $\inum I_2 \coloneqq \{ z \in \mathbb{C} \mid \RE z = 0, \IM z \in I_2 \}$. dn function has the following mapping property:
\begin{theorem}[{\cite[Sec.~13.3 $w = \dn z$]{Kober1957}}]
\label{the:dn_conformal}
Let $0 < r < 1$ and $k \coloneqq \sqrt{1 - r^2}$. Then $\dn(v, k)$ is a one-to-one conformal mapping of $(-\EK(k), \EK(k)) + \inum (0, \EK(r))$ onto $\rhp \setminus (0, 1]$. Furthermore, $\dn(v, k)$ is a one-to-one mapping of
\begin{subequations}
\begin{align}
\label{eq:dn_conformal_tl}
&\text{$(-\EK(k), 0) + \inum (0, \EK(r))$ onto $(0, \infty) + \inum (0, \infty)$,} \\
\label{eq:dn_conformal_tr}
&\text{$(0, \EK(k)) + \inum (0, \EK(r))$ onto $(0, \infty) + \inum (-\infty, 0)$,} \\
&\text{$\inum (0, \EK(r))$ onto $(1, \infty)$.}
\end{align}
\end{subequations}
\end{theorem}
The rectangular region $(-\EK(k), \EK(k)) + \inum (0, \EK(r))$ in Theorem \ref{the:dn_conformal} do not include the real axis. To fit our purpose, we use the following rectangular region, which is symmetric with respect to the real axis ($v$-plane in Figure \ref{fig:map}):
\begin{corollary}
\label{cor:dn_conformal}
Let $0 < r < 1$ and $k \coloneqq \sqrt{1 - r^2}$. Then $\dn(v, k)$ is a one-to-one conformal mapping of $(0, \EK(k)) + \inum (- \EK(r), \EK(r))$ onto $\rhp \setminus ( (0, r] \cup [1, \infty) )$ ($v$-plane to $\tau$-plane in Figure \ref{fig:map}).
\end{corollary}
\begin{proof}
By \eqref{eq:dneven}, \eqref{eq:dn_conformal_tl}, and \eqref{eq:dn_conformal_tr}, $\dn(v, k)$ is a one-to-one mapping of $( (0, \EK(k)) + \inum (-\EK(r), \EK(r)) ) \setminus (0, \EK(k))$ onto $\rhp \setminus (0, \infty)$. Since $\dn(v, k)$ is strictly decreasing for $0 \le v \le \EK(k)$ and satisfies $\dn(0, k) = 1$ and $\dn(\EK(k), k) = r$, $\dn(v, k)$ is a one-to-one mapping of $(0, \EK(k))$ onto $(r, 1)$. Then $\dn(v, k)$ is a one-to-one mapping of $(0, \EK(k)) + \inum (-\EK(r), \EK(r))$ onto $\rhp \setminus ( (0, r] \cup [1, \infty) )$. Since $\dn(v, k)$ is the meromorphic function and there is no pole on $(0, \EK(k)) + \inum (-\EK(r), \EK(r) )$, $\dn(v, k)$ is analytic on there. Then by Lemma \ref{lem:dfne0}, $\frac{ \del \dn(v, k) }{ \del v } \ne 0$ on $(0, \EK(k)) + \inum (-\EK(r), \EK(r) )$. Thus $\dn(v, k)$ is a one-to-one conformal mapping of $(0, \EK(k)) + \inum (-\EK(r), \EK(r) )$ onto $\rhp \setminus ( (0, r] \cup [1, \infty) )$.
\end{proof}

\begin{figure}[htbp]
\centering
\includegraphics{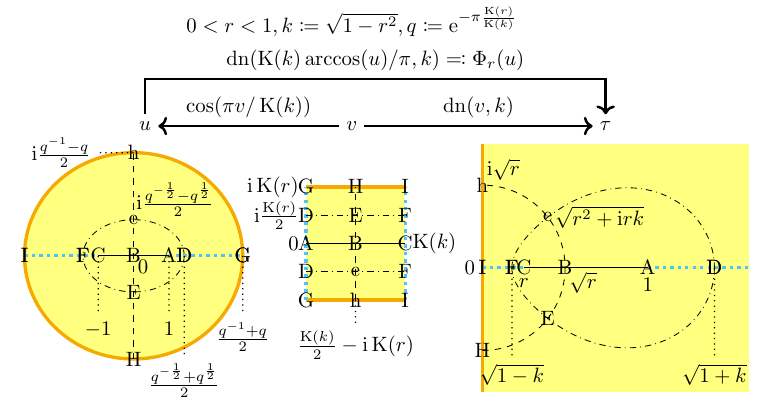}
\caption{Construction of the mapping of an open Bernstein ellipse onto the right half-plane such that $-1$ and $1$ are mapped onto $0 < r < 1$ and $1$, respectively. Since $v = \pm \inum \EK(r)$ are poles of $\dn(v, k)$, G is not displayed on $\tau$-plane. The dashed curve on $\tau$-plane corresponds to the half-circle satisfying $\lvert \tau \rvert = \sqrt{r}$ by \eqref{eq:dn_K_2_iKrl} or \eqref{eq:Phi_iy}. The dash-dotted curve on $\tau$-plane corresponds to the Cassinian oval satisfying $\lvert \tau + 1 \rvert \lvert \tau - 1 \rvert = k$ \cite[Sec.~13.3 $w = \dn z$]{Kober1957} by \eqref{eq:dnxmiKr_2} or \eqref{eq:Phi_ellipse1_2}. In this figure, we set $r = 1/4$.}
\label{fig:map}
\end{figure}

Based on the mapping in Corollary \ref{cor:dn_conformal}, we can construct the mapping of an open Bernstein ellipse with slits onto the right half-plane with slits ($u$-plane to $\tau$-plane in Figure \ref{fig:map}; slits are represented by horizontal bold dotted lines) as follows:
\begin{lemma}
\label{lem:utotau}
Let $\arccos(u)$ be the principal value of the inverse cosine function on $\mathbb{C} \setminus ( (-\infty, -1] \cup [1, \infty) )$, i.e., $\mathbb{C} \setminus ( (-\infty, -1] \cup [1, \infty) )$ is mapped onto $(0, \pi) + \inum (-\infty, \infty)$. Let $0 < r < 1$, $k \coloneqq \sqrt{ 1 - r^2 }$, and $q \coloneqq \enum^{ - \pi \frac{\EK(r)}{\EK(k)} }$. Then $\tau = \dn(\EK(k) \arccos(u) / \pi, k)$ is a one-to-one conformal mapping of $\ellipse_{q^{-1}} \setminus ( (- (q^{-1} + q) / 2, -1] \cup [1,  (q^{-1} + q) / 2) )$ onto $\rhp \setminus ( (0, r] \cup [1, \infty) )$.
\end{lemma}
\begin{proof}
Let $u_1 \coloneqq (q^{-1} + q) / 2$. The cosine function $u = \cos(\pi v / \EK(k))$ is a one-to-one conformal mapping of $(0, \EK(k)) + \inum (-\EK(r), \EK(r))$ onto $\ellipse_{q^{-1}} \setminus ( (- u_1, -1] \cup [1, u_1) )$ \cite[Sec.~1.13.B]{Davis1975} ($v$-plane to $u$-plane in Figure \ref{fig:map}). Then $v = \EK(k) \arccos(u) / \pi$ is a one-to-one conformal mapping of $\ellipse_{q^{-1}} \setminus ( (- u_1, -1] \cup [1, u_1) )$ onto $(0, \EK(k)) + \inum (-\EK(r), \EK(r))$. The statement follows from the composition of this mapping and the mapping in Corollary \ref{cor:dn_conformal}.
\end{proof}

We considered the open Bernstein with slits $\ellipse_{q^{-1}} \setminus ( (- (q^{-1} + q) / 2, -1] \cup [1,  (q^{-1} + q) / 2) )$ to keep the single-valuedness of the mapping from $u$ to $v$. However, the slits are unnecessary for the mapping from $u$ to $\tau$ as follows. The multivaluedness of the mapping from $u$ to $v$ is caused by the invariance of $u = \cos(\pi v / \EK(k))$ by changing $v$ to $\pm v + 2 \EK(k) n$ for $n \in \mathbb{Z}$. However, since the equality $\dn(\pm v + 2 \EK(k) n, k) = \dn(v, k)$ for $n \in \mathbb{Z}$ holds by \eqref{eq:dnperiod} and \eqref{eq:dneven}, the mapping from $u$ to $\tau$ becomes single-valued. This point can be clarified with the following infinite product representation:
\begin{lemma}
\label{lem:Phi}
Let $0 < r < 1$, $k \coloneqq \sqrt{ 1 - r^2 }$, and $q \coloneqq \enum^{- \pi \frac{ \EK(r) }{ \EK(k) } }$. Let us introduce \begin{equation}
\label{eq:Phi}
\Phi_r(u) \coloneqq \dn(\EK(k) \arccos(u) / \pi, k), u \in \mathbb{C},
\end{equation}
where $\arccos(u)$ is the multivalued inverse cosine function on $\mathbb{C}$.
Then one has
\begin{equation}
\label{eq:Phi_prod}
\Phi_r(u) = \sqrt{r} \frac{\prod_{n = 1}^\infty (1 + 2 q^{2n-1} u + q^{4n-2})} {\prod_{n = 1}^\infty (1 - 2 q^{2n-1} u + q^{4n- 2})}, u \in \mathbb{C},
\end{equation}
which indicates $\Phi_r(u)$ is the single-valued meromorphic function on $\mathbb{C}$ with simple zeros $u = - (q^{-2n+1} + q^{2n-1}) / 2$ and simple poles $u = (q^{-2n+1} + q^{2n-1}) / 2$ for $n \in \intge{1}$.
\end{lemma}
\begin{proof}
By substituting $z = \arccos(u) / 2$ for Jacobi's infinite product representation of $\dn$ function \eqref{eq:dn_ip}, we have \eqref{eq:Phi_prod}. By introducing $u_n \coloneqq (q^{-2n+1} + q^{2n-1}) / 2, n \in \intge{1}$, and $G(u) \coloneqq \prod_{n = 1}^\infty (1 - u / u_n)$, \eqref{eq:Phi_prod} can be expressed as $\sqrt{r} G(-u) / G(u)$. Since inequalities $\sum_{n = 1}^\infty 1 / \lvert u_n \rvert < 2 q / (1 - q^2) < \infty$ hold by noting $0 < q < 1$ for $0 < r < 1$, $G(u)$ is an entire function \cite[Sec.~5.2.3]{Ahlfors1979} with simple zeros $u = u_n, n \in \intge{1}$. Thus $\sqrt{r} G(-u) / G(u)$ is the meromorphic function with simple zeros $u = - u_n$ and simple poles $u = u_n$ for $n \in \intge{1}$.
\end{proof}
Based on the properties of $\dn$ function and Lemma \ref{lem:Phi}, $\Phi_r(u)$ satisfies the following properties:
\begin{definition}
Let $0 < k < 1$, $z \in \mathbb{C}$, and $x \in \mathbb{R}$. Let $\cs(z, k) \coloneqq \frac{\cn(z, k)}{\sn(z, k)}$, $\dc(z, k) \coloneqq \frac{\dn(z, k)}{\cn(z, k)}$, and $\cd(z, k) \coloneqq \frac{\cn(z, k)}{\dn(z, k)}$. $\am(x, k)$ is Jacobi's amplitude function with the modulus $k$. For $n \in \intge{0}$, $V_n(z)$ is the $n$-degree Chebyshev polynomial of the third kind \cite{Mason2003}.
\end{definition}
\begin{lemma}
\label{lem:Phi_eq}
Let us use the notation of Lemma \ref{lem:Phi}. Let $\sqrt{z}$ be the principal value of the square root function on $\mathbb{C} \setminus (-\infty, 0]$, i.e., $\sqrt{z} \in \rhp$. Then one has
\begin{align}
\label{eq:Phi_cos}
&\Phi_r(\cos z) = \dn(\EK(k) z / \pi , k), z \in \mathbb{C}, \\
\label{eq:Phi_ellipse}
&\Phi_r\left( \frac{ q^{-\lambda} \enum^{\inum \theta} + q^\lambda \enum^{-\inum \theta} }{ 2 } \right) = \dn\left( \frac{ \EK(k) }{ \pi } \theta - \inum \EK(r) \lambda, k\right), \lambda \in \mathbb{R}, \theta \in \mathbb{R}, \\
\label{eq:Phi_symmetry}
&\Phi_r(-u) = r / \Phi_r(u), u \in \mathbb{C}, \\
\label{eq:Phi_symmetry_ellipse}
&\Phi_r\left( \frac{ q^{-2n} z + q^{2n} z^{-1} }{ 2 } \right) = (-1)^n \Phi_r\left( \frac{z + z^{-1}}{2} \right), n \in \mathbb{Z}, z \in \mathbb{C} \setminus \{ 0 \}, \\
\label{eq:Phi_ellipse1}
&\Phi_r((q^{-1} \enum^{\inum \theta} + q \enum^{-\inum \theta}) / 2) = \inum \cs(\EK(k) \theta / \pi, k), \theta \in \mathbb{R}, \\
\label{eq:Phi_p1}
&\Phi_r((q^{-\lambda} + q^\lambda) / 2) = \dc(\EK(r) \lambda, r), \lambda \in \mathbb{R}, \\
\label{eq:Phi_m1}
&\Phi_r(- (q^{-\lambda} + q^\lambda) / 2) = r \cd(\EK(r) \lambda, r), \lambda \in \mathbb{R}, \\
\label{eq:Phi_iy}
&\Phi_r\left( \inum (q^{-\lambda} - q^\lambda) / 2 \right) =  \sqrt{r} \enum^{ \inum \left( \am\left(\EK\left( \frac{ 2 \sqrt{r} }{ 1 + r } \right) (\lambda - 1), \frac{ 2 \sqrt{r} }{ 1 + r } \right) + \frac{\pi}{2} \right) }, \lambda \in \mathbb{R}, \\
\label{eq:Phi_ellipse1_2}
&\Phi_r\left( ( q^{-1/2} \enum^{\inum \theta} + q^{1/2} \enum^{-\inum \theta} ) / 2 \right) = \sqrt{ 1 + k \enum^{ \inum 2 \am\left( \EK\left( \frac{ 2 \sqrt{k} }{ 1 + k } \right) \frac{ \theta }{ \pi }, \frac{ 2 \sqrt{k} }{ 1 + k } \right) } }, \theta \in \mathbb{R}, \\
\label{eq:Phi_value}
&\Phi_r(-1) = r, \Phi_r(0) = \sqrt{r}, \Phi_r(1) = 1, \\
\label{eq:Phi_ineq}
&\Phi_r^{(n)}(u) > 0, n \in \intge{0}, - (q^{-1} + q) / 2 < u < ( q^{-1} + q ) / 2, \\
\label{eq:Phi_ana}
&1 < \rhoana[\Phi_r] = q^{-1} = \enum^{ \pi \frac{ \EK(r) }{ \EK(k) } } < \infty, \\
\label{eq:Phi_chebyshev}
&\Phi_r(u) = \frac{\pi}{2 \EK(k)} + \frac{\pi}{\EK(k)} \sum_{n = 1}^\infty \frac{2}{q^{-n} + q^n} T_n(u), u \in \ellipse_{q^{-1}}, \\
\label{eq:Phi_qn2}
&\Phi_r(u) = \sqrt{r} \frac{ 1 + 2 \sum_{n = 1}^\infty q^{n^2} T_n(u) }{ 1 + 2 \sum_{n = 1}^\infty q^{n^2} T_n(-u) }, u \in \mathbb{C}, \\
\label{eq:dPhi2}
&(1 - u^2) \Phi_r'(u)^2 = \frac{ \EK(k)^2 }{ \pi^2 } (1 - \Phi_r(u)^2) (\Phi_r(u)^2 - r^2), u \in \mathbb{C}, \\
\label{eq:dPhi_prod}
&\Phi_r'(u) = k \sqrt{r q} \frac{ 2 \EK(k) }{ \pi } \frac{ \prod_{n = 1}^\infty (1 - 2 q^{2n} u + q^{4n}) (1 + 2 q^{2n} u + q^{4n}) }{ \prod_{n = 1}^\infty (1 - 2 q^{2n - 1} u + q^{4n - 2})^2 }, u \in \mathbb{C}, \\
\label{eq:dPhi_V}
&\Phi_r'(u) = k r^\frac{3}{4} \sqrt{q} \left( \frac{ 2 \EK(k) }{ \pi } \right)^\frac{3}{2} \frac{ \sum_{n = 0}^\infty q^{2 n^2 + 2n} V_n(1 - 2 u^2) }{( 1 + 2 \sum_{n = 1}^\infty q^{n^2} T_n(-u) )^2 }, u \in \mathbb{C}, \\
\label{eq:dPhi_value}
&\Phi_r'(-1) = \frac{ r k^2 \EK(k)^2 }{ \pi^2 }, \Phi_r'(0) = \frac{ \sqrt{r} (1 - r) \EK(k)}{\pi}, \Phi_r'(1) = \frac{ k^2 \EK(k)^2 }{ \pi^2 }.
\end{align}
\end{lemma}
\begin{proof}
Although $\arccos(u)$ in \eqref{eq:Phi} is defined as the multivalued inverse cosine function, $\Phi_r(u)$ is the single-valued function by Lemma \ref{lem:Phi}. Therefore, it is sufficient to select a value of $\arccos(u)$ to evaluate $\Phi_r(u)$. By setting $u = \cos z$ and selecting the value of $\arccos(u)$ as $z$, we have \eqref{eq:Phi_cos}. By substituting $z = \theta + \inum \log(q^\lambda)$ for \eqref{eq:Phi_cos} and using $q = \enum^{ - \pi \frac{ \EK(r) }{ \EK(k) } }$ and $\cos(\theta + \inum \log(q^\lambda)) = ( q^{-\lambda} \enum^{ \inum \theta } + q^\lambda \enum^{ - \inum \theta } ) / 2$, we have \eqref{eq:Phi_ellipse}. By \eqref{eq:Phi_cos} and \eqref{eq:dnK}, we have $\Phi_r(- \cos z) = \Phi_r(\cos(z + \pi)) = \dn(\EK(k) z / \pi + \EK(k) , k) = r / \Phi_r(\cos z)$, which leads to \eqref{eq:Phi_symmetry}. \eqref{eq:Phi_symmetry} also follows from \eqref{eq:Phi_prod}. By substituting $\lambda = X + 2n, X \in \mathbb{R}, n \in \mathbb{Z}$, and $z = q^{-X} \enum^{\inum \theta}$ for \eqref{eq:Phi_ellipse} and using \eqref{eq:dni2Kp}, we have \eqref{eq:Phi_symmetry_ellipse}. By substituting $\lambda = 1$ for \eqref{eq:Phi_ellipse} and using $\dn(-z, k) = \dn(z, k)$, $\dn(z + \inum \EK(r), k) = - \inum \cs(z, k)$ \cite[Table 22.4.3]{DLMF}, and $\cs(-z, k) = -\cs(z, k)$, we have \eqref{eq:Phi_ellipse1}. By substituting $\theta = 0$ for \eqref{eq:Phi_ellipse} and using $\dn(-z, k) = \dn(z, k)$ and Jacobi's imaginary transformation $\dn(\inum z, k) = \dc(z, r)$ \cite[Table 22.6.1]{DLMF}, we have \eqref{eq:Phi_p1}. \eqref{eq:Phi_m1} follows from \eqref{eq:Phi_symmetry} and \eqref{eq:Phi_p1}. By substituting $\theta = \pi/2$ for \eqref{eq:Phi_ellipse} and using \eqref{eq:dn_K_2_iKrl} and $-\am(x, \frac{2 \sqrt{r}}{1 + r}) = \am(-x, \frac{2 \sqrt{r}}{1 + r})$ for $x \in \mathbb{R}$, we have \eqref{eq:Phi_iy}. By substituting $\lambda = 1/2$ for \eqref{eq:Phi_ellipse} and using \eqref{eq:dnxmiKr_2}, we have \eqref{eq:Phi_ellipse1_2}. By substituting $z = \pi, \pi/2, 0$ for \eqref{eq:Phi_cos} and using \eqref{eq:dnvalue}, we have \eqref{eq:Phi_value}.

Let $m \in \intge{1}$, $n \in \intge{0}$, and $u_m \coloneqq ( q^{ - 2 m + 1 } + q^{ 2 m - 1 } ) / 2$. By Lemma \ref{lem:Phi}, $\Phi_r(u)$ is real and infinitely differentiable for $-\infty < u < u_1$. By introducing $F_m(u) \coloneqq \frac{ 1 + u / u_m }{ 1 - u / u_m }$, \eqref{eq:Phi_prod} can be expressed as $\Phi_r(u) = \sqrt{r} \prod_{m = 1}^\infty F_m(u), u \in \mathbb{C}$. By noting $0 < F_m^{(n)}(u) < \infty$ for $- u_1 < u < u_1$, we have \eqref{eq:Phi_ineq}.

By Lemma \ref{lem:Phi}, $\Phi_r(u)$ is analytic on $\ellipse_{q^{-1}}$ and is not analytic at $u = (q^{-1} + q) / 2 \in \bd{\ellipse_{q^{-1}}}$, which leads to $\rhoana[\Phi_r] = q^{-1}$ in \eqref{eq:Phi_ana}. The inequalities in \eqref{eq:Phi_ana} follow from $0 < q < 1$ for $0 < r < 1$.

\eqref{eq:Phi_chebyshev} follows from Jacobi's Fourier series representation of dn function \cite[Eq. 39.25]{Jacobi1829} \cite[Eq. 22.11.3]{DLMF}
\begin{equation}
\label{eq:dn_fourier}
\dn(\EK(k) z / \pi, k) = \frac{\pi}{2 \EK(k)} + \frac{\pi}{\EK(k)} \sum_{n = 1}^\infty \frac{2 \cos(n z)}{q^{-n} + q^n}, z \in \imset_{\log(q^{-1})},
\end{equation}
\eqref{eq:Phi_cos}, and $\cos(n z) = T_n(\cos z)$ for $z \in \mathbb{C}$ and $n \in \intge{0}$. \eqref{eq:Phi_qn2} follows from \eqref{eq:Phi_prod} and \eqref{eq:prod_Tn}.

Let $z \in \mathbb{C}$. By differentiating \eqref{eq:Phi_cos} with respect to $z$ and using \eqref{eq:ddn}, we have
\begin{equation}
\label{eq:dPhicos}
\sin(z) \Phi_r'(\cos z) = \frac{ \EK(k) }{ \pi } k^2 \sn(\EK(k) z / \pi, k) \cn(\EK(k) z / \pi, k).
\end{equation}
By squaring \eqref{eq:dPhicos}, using \eqref{eq:dn2} and \eqref{eq:Phi_cos}, and substituting $u = \cos z$, we have \eqref{eq:dPhi2}.
If $\sin z \ne 0$ (i.e., $z \ne m \pi$, $m \in \mathbb{Z}$), by using \eqref{eq:sncn_prod} and \eqref{eq:dPhicos} and substituting $u = \cos z \in \mathbb{C} \setminus \{-1, 1\}$, we have \eqref{eq:dPhi_prod} for $u \in \mathbb{C} \setminus \{-1, 1\}$. Since $\Phi_r(u)$ is analytic at $u = \pm 1$ by Lemma \ref{lem:Phi}, \eqref{eq:dPhi_prod} also holds at $u = \pm 1$ by the continuity of $\Phi_r'(u)$ at $u = \pm 1$. \eqref{eq:dPhi_V} follows from \eqref{eq:dPhi_prod}, \eqref{eq:prod_Tn}, and \eqref{eq:prod_Vn2}. By substituting $u = \pm 1$ for \eqref{eq:dPhi_prod} and using \eqref{eq:prod_q2nm1} and \eqref{eq:prod_q2n}, we have the values of $\Phi_r'(\pm 1)$ in \eqref{eq:dPhi_value}. By substituting $u = 0$ for \eqref{eq:dPhi2} and using \eqref{eq:Phi_value} and \eqref{eq:Phi_ineq}, we have the value of $\Phi_r'(0)$ in \eqref{eq:dPhi_value}.
\end{proof}
In Section \ref{sec:numerical}, we need $\vt_r'(u)$ as well as $\vt_r(u)$ on $(-1, 1)$ as shown in \eqref{eq:ydDS} and \eqref{eq:ucDS}. Therefore, we described the expressions of $\Phi_r'(u)$ as well as $\Phi_r(u)$ in Lemma \ref{lem:Phi_eq}. $\Phi_r(u)$ and $\Phi_r'(u)$ on $[-1, 1]$ are shown in Figure \ref{fig:Phi}. If an infinitely differentiable function $F(x)$ satisfies $F^{(n)}(x) > 0$ for $x_0 < x < x_1$ and $n \in \intge{0}$, $F(x)$ is called strictly absolutely monotonic in $x_0 < x < x_1$. By \eqref{eq:Phi_ineq}, $\Phi_r(u)$ is strictly absolutely monotonic in $- (q^{-1} + q) / 2 < u < (q^{-1} + q) / 2$. Especially, $\Phi_r(u)$ is positive, strictly increasing, and strictly convex for $-(q^{-1} + q) / 2 < u < (q^{-1} + q) / 2$.
\begin{figure}[htbp]
\centering
\includegraphics{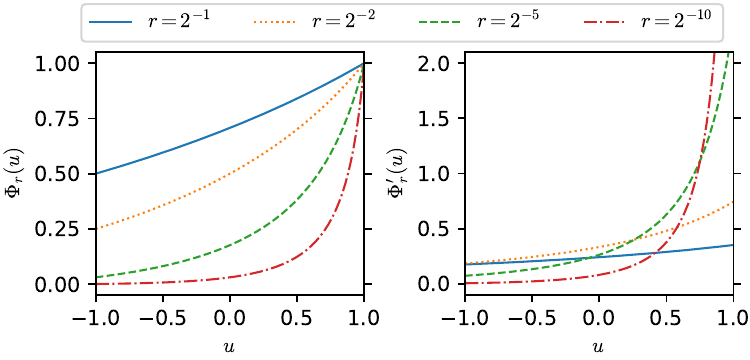}
\caption{$\Phi_r(u)$ (left) and $\Phi_r'(u)$ (right) on $[-1, 1]$. The computational methods are given in Section \ref{sec:Phi_numerical}.
}
\label{fig:Phi}
\end{figure}

Based on the above results, $\Phi_r(u)$ has the following mapping property.
\begin{lemma}
\label{lem:Phi_conformal}
Let us use the notation of Lemma \ref{lem:Phi}. Then $\Phi_r(u)$ is a one-to-one conformal mapping of $\ellipse_{q^{-1}}$ onto $\rhp$ ($u$-plane to $\tau$-plane in Figure \ref{fig:map}).
\end{lemma}
\begin{proof}
Let $u_1 \coloneqq (q^{-1} + q) / 2$. By Lemma \ref{lem:utotau}, $\Phi_r(u)$ is a one-to-one conformal mapping of $\ellipse_{q^{-1}} \setminus ( (- u_1, -1] \cup [1,  u_1) )$ onto $\rhp \setminus ( (0, r] \cup [1, \infty) )$. We note that $\Phi_r(- u_1) = 0$ and $\Phi_r(u_1) = \infty$ by Lemma \ref{lem:Phi} and $\Phi_r(-1) = r$ and $\Phi_r(1) = 1$ by \eqref{eq:Phi_value}. Since $\Phi_r(u)$ is strictly increasing for $- u_1 < u < u_1$ by \eqref{eq:Phi_ineq}, $\Phi_r(u)$ is a one-to-one mapping of $( - u_1, -1]$ and $[1, u_1)$ onto $(0, r]$ and $[1, \infty)$, respectively. Thus $\Phi_r(u)$ is a one-to-one mapping of $\ellipse_{q^{-1}}$ onto $\rhp$. By using \eqref{eq:Phi_ana} and Lemma \ref{lem:dfne0}, $\Phi_r'(u) \ne 0$ on $u \in \ellipse_{q^{-1}}$. Thus $\Phi_r(u)$ is a one-to-one conformal mapping of $\ellipse_{q^{-1}}$ onto $\rhp$.
\end{proof}

Based on these results, we can conclude that $\Phi_r(u)$ defined in \eqref{eq:Phi} gives the function $\Phi_r(u)$ in Lemma \ref{lem:max_hatrho} as follows:
\begin{theorem}
\label{the:Phi}
Let $0 < r < 1$, $\vt_r \in \vtset_r$, and $\rhoana[\vt_r] > 1$. Let us use the notation of Lemma \ref{lem:Phi}. Then the maximum value of $\hat{\rho}[\vt_r]$ is attained iff $\vt_r(u) = \Phi_r(u)$ on $\ellipse_{q^{-1}}$ and $\hat{\rho}[\Phi_r]$ is given as
\begin{equation}
\label{eq:hatrho_Phi}
\hat{\rho}[\Phi_r] = \rhoana[\Phi_r] = q^{-1} = \enum^{ \pi \frac{ \EK(r) }{ \EK(k) } } = \enum^{ \pi \frac{ \EK(r) }{ \EK\left( \sqrt{ 1 - r^2 } \right) } }.
\end{equation}
\end{theorem}
\begin{proof}
By \eqref{eq:Phi_value} and \eqref{eq:Phi_ineq}, we have $\Phi_r \in \vtset_r$. By \eqref{eq:Phi_ana}, we have $1 < \rhoana[\Phi_r] = q^{-1} < \infty$. By $\rhoana[\Phi_r] = q^{-1}$ and Lemma \ref{lem:Phi_conformal}, we have
\begin{equation}
\label{eq:Phi_infre0}
\inf_{u \in \ellipse_{\rhoana[\Phi_r]}} \RE \Phi_r(u) = \inf_{u \in \ellipse_{q^{-1}}} \RE \Phi_r(u) = 0.
\end{equation}
Thus we have $\hat{\rho}[\Phi_r] = \rhoana[\Phi_r] = q^{-1}$. By $\hat{\rho}[\Phi_r] = q^{-1}$ and Lemma \ref{lem:Phi_conformal}, $\Phi_r(u)$ is a one-to-one mapping of $\ellipse_{\hat{\rho}[\Phi_r]}$ onto $\rhp$. Then the statement follows from Lemma \ref{lem:max_hatrho}.
\end{proof}
\begin{corollary}
Let $0 < a < b < \infty$ and $M \in \intge{1}$. Let $W(t)$ be a bounded nondecreasing function with at least $M + 1$ points of increase for $a \le t \le b$. Let $f(x) \coloneqq \int_a^b \enum^{-xt} \del W(t)$ for $0 \le x < \infty$. Then one has
\begin{align}
&0 < \dist\left( f, \eset_M, [0, \infty) \right) < \frac{16}{\pi} \rho^{-2M} f(0), \rho \coloneqq \enum^{\pi \frac{ \EK(a/b) }{ \EK(\sqrt{1 - (a/b)^2}) }}.
\end{align}
\end{corollary}
\begin{proof}
The lower bound follows from Lemma \ref{lem:kammler_zero}. The upper bound follows from Theorem \ref{the:Evt} and Theorem \ref{the:Phi}.
\end{proof}

\begin{table}[htbp]
{\footnotesize
\caption{Values of $\hat{\rho}[\vt_r]$ and $\hat{\rho}[\vt_r]^2$. Expressions are found in \eqref{eq:hatrho_Phi},  \eqref{eq:hatrho_exp}, \eqref{eq:hatrho_P2}, \eqref{eq:hatrho_P1}, and \eqref{eq:hatrho_R01}.
}
\label{tab:hatrho}
\begin{center}
\begin{tabular}{l|cc|cc|cc|cc}
\hline
  &
  \multicolumn{2}{c|}{$\Phi_r$} &
  \multicolumn{2}{c|}{$\vt_{\exp, r}$} &
  \multicolumn{2}{c|}{$\vt_{\pset_2, r}$} &
  \multicolumn{2}{c}{$\vt_{\pset_1, r}$ or $\vt_{\rset_{0,1}, r}$} \\ \cline{2-9}
 $r$ & $\hat{\rho}$  & $\hat{\rho}^2$ & $\hat{\rho}$ & $\hat{\rho}^2$ & $\hat{\rho}$ & $\hat{\rho}^2$ & $\hat{\rho}$ & $\hat{\rho}^2$ \\
\hline
$2^{-1}$ & 11.6556 & 135.853 & 9.17373 & 84.1573 & 8.24175 & 67.9264 & 5.82843 & 33.9706 \\
$2^{-2}$ & 5.99070 & 35.8885 & 4.74319 & 22.4978 & 4.23607 & 17.9443 & 3.00000 & 9.00000 \\
$2^{-3}$ & 4.15994 & 17.3051 & 3.32255 & 11.0393 & 2.94155 & 8.65273 & 2.09384 & 4.38415 \\
$2^{-4}$ & 3.27674 & 10.7370 & 2.64435 & 6.99256 & 2.31718 & 5.36931 & 1.66667 & 2.77778 \\
$2^{-5}$ & 2.76519 & 7.64629 & 2.25617 & 5.09031 & 1.95586 & 3.82538 & 1.42947 & 2.04340 \\
$2^{-6}$ & 2.43498 & 5.92910 & 2.00864 & 4.03462 & 1.72318 & 2.96935 & 1.28571 & 1.65306 \\
$2^{-7}$ & 2.20571 & 4.86514 & 1.83879 & 3.38117 & 1.56245 & 2.44125 & 1.19392 & 1.42544 \\
$2^{-8}$ & 2.03794 & 4.15322 & 1.71588 & 2.94425 & 1.44587 & 2.09054 & 1.13333 & 1.28444 \\
$2^{-9}$ & 1.91022 & 3.64895 & 1.62324 & 2.63492 & 1.35831 & 1.84500 & 1.09248 & 1.19350 \\
$2^{-10}$ & 1.80992 & 3.27582 & 1.55115 & 2.40608 & 1.29083 & 1.66623 & 1.06452 & 1.13319 \\
$2^{-11}$ & 1.72918 & 2.99006 & 1.49359 & 2.23082 & 1.23782 & 1.53220 & 1.04519 & 1.09243 \\
$2^{-12}$ & 1.66284 & 2.76505 & 1.44665 & 2.09279 & 1.19558 & 1.42941 & 1.03175 & 1.06450 \\
$2^{-13}$ & 1.60742 & 2.58378 & 1.40768 & 1.98155 & 1.16155 & 1.34919 & 1.02234 & 1.04519 \\
$2^{-14}$ & 1.56043 & 2.43495 & 1.37484 & 1.89018 & 1.13389 & 1.28570 & 1.01575 & 1.03174 \\
$2^{-15}$ & 1.52012 & 2.31076 & 1.34681 & 1.81390 & 1.11127 & 1.23491 & 1.01111 & 1.02234 \\
$2^{-16}$ & 1.48516 & 2.20570 & 1.32262 & 1.74932 & 1.09266 & 1.19392 & 1.00784 & 1.01575 \\
$2^{-17}$ & 1.45456 & 2.11576 & 1.30154 & 1.69401 & 1.07730 & 1.16059 & 1.00554 & 1.01111 \\
$2^{-18}$ & 1.42757 & 2.03794 & 1.28301 & 1.64612 & 1.06458 & 1.13333 & 1.00391 & 1.00784 \\
$2^{-19}$ & 1.40357 & 1.97001 & 1.26660 & 1.60428 & 1.05401 & 1.11094 & 1.00277 & 1.00554 \\
$2^{-20}$ & 1.38211 & 1.91022 & 1.25197 & 1.56744 & 1.04522 & 1.09248 & 1.00196 & 1.00391 \\
\hline
\end{tabular}
\end{center}
}
\end{table}

\begin{figure}[htbp]
\centering
\includegraphics{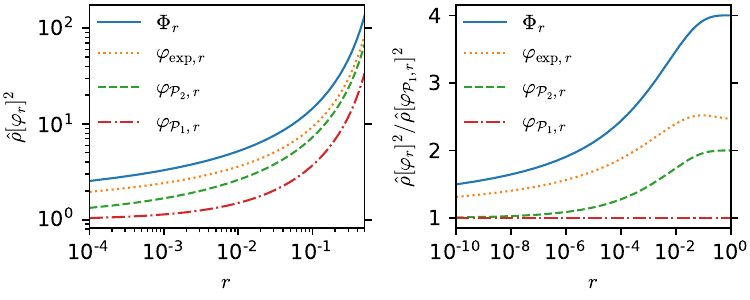}
\caption{$\hat{\rho}[\vt_r]^2$ for $10^{-4} \le r \le 0.5$ (left) and $\hat{\rho}[\vt_r]^2 / \hat{\rho}[\vt_{\pset_1, r}]^2$ for $10^{-10} \le r \le 0.99$ (right). $\vt_r$ is specified in the legend.
}
\label{fig:hatrho2}
\end{figure}
Some values of $\hat{\rho}[\vt_r]^2$ are shown in Table \ref{tab:hatrho} and Figure \ref{fig:hatrho2}. We can show
\begin{equation}
\label{eq:lim_hatrho2}
\lim_{r \uparrow 1} \frac{ \hat{\rho}[\vt_{\pset_2, r}]^2 }{ \hat{\rho}[\vt_{\pset_1, r}]^2 } = 2, \lim_{r \uparrow 1} \frac{ \hat{\rho}[\vt_{\exp, r}]^2 }{ \hat{\rho}[\vt_{\pset_1, r}]^2 } = \frac{\pi^2}{4} \approx 2.4674, \lim_{r \uparrow 1} \frac{ \hat{\rho}[\Phi_r]^2 }{ \hat{\rho}[\vt_{\pset_1, r}]^2 } = 4,
\end{equation}
where the second equality follows from $\lim\limits_{r \uparrow 1} \frac{ 1 - \sqrt{r} }{ \log(1/r) } = \frac{1}{2}$. The last equality and some properties of $\frac{ \hat{\rho}[\Phi_r] }{ \hat{\rho}[\vt_{\pset_1, r}] } = \frac{ 1 - \sqrt{r} }{ 1 + \sqrt{r} } \enum^{ \pi \frac{ \EK(r) }{ \EK(k) } }$ can be shown as follows:
\begin{corollary}
Let $0 < r < 1$ and $k \coloneqq \sqrt{1 - r^2}$. Then one has
\begin{align}
\label{eq:dhatrho_Phi_P1_ineq}
&\frac{ \del }{ \del r } \left( \frac{ 1 - \sqrt{r} }{ 1 + \sqrt{r} } \enum^{ \pi \frac{ \EK(r) }{ \EK(k) } } \right) > 0, \\
\label{eq:dhatrho_Phi_P1_lim}
&\lim_{ r \downarrow 0 } \left( \frac{ 1 - \sqrt{r} }{ 1 + \sqrt{r} } \enum^{ \pi \frac{ \EK(r) }{ \EK(k) } } \right) = 1, \lim_{ r \uparrow 1 } \left( \frac{ 1 - \sqrt{r} }{ 1 + \sqrt{r} } \enum^{ \pi \frac{ \EK(r) }{ \EK(k) } } \right) = 2, \\
\label{eq:hatrho_Phi_P1_ineq}
&1 < \frac{ 1 - \sqrt{r} }{ 1 + \sqrt{r} } \enum^{ \pi \frac{ \EK(r) }{ \EK(k) } } < 2.
\end{align}
\end{corollary}
\begin{proof}
By using $\frac{ \del }{ \del r} \frac{ \EK(r) }{ \EK(k) } =\frac{ \pi }{ 2 r k^2 \EK(k)^2 }$ \cite[Sec.~32]{Jacobi1829}, we have
\begin{equation}
\label{eq:dhatrho_Phi_P1}
\frac{ \del }{ \del r } \left( \frac{ 1 - \sqrt{r} }{ 1 + \sqrt{r} } \enum^{ \pi \frac{ \EK(r) }{ \EK(k) } } \right) = \frac{ \pi^2 - 2 (1 + r) \sqrt{r} \EK(k)^2 }{ 2 (1 + r) (1 + \sqrt{r})^2 r \EK(k)^2 } \enum^{ \pi \frac{ \EK(r) }{ \EK(k) } }.
\end{equation}
By using $\frac{ 2 \EK(k) }{ \pi } = \sqrt{ \frac{ \prod_{n = 1}^\infty r_n }{ r } }$, $r_0 \coloneqq r$, $r_{n+1} \coloneqq \frac{ 2 \sqrt{r_n} }{ 1 + r_n }$, $n \in \intge{0}$ \cite[Eq. 38.1]{Jacobi1829}, we have $\left( \frac{ 2 \EK(k) }{ \pi } \right)^2 < \frac{ r_1 }{ r } = \frac{1}{r} \frac{ 2 \sqrt{r} }{ 1 + r }$. With this inequality and \eqref{eq:dhatrho_Phi_P1}, we have \eqref{eq:dhatrho_Phi_P1_ineq}.
 \eqref{eq:dhatrho_Phi_P1_lim} follows from $\lim\limits_{r \downarrow 0} \enum^{ -\pi \frac{ \EK(r) }{ \EK(k) } } = 1$ and $\lim\limits_{r \uparrow 1} \frac{ \enum^{ -\pi \frac{ \EK(r) }{ \EK(k) } } }{ k^2 } = \frac{1}{16}$. \eqref{eq:hatrho_Phi_P1_ineq} follows from \eqref{eq:dhatrho_Phi_P1_ineq} and \eqref{eq:dhatrho_Phi_P1_lim}.
\end{proof}

We also noticed the direct connections among $\hat{\rho}[\vt_{\pset_1, r}]$, $\hat{\rho}[\vt_{\pset_2, r}]$, $\hat{\rho}[\vt_{\rset_{0, 1}, r}]$, and $\hat{\rho}[\Phi_r]$ as follows:
\begin{corollary}
Let $0 < r < 1$. Expressions of $\hat{\rho}[\vt_{\pset_1, r}]$, $\hat{\rho}[\vt_{\pset_2, r}]$, $\hat{\rho}[\vt_{\rset_{0, 1}, r}]$, and $\hat{\rho}[\Phi_r]$ are given in \eqref{eq:hatrho_P1}, \eqref{eq:hatrho_P2}, \eqref{eq:hatrho_R01}, and \eqref{eq:hatrho_Phi}, respectively. Let $k(x)$ be the modulus corresponding to the nome $0 < x < 1$. Then one has
\begin{align}
\label{eq:hatrho_P1Phi}
&\hat{\rho}[\vt_{\pset_1, r}]^{-2} = k(\hat{\rho}[\Phi_r]^{-4}), \\
\label{eq:hatrho_P2Phi}
&\hat{\rho}[\vt_{\pset_2, r}]^{-4} = k(\hat{\rho}[\Phi_r]^{-8}) = \hat{\rho}\left[ \vt_{\pset_1, \frac{ 2 \sqrt{r} }{ 1 + r }} \right]^{-2}, \\
\label{eq:hatrho_R01_P1}
&\hat{\rho}[\vt_{\rset_{0, 1}, r}] = \hat{\rho}[\vt_{\pset_1, r}].
\end{align}
\end{corollary}
\begin{proof}
Let $q \coloneqq \enum^{- \pi \frac{ \EK(r) }{ \EK(\sqrt{1 - r^2}) }}$. By noting $k(q) = \sqrt{1 - r^2}$ and using Theorem \ref{the:q2}, we have
\begin{align}
&k(q^2) = \frac{ 1 - \sqrt{1 - k(q)^2} }{ 1 + \sqrt{1 - k(q)^2} } = \frac{ 1 - r }{ 1 + r }, \\
\label{eq:kq4}
&k(q^4) = \frac{ 1 - \sqrt{1 - k(q^2)^2} }{ 1 + \sqrt{1 - k(q^2)^2} } = \frac{ (1 - \sqrt{r})^2 }{ (1 + \sqrt{r})^2 }, \\
\label{eq:kq8}
&k(q^8) = \frac{ 1 - \sqrt{1 - k(q^4)^2} }{ 1 + \sqrt{1 - k(q^4)^2} } = \frac{ (\sqrt{1 + r} - \sqrt{ 2 \sqrt{r} })^2 }{ (\sqrt{1 + r} + \sqrt{ 2 \sqrt{r} })^2 } = \frac{ (1 - \sqrt{r})^4 }{ (\sqrt{1 + r} + \sqrt{ 2 \sqrt{r} })^4 }.
\end{align}
By substituting \eqref{eq:hatrho_Phi} and \eqref{eq:hatrho_P1} for \eqref{eq:kq4}, we have \eqref{eq:hatrho_P1Phi}. By substituting \eqref{eq:hatrho_Phi} and \eqref{eq:hatrho_P2} for \eqref{eq:kq8}, we have the first equality in \eqref{eq:hatrho_P2Phi}. By using \eqref{eq:hatrho_Phi} and Theorem \ref{the:q2}, we have
\begin{equation}
\label{eq:hatrho_Phi2}
\hat{\rho}\left[\Phi_{\frac{ 2 \sqrt{r} }{ 1 + r }}\right] = \hat{\rho}[\Phi_r]^2.
\end{equation}
By replacing $r$ with $\frac{ 2 \sqrt{r} }{ 1 + r }$ in \eqref{eq:hatrho_P1Phi} and using \eqref{eq:hatrho_Phi2}, we have the second equality in \eqref{eq:hatrho_P2Phi}. \eqref{eq:hatrho_R01_P1} follows from \eqref{eq:hatrho_P1} and \eqref{eq:hatrho_R01}.
\end{proof}

\subsection{ Basis functions }
\label{sec:Phi_basis}
We consider the basis functions associated with $\Phi_r(u)$. We have some concrete expressions as follows:
\begin{lemma}
\label{lem:Phi_basis}
Let us use the notation of Lemma \ref{lem:Phi}. Let us introduce
\begin{equation}
\label{eq:basis_Phi_cos}
\basis{\Phi}_{r, n}(z) \coloneqq \frac{1}{\pi} \int_0^\pi \enum^{- z \Phi_r(\cos\theta)} \cos(n \theta) \del \theta = \frac{1}{\pi} \int_0^\pi \enum^{- z \dn(\EK(k) \theta / \pi, k) } \cos(n \theta) \del \theta, z \in \mathbb{C}, n \in \intge{0}.
\end{equation}
Let $\arcdn(\tau, k) \coloneqq \int_\tau^1 \frac{ \del t }{ \sqrt{ (1 - t^2) ( t^2 - r^2) } }$, $r \le \tau \le 1$, which satisfies $\arcdn(r, k) = \EK(k)$, $\arcdn(1, k) = 0$, and $\dn(\arcdn(\tau, k), k) = \tau$ \cite[Sec.~22.15]{DLMF}. Let $\EE(k)$ be the complete elliptic integral of the second kind with the modulus $k$. Then one has
\begin{align}
\label{eq:basis_Phi_exp}
&\enum^{ - z \dn( \EK(k) \theta / \pi, k ) } = \basis{\Phi}_{r, 0}(z) + 2 \sum_{n = 1}^\infty \basis{\Phi}_{r, n}(z) \cos(n \theta), \theta \in \imset_{\log(q^{-1})}, z \in \mathbb{C}, \\
\label{eq:basis_Phi_ineq_hatrho}
&\lvert \basis{\Phi}_{r, n}(x) \rvert \le q^n, 0 < x < \infty, n \in \intge{0}, \\
\label{eq:W0_Phi}
&W_{\Phi_r, 0}(\tau) = \frac{ \EK(k) - \arcdn(\tau, k) }{ \EK(k) } = \int_r^\tau \frac{ \del t }{ \EK(k) \sqrt{ (1 - t^2) ( t^2 - r^2) } }, r \le \tau \le 1, \\
\label{eq:basis_Phi_dt}
&\basis{\Phi}_{r, n}(z) = \int_r^1 \enum^{ -z \tau } \frac{ \cos( n \pi \arcdn(\tau, k) / \EK(k) ) }{ \EK(k) \sqrt{ (1 - \tau^2) ( \tau^2 - r^2 ) } } \del \tau, z \in \mathbb{C}, n \in \intge{0}, \\
\label{eq:basis_Phi_sin}
&\basis{\Phi}_{r, n}(z) = - \frac{z}{n \pi} \int_r^1 \enum^{ -z \tau } \sin( n \pi \arcdn(\tau, k) / \EK(k) ) \del \tau, z \in \mathbb{C}, n \in \intge{1}, \\
\label{eq:basis_Phi_deriv10}
&\basis{\Phi}_{r, n}'(0) = - \frac{ \pi }{ \EK(k) } \frac{ 1 }{ q^{-n} + q^n }, n \in \intge{0}, \\
\label{eq:basis_Phi_deriv20}
&\basis{\Phi}_{r, 0}''(0) = \frac{ \EE(k) }{ \EK(k) }, \basis{\Phi}_{r, n}''(0) = \frac{ \pi^2 }{ \EK(k)^2 } \frac{ n }{ q^{-n} - q^n }, n \in \intge{1}, \\
\label{eq:V_Phi}
&V_{\Phi_r} = (\EK(k) \EE(k) - (\pi / 2)^2) / \EK(k)^2.
\end{align}
\end{lemma}
\begin{proof}
The second equality in \eqref{eq:basis_Phi_cos} follows from \eqref{eq:Phi_cos}. \eqref{eq:basis_Phi_exp} follows from \eqref{eq:exp_vt_fourier}, \eqref{eq:Phi_cos}, and \eqref{eq:hatrho_Phi}. \eqref{eq:basis_Phi_ineq_hatrho} follows from \eqref{eq:basis_ineq_hatrho} and \eqref{eq:hatrho_Phi}. Let $\arccos(u)$ be the principal value of the inverse cosine function for $-1 \le u \le 1$. By setting $u = \Phi_r^{-1}(\tau)$ for $r \le \tau \le 1$ , we have $\tau = \Phi_r(u) = \dn(\EK(k) \arccos(u) / \pi, k)$ for $-1 \le u \le 1$ by \eqref{eq:Phi}. Then we have $\arccos(\Phi_r^{-1}(\tau)) = \pi \arcdn(\tau, k) / \EK(k)$ for $r \le \tau \le 1$. With this equality and \eqref{eq:basis_W0}, we have \eqref{eq:W0_Phi}. \eqref{eq:basis_Phi_dt} follows from \eqref{eq:W0_Phi}, \eqref{eq:basis_Wge1}, and \eqref{eq:basis_fl}. \eqref{eq:basis_Phi_sin} follows from \eqref{eq:basis_sin} and \eqref{eq:W0_Phi}. 

\eqref{eq:basis_Phi_deriv10} follows from $\basis{\Phi}_{r, n}'(0) = - \frac{1}{\pi} \int_0^\pi \dn(\EK(k) \theta / \pi, k) \cos(n \theta) \del \theta$ and \eqref{eq:dn_fourier}.

\eqref{eq:basis_Phi_deriv20} follows from $\basis{\Phi}_{r, n}''(0) = \frac{1}{\pi} \int_0^\pi \dn(\EK(k) \theta / \pi, k)^2 \cos(n \theta) \del \theta$ and Jacobi's Fourier series representation of $\dn^2$ function \cite[Eq.~42.4]{Jacobi1829}\cite[Eq. 22.11.13]{DLMF}
\begin{equation}
\label{eq:dn2_fourier}
\dn(\EK(k) \theta / \pi, k)^2 = \frac{ \EE(k) }{ \EK(k) } + \frac{ 2 \pi^2 }{ \EK(k)^2 } \sum_{n = 1}^\infty \frac{ n \cos(n \theta) }{ q^{-n} - q^n }, \theta \in \imset_{\log(q^{-1})}.
\end{equation}
\eqref{eq:V_Phi} follows from \eqref{eq:V_vt_eq}, \eqref{eq:basis_Phi_deriv10}, and \eqref{eq:basis_Phi_deriv20}.
\end{proof}
We can show that the basis functions $\basis{\Phi}_{r, n}(z)$ are eigenfunctions of a fourth order differential operator as follows:
\begin{theorem}
\label{the:L}
Let $0 < r < 1$ and $z \in \mathbb{C}$. Let $\mathcal{D}_r$ be a fourth order differential operator
\begin{equation}
\label{eq:L}
\mathcal{D}_r \coloneqq z^2 \frac{ \del^4 }{ \del z^4 } + 2 z \frac{ \del^3 }{ \del z^3 } - (1 + r^2) z^2 \frac{ \del^2 }{ \del z^2 } - (1 + r^2) z \frac{ \del }{ \del z } + r^2 z^2 = z \left( \frac{ \del^2 }{ \del z^2 } z \frac{ \del^2 }{ \del z^2 } - \frac{ \del }{ \del z } (1 + r^2) z \frac{ \del }{ \del z } + r^2 z \right).
\end{equation}
Let us use the notation of Lemma \ref{lem:Phi_basis}. Then one has
\begin{equation}
\label{eq:L_eigen}
\mathcal{D}_r \basis{\Phi}_{r, n}(z) = \frac{ n^2 \pi^2 }{ \EK(k)^2 } \basis{\Phi}_{r, n}(z), n \in \intge{0}.
\end{equation}
\end{theorem}
\begin{proof}
Let $D(\theta) \coloneqq \dn(\EK(k) \theta / \pi, k)$. By \eqref{eq:dnvalue}, we have $D'(0) = D'(\pi) = 0$. Then by applying integration by parts to \eqref{eq:basis_Phi_cos}, we have
\begin{equation}
\label{eq:basis_Phi_cos2}
\basis{\Phi}_{r, n}(z)
= \frac{z}{n \pi} \int_0^\pi D'(\theta) \enum^{ - z D(\theta) } \sin(n \theta) \del \theta = \frac{z}{n^2 \pi} \int_0^\pi \left( D''(\theta) - z D'(\theta)^2 \right) \enum^{ - z D(\theta) } \cos(n \theta) \del \theta, n \in \intge{1}.
\end{equation}
By using the differential equations satisfying $\dn$ function~\cite[Sec.~22.13]{DLMF}
\begin{subequations}
\label{eq:dn_de}
\begin{align}
&D'(\theta)^2 = (\EK(k) / \pi)^2 ( 1 - D(\theta)^2 ) ( D(\theta)^2 - r^2 ), \theta \in \mathbb{R}, \\
&D''(\theta) = (\EK(k) / \pi)^2 \left( (1 + r^2) D(\theta) - 2 D(\theta)^3 \right), \theta \in \mathbb{R},
\end{align}
\end{subequations}
\eqref{eq:basis_Phi_cos2} is expressed as
\begin{align}
\label{eq:basis_Phi_cos2_dn}
&\basis{\Phi}_{r, n}(z) = \frac{ \EK(k)^2 z }{n^2 \pi^3} \int_0^\pi \left( (1 + r^2) D(\theta) - 2 D(\theta)^3 - z ( 1 - D(\theta)^2 ) ( D(\theta)^2 - r^2 ) \right) \enum^{ - z D(\theta) } \cos(n \theta) \del \theta, n \in \intge{1}.
\end{align}
By using \eqref{eq:basis_deriv_cos} and \eqref{eq:basis_Phi_cos2_dn}, we have \eqref{eq:L_eigen} for $n \in \intge{1}$.
By using \eqref{eq:basis_deriv_cos} and \eqref{eq:dn_de}, we have
\begin{equation}
\mathcal{D}_r \basis{\Phi}_{r, 0}(z) = \frac{ \pi }{ \EK(k)^2 } \int_0^\pi \left( - z^2 D'(\theta)^2 + z D''(\theta) \right) \enum^{ - z D(\theta) } \del \theta.
\end{equation}
With this equality and $\int_0^\pi D''(\theta) \enum^{ - z D(\theta) } \del \theta = z \int_0^\pi D'(\theta)^2 \enum^{ - z D(\theta) } \del \theta$, we have $\mathcal{D}_r \basis{\Phi}_{r, 0}(z) = 0$.
\end{proof}

It is known that a differential operator of the form
\begin{equation}
\frac{1}{w(x)} \left( \frac{ \del^2 }{ \del x^2 } p_2(x) \frac{ \del^2 }{ \del x^2 } - \frac{ \del }{ \del x } p_1(x) \frac{ \del }{ \del x }
+ p_0(x) \right)
\end{equation}
can be symmetric with respect to the inner product with the weight function $w(x)$ by imposing some conditions on $w, p_0, p_1, p_2,$ and the boundary conditions \cite[Sec. 1]{Weidmann1987}. The operator \eqref{eq:L} has this form and we have $w(x) = 1 / x$. For $n \in \intge{1}$, we can show the orthogonality of $\basis{\Phi}_{r, n}(x)$ with the weight function $1 / x$ as follows:
\begin{theorem}
\label{the:basis_Phi_orthogonal}
Let us use the notation of Lemma \ref{lem:Phi_basis}. Then one has
\begin{equation}
\label{eq:basis_Phi_orthogonal}
\int_0^\infty \basis{\Phi}_{r, m}(x) \basis{\Phi}_{r, n}(x) \frac{1}{x} \del x =
\begin{cases}
0, & m \neq n, \\
\frac{1}{n} \frac{ 1 }{ q^{-2n} - q^{2n} }, & m = n,
\end{cases}
m \in \intge{1}, n \in \intge{1}.
\end{equation}
\end{theorem}
\begin{proof}
Let $\langle f, g \rangle \coloneqq \int_0^\infty f(x) g(x) \frac{ 1 }{ x } \del x$. Let $m \in \intge{1}$, $n \in \intge{1}$, $F \coloneqq \basis{\Phi}_{r, m}$, and $G \coloneqq \basis{\Phi}_{r, n}$. By repeating integration by parts, we have
\begin{align}
\label{eq:L_green}
&\langle \mathcal{D}_r F, G \rangle - \langle F, \mathcal{D}_r G \rangle = \left[ x ( F'''(x) G(x) - F(x) G'''(x) ) \right. \\ \nonumber
&\qquad \left. + F''(x) G(x) - F(x) G''(x) + x ( F'(x) G''(x) - F''(x) G'(x) ) + (1 + r^2) x ( F(x) G'(x) - F'(x) G(x) ) \right]_0^\infty.
\end{align}
By \eqref{eq:basis_0_dWn}, we have $\basis{\Phi}_{r, n}(0) = 0$. By \eqref{eq:basis_deriv_ineq}, we have $\lvert \basis{\Phi}_{r, n}^{(l)}(0) \rvert \le 2/\pi$, $\lim\limits_{x \to \infty} \basis{\Phi}_{r, n}^{(l)}(x) = 0$, and $\lim\limits_{x \to \infty} x \basis{\Phi}_{r, n}^{(l)}(x) = 0$ for  $l \in \intge{0}$. Thus the right-hand side of \eqref{eq:L_green} becomes 0 and the equality $\langle \mathcal{D}_r \basis{\Phi}_{r, m}, \basis{\Phi}_{r, n} \rangle = \langle \basis{\Phi}_{r, m}, \mathcal{D}_r \basis{\Phi}_{r, n} \rangle$ holds. By substituting \eqref{eq:L_eigen} for the equality, we have $(m^2 - n^2) \langle \basis{\Phi}_{r, m}, \basis{\Phi}_{r, n} \rangle = 0$, which indicates $\langle \basis{\Phi}_{r, m}, \basis{\Phi}_{r, n} \rangle = 0$ for $m \ne n$. By substituting $\theta = 0$ for \eqref{eq:basis_Phi_exp} followed by operating $\mathcal{D}_r$ and using \eqref{eq:L_eigen}, we have
\begin{equation}
- k^2 z \enum^{- z} = 2 \sum_{m = 1}^\infty \frac{ m^2 \pi^2 }{ \EK(k)^2 } \basis{\Phi}_{r, m}(z), z \in \mathbb{C}.
\end{equation}
With this equality, $\langle \basis{\Phi}_{r, m}, \basis{\Phi}_{r, n} \rangle = 0$ for $m \ne n$, and \eqref{eq:basis_Phi_cos}, we have
\begin{equation}
\label{eq:basis_Phi_ip}
\langle \basis{\Phi}_{r, n}, \basis{\Phi}_{r, n} \rangle = - \frac{ k^2 \EK(k)^2 }{ 2 n^2 \pi^3 } \int_0^\pi \frac{ \cos(n \theta) }{ 1 + \dn(\EK(k) \theta / \pi, k) } \del \theta, n \in \intge{1}.
\end{equation}
By substituting \eqref{eq:inv_1_dn_fourier} for \eqref{eq:basis_Phi_ip}, we have $\langle \basis{\Phi}_{r, n}, \basis{\Phi}_{r, n} \rangle = \frac{1}{n} \frac{ 1 }{ q^{-2n} - q^{2n} }$.
\end{proof}
\begin{corollary}
\label{cor:basis_Phi_integral_eq}
Let us use the notation of Lemma \ref{lem:Phi_basis}. For $z \in \mathbb{C}$ and $n \in \intge{1}$, one has
\begin{equation}
\label{eq:basis_Phi_integral_eq}
\basis{\Phi}_{r, n}(z) = n \frac{ q^{-2n} - q^{2n} }{ 2 } \int_0^\infty \left( \basis{\Phi}_{r, 0}(z + y) - \basis{\Phi}_{r, 0}(z)  \basis{\Phi}_{r, 0}(y) \right) \basis{\Phi}_{r, n}(y) \frac{1}{y} \del y.
\end{equation}
\end{corollary}
\begin{proof}
The statement follows from \eqref{eq:basis_add} for $n = 0$ and \eqref{eq:basis_Phi_orthogonal}.
\end{proof}

Kellogg showed that if real continuous functions on $[0, 1]$ form an orthonormal system on $[0, 1]$ and Chebyshev system on $(0, 1)$, the zeros of functions satisfy the following property called the interlacing property of zeros:
\begin{theorem}[Kellogg \cite{Kellogg1916}]
\label{the:Kellogg}
Let $\phi_i(u)$ satisfy
\begin{subequations}
\label{eq:kellogg_condition}
\begin{align}
\label{eq:kellogg_continuous}
&\text{$\phi_i(u)$ is a real continuous function on $[0, 1]$, $i \in \intge{0}$,} \\
\label{eq:kellogg_orthogonal}
&\int_0^1 \phi_i(u) \phi_j(u) \del u =
\begin{cases}
1, & i = j, \\
0, & i \ne j,
\end{cases}
i \in \intge{0}, j \in \intge{0}, \\
\label{eq:kellogg_chebyshev}
&\det( \phi_j(u_i) )_{i, j = 0}^n > 0, 0 < u_0 < \cdots < u_n < 1, n \in \intge{0}.
\end{align}
\end{subequations}
For $i \in \intge{0}$, $\phi_i(u)$ vanishes exactly $i$ times on $(0, 1)$ and changes sign at each zero. For $i \in \intge{1}$, $\phi_i(u)$ changes sign between two successive zeros of $\phi_{i + 1}(u)$ on $(0, 1)$.
\end{theorem}
Based on Kellogg's theorem, we can show the interlacing property of zeros of $\basis{\Phi}_{r, n}(x)$, $n \in \intge{1}$, on $(0, \infty)$ (Figure \ref{fig:Phi_basis}) as follows:
\begin{theorem}
\label{the:basis_Phi_interlacing}
Let us use the notation of Lemma \ref{lem:Phi_basis}. For $n \in \intge{1}$, $\basis{\Phi}_{r, n}(x)$ vanishes exactly $n - 1$ times on $(0, \infty)$ and changes sign at each zero. For $n \in \intge{2}$, $\basis{\Phi}_{r, n}(x)$ changes sign between two successive zeros of $\basis{\Phi}_{r, n + 1}(x)$ on $(0, \infty)$.
\end{theorem}
\begin{proof}
We introduce
\begin{align}
Z_{r, n} &\coloneqq \sqrt{ \int_0^\infty \basis{\Phi}_{r, n}(x)^2 \frac{1}{x} \del x } = \frac{1}{\sqrt{n}} \frac{1}{ \sqrt{ q^{-2n} - q^{2n} } } > 0, n \in \intge{1}, \\
\label{eq:phi_dn}
\phi_i(u) &\coloneqq
\begin{cases}
- \frac{ \basis{\Phi}_{r, i + 1}( \log(1 / u) / r ) }{ Z_{r, i+1} \sqrt{ u \log(1 / u) } }, & 0 < u < 1, \\
0, & u \in \{0, 1\}, \\
\end{cases}
i \in \intge{0}.
\end{align}
Since $\basis{\Phi}_{r,n}(x)$, $n \in \intge{1}$, is real continuous for $x \in \mathbb{R}$ by \eqref{eq:basis_entire}, $\phi_i(u), i \in \intge{0},$ is real continuous for $0 < u < 1$. By using \eqref{eq:basis_ineq1} and \eqref{eq:phi_dn}, the inequality
\begin{equation}
\lvert \phi_i(u) \rvert \le \frac{ 1 - r }{ Z_{r, i + 1} (i + 1) \pi r } \sqrt{ u \log(1 / u) }, 0 < u < 1,
i \in \intge{0},
\end{equation}
holds, which leads to $\lim\limits_{ u \downarrow 0 } \phi_i(u) = \lim\limits_{ u \uparrow 1 } \phi_i(u) = 0$. Thus $\phi_i(u)$ satisfies \eqref{eq:kellogg_continuous}. By applying the change of variable $u = \enum^{-r x}$ to \eqref{eq:basis_Phi_orthogonal}, \eqref{eq:phi_dn} satisfies \eqref{eq:kellogg_orthogonal}.

By introducing $x_i \coloneqq \log(1 / u_{n + 1 - i}) / r$, $i = 1, \ldots, n + 1$, for $0 < u_0 < \cdots < u_n < 1$ and $n \in \intge{0}$, we have $0 < x_1 < \cdots < x_{n + 1} < \infty$. By using \eqref{eq:basis_det_ineq} and \eqref{eq:phi_dn}, we have $\det( \phi_j(u_i) )_{i, j = 0}^n = \frac{ \det( (-1)^j \basis{\Phi}_{r, j}(x_i) )_{i, j = 1}^{n + 1} }{ \prod_{j = 1}^{n + 1} Z_{r, j} \prod_{i = 0}^n \sqrt{ u_i \log(1 / u_i) } } > 0$. Thus $\phi_i(u)$ satisfies \eqref{eq:kellogg_condition} and the statement follows from Theorem \ref{the:Kellogg} and \eqref{eq:phi_dn}.
\end{proof}
\begin{figure}[htbp]
\centering
\includegraphics{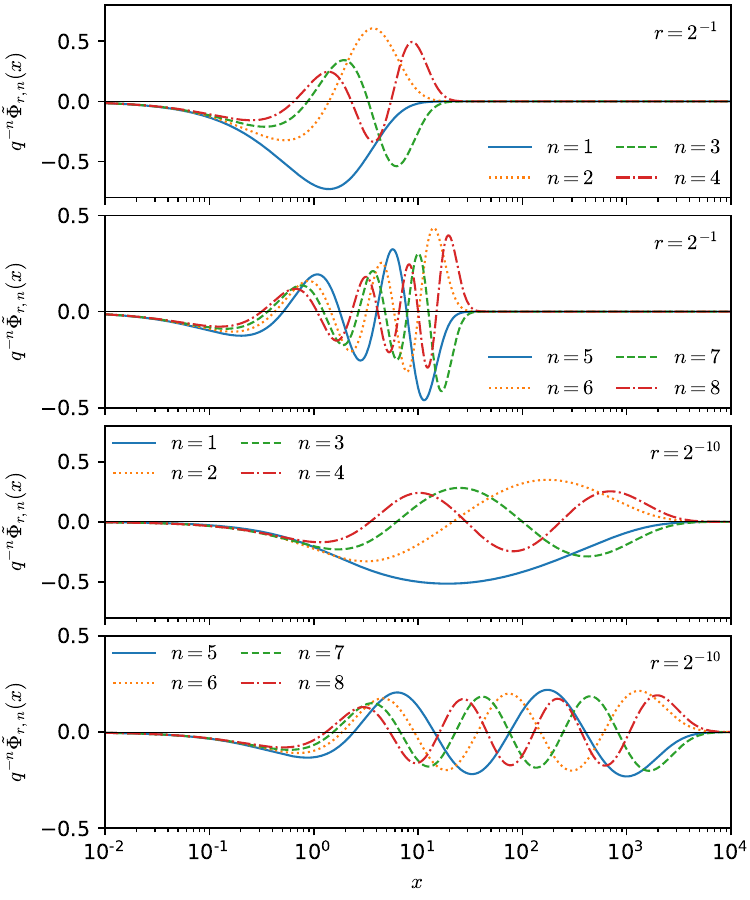}
\caption{Basis functions $\basis{\Phi}_{r, n}(x)$ associated with $\Phi_r(u)$. We multiplied $q^{-n}$ based on the inequality $\lvert q^{-n} \basis{\Phi}_{r, n}(x) \rvert \le 1$ on $(0, \infty)$ by \eqref{eq:basis_Phi_ineq_hatrho}. The computational method is given in Section \ref{sec:basis_numerical}.
}
\label{fig:Phi_basis}
\end{figure}

\section{Numerical results}
\label{sec:numerical}
We describe the numerical implementations and conduct numerical experiments. The corresponding Python codes are available in GitHub \cite{KoyamaGitHub2023}.

\subsection{Computation of $\Phi_r(u)$ and $\Phi_r'(u)$ on $[-1, 1]$}
\label{sec:Phi_numerical}
Let $-1 \le u \le 1$. For the computation of $\Phi_r(u)$ and $\Phi_r'(u)$, we used \eqref{eq:Phi_qn2} and \eqref{eq:dPhi_V}, which require the computation of $1 + 2 \sum_{n = 1}^\infty q^{n^2} T_n(\pm u)$ and $\sum_{n = 0}^\infty q^{ 2 n^2 + 2 n } V_n(1 - 2 u^2)$. We approximated them by $1 + 2 \sum_{n = 1}^{N_0} q^{n^2} T_n(\pm u)$ with $N_0 \in \intge{1}$ and $\sum_{n = 0}^{N_1} q^{ 2 n^2 + 2 n } V_n(1 - 2 u^2)$ with $N_1 \in \intge{1}$, respectively. By using inequalities 
\begin{subequations}
\begin{align}
&\sum_{n = N_0 + 1}^\infty q^{n^2} < \int_{N_0}^\infty \enum^{ - \pi \frac{ \EK(r) }{ \EK(k) } x^2 } \frac{ x }{ N_0 } \del x = \frac{ q^{N_0^2} }{ 2 N_0 \log(q^{-1}) }, \\
&\sum_{n = N_1 + 1}^\infty q^{2 n ^2 + 2 n} (2 n + 1) < \int_{N_1}^\infty \enum^{- \pi \frac{ \EK(r) }{ \EK(k) } (2 x^2 + 2 x)} \frac{ 2 N_1 + 3 }{ 2 N_1 + 1 } (2 x + 1) \del x = \frac{ 2 N_1 + 3 }{ 2 N_1 + 1 } \frac{ q^{ 2 N_1^2 + 2 N_1 } }{ 2 \log(q^{-1}) },
\end{align}
\end{subequations}
\eqref{eq:qn2Tn_ineq}, \eqref{eq:q2n22nVn_ineq}, $\lvert T_n(u) \rvert \le 1$, and $\lvert V_n(u) \rvert \le 2 n + 1$ for $n \in \intge{0}$, the relative errors by the truncations can be bounded as
\begin{subequations}
\label{eq:TV_re}
\begin{align}
&\left\lvert \frac{ 2 \sum_{n = N_0 + 1}^\infty q^{n^2} T_n(\pm u) }{ 1 + 2 \sum_{n = 1}^\infty q^{n^2} T_n(\pm u) } \right\rvert < \frac{ q^{N_0^2} }{ N_0 \sqrt{2 r \EK(k) / \pi} \log(q^{-1}) }, \\
&\left\lvert \frac{ \sum_{n = N_1 + 1}^\infty q^{ 2 n^2 + 2 n } V_n(1 - 2 u^2) }{ \sum_{n = 0}^\infty q^{ 2 n^2 + 2 n } V_n(1 - 2 u^2) } \right\rvert < \frac{ 2 N_1 + 3 }{ 2 N_1 + 1 } \frac{ 2 q^{ 2 (N_1 + \frac{1}{2})^2 } }{ k r^\frac{1}{4} (2 \EK(k) / \pi)^\frac{3}{2} \log(q^{-1}) }.
\end{align}
\end{subequations}

For the computation of the right-hand sides of \eqref{eq:Phi_qn2}, \eqref{eq:dPhi_V}, and \eqref{eq:TV_re}, we need the values of  $q$, $2 \EK(k) / \pi$, and $k$. To compute these values from our primary parameter $r$ (or complementary modulus $k' \coloneqq \sqrt{1 - k^2} = r$), we used the following equalities based on the arithmetic-geometric mean \cite[Sec.~52]{Jacobi1829}
\begin{subequations}
\label{eq:agm}
\begin{align}
&a_0 \coloneqq 1, g_0 \coloneqq r, a_{n + 1} \coloneqq (a_n + g_n) / 2, g_{n + 1} \coloneqq \sqrt{a_n g_n}, n \in \intge{0}, \\
&q = \frac{ 1 - r^2 }{ 16 r } \prod_{n = 1}^\infty \left( \frac{g_n}{a_n} \right)^\frac{3}{2^n}, \frac{ \pi }{ 2 \EK(k) } = \lim_{n \to \infty} a_n = \lim_{n \to \infty} g_n, k = \sqrt{1 - r^2}.
\end{align}
\end{subequations}

For the numerical computation, we used mpmath library \cite{mpmath}, which is the arbitrary precision Python library. Let $n_\mathrm{prec} \in \intge{1}$ be the precision in bits (mp.prec parameter in mpmath). For \eqref{eq:agm}, $n$ was increased to satisfy $\lvert a_n - g_n \rvert \le g_n 2^{- (n_\mathrm{prec} - 1)}$. Then we determined $N_0$ and $N_1$ to satisfy right-hand sides of \eqref{eq:TV_re} be equal or less than $2^{- n_\mathrm{prec}}$. $1 + 2 \sum_{n = 1}^{N_0} q^{n^2} T_n(\pm u)$ and $\sum_{n = 0}^{N_1} q^{ 2 n^2 + 2 n } V_n(1 - 2 u^2)$ were evaluated by Clenshaw algorithm \cite[Sec.~2.4.1]{Mason2003}.

\subsection{Computation of basis functions}
\label{sec:basis_numerical}
Let $0 < r < 1$, $\vt_r \in \vtset_r$, and $\rhoana[\vt_r] > 1$. Let $n \in \intge{0}$, $M \in \intge{1}$, and $z \in \mathbb{C}$. For the numerical computation of basis functions, we consider the approximation based on the Gaussian quadrature. By the change of variable $u = \cos\theta$ for \eqref{eq:basis}, we have
\begin{equation}
\label{eq:basisn_u}
\basis{\vt}_{r, n}(z) = \frac{1}{\pi} \int_{-1}^1 \enum^{- z \vt_r(u)} \frac{T_n(u)}{\sqrt{1 - u^2}} \del u, z \in \mathbb{C}, n \in \intge{0}.
\end{equation}
By applying $M$-point Gauss-Chebyshev quadrature \cite[Example 1.49]{Gautschi2004} to the right-hand side of \eqref{eq:basisn_u}, we have
\begin{equation}
\label{eq:basisn_gauss}
\basis{\vt}_{r, n}(z) \approx \frac{1}{M} \sum_{\nu = 1}^M \enum^{- \vt_r\left( \cos\left( \frac{ 2 \nu - 1 }{ 2 M } \pi \right) \right) z} \cos\left( n \frac{ 2 \nu - 1 }{ 2 M } \pi \right), z \in \mathbb{C}, n \in \intge{0}.
\end{equation}
By using \eqref{eq:exp_vt_fourier} and
$
\frac{1}{M} \sum_{\nu = 1}^M \cos\left( j \frac{ 2 \nu - 1 }{ 2 M } \pi \right) =
\begin{cases}
0, & j \in \mathbb{Z} \setminus 2 M \mathbb{Z}, \\
(-1)^\frac{ j }{ 2M }, & j \in 2 M \mathbb{Z},
\end{cases}
$
the approximation error of \eqref{eq:basisn_gauss} is expressed as
\begin{equation}
\label{eq:basisn_gauss_error}
E_{n, M}(z) \coloneqq \basis{\vt}_{r, n}(z) - \frac{1}{M} \sum_{\nu = 1}^M \enum^{- \vt_r\left( \cos\left( \frac{ 2 \nu - 1 }{ 2 M } \pi \right) \right) z} \cos\left( n \frac{ 2 \nu - 1 }{ 2 M } \pi \right) = \sum_{l =1 }^\infty (-1)^{l-1} (\basis{\vt}_{r, \lvert 2Ml-n \rvert}(z) + \basis{\vt}_{r, 2Ml+n}(z) ), z \in \mathbb{C}.
\end{equation}
By using \eqref{eq:basis_0_dWn} and \eqref{eq:basisn_gauss_error}, we have
\begin{equation}
E_{n, M}(0) =
\begin{cases}
0, & n \in \intge{0} \setminus 2 M \intge{1}, \\
(-1)^{ \frac{n}{2M} - 1 }, & n \in 2 M \intge{1},
\end{cases}
M \in \intge{1},
\end{equation}
which indicates
\begin{equation}
\label{eq:basisn_gauss_E0}
E_{n, M}(0) = 0, M \in \intge{(n + 1) / 2}, n \in \intge{0}.
\end{equation}
$T_n(u)$ is an entire function, real on the real axis, and satisfies $\sup\limits_{u \in \ellipse_\rho} \lvert T_n(u) \rvert = \frac{ \rho^n + \rho^{ -n } }{ 2 }$ for $1 < \rho < \infty$. Then for a given $0 < x < \infty$, the integrand $\enum^{- x \vt_r(u)} T_n(u)$ of \eqref{eq:basisn_u} is analytic on $\ellipse_{\rhoana[\vt_r]}$, real on $[-1, 1]$, and satisfies
\begin{equation}
0 < \enum^{- r x} = \max_{u \in [-1, 1]} \lvert \enum^{- x \vt_r(u)} T_n(u) \rvert \le \sup_{u \in \ellipse_\rho} \lvert \RE (\enum^{- x \vt_r(u)} T_n(u)) \rvert \le \enum^{- x \inf\limits_{u \in \ellipse_\rho} \RE \vt_r(u)} \frac{ \rho^n + \rho^{ -n } }{ 2 } < \infty, \rho \in I_{\vt_r}, 0 < x < \infty.
\end{equation}
Then by applying Theorem \ref{the:achieser_stenger} to \eqref{eq:basisn_u} with $\rho = \hat{\rho}[\vt_r]$ and using \eqref{eq:infRe_hatrho_ineq} and $\int_{-1}^1 \frac{ 1 }{ \pi \sqrt{1 - u^2} } \del u = 1$, we have
\begin{equation}
\label{eq:basisn_gauss_stenger}
\lvert E_{n, M}(x) \rvert < \frac{16}{\pi} \frac{ \hat{\rho}[\vt_r]^n + \hat{\rho}[\vt_r]^{ -n } }{ 2 } \hat{\rho}[\vt_r]^{-2M}, 0 < x < \infty, M \in \intge{1}, n \in \intge{0}.
\end{equation}
For a given $n \in \intge{0}$, by using \eqref{eq:basisn_gauss_E0} and \eqref{eq:basisn_gauss_stenger}, we determined $M$ to satisfy $E_{n, M}(0) = 0$ and  $\lvert E_{n, M}(x) \rvert < 10^{-10} \hat{\rho}[\vt_r]^{-n}$ on $(0, \infty)$, where we multiplied $\hat{\rho}[\vt_r]^{-n}$ based on \eqref{eq:basis_ineq_hatrho}.

\subsection{Initialization of Remez algorithm}
\label{sec:init_remez}
Let $0 < a < b < \infty$ and $M \in \intge{1}$. Let $W(t)$ be a bounded nondecreasing function with at least $M + 1$ points of increase for $a \le t \le b$. We describe a computational method to find the best exponential sum for $f(x) \coloneqq \int_a^b \enum^{ - x t } \del W(t)$, $0 \le x < \infty$. By Theorem \ref{the:kammler_best}, the best exponential sum $\sum_{\nu = 1}^M c_\nu \enum^{- t_\nu x}$ exists and the equations to be satisfied are given as
\begin{subequations}
\label{eq:best_eq}
\begin{align}
&\text{$E_M(x) \coloneqq f(x) - \sum_{\nu = 1}^M c_\nu \enum^{- t_\nu x}$ and $x_0 \coloneqq 0 < x_1 < \cdots < x_{2M} < \infty$}, \\
\label{eq:best_eq_st}
&E'_M(x_i) = f'(x_i) + \sum_{\nu = 1}^M c_\nu t_\nu \enum^{- t_\nu x_i} = 0, i = 1, \ldots, 2 M, \\
\label{eq:best_eq_eq}
&E_M(x_{i - 1}) + E_M(x_i) = f(x_{i-1}) - \sum_{\nu = 1}^M c_\nu \enum^{- t_\nu x_{i-1}} + f(x_i) - \sum_{\nu = 1}^M c_\nu \enum^{- t_\nu x_i} = 0, i = 1, \ldots, 2 M.
\end{align}
\end{subequations}
To find the best exponential approximation, Remez algorithm can be used \cite{Barrar1970,Kammler1976,Braess2009,Hackbusch2019}. In the Remez algorithm, (i) each $x_1, \ldots, x_{2M}$ is moved to increase $\lvert E_M(x_i) \rvert$ for given $t_1, \ldots, t_M$ and $c_1, \ldots, c_M$. Then (ii) $c_1, \ldots, c_M$ and $t_1, \ldots, t_M$ are determined to satisfy $\eqref{eq:best_eq_eq}$ for given $x_1, \ldots, x_{2M}$. Procedures (i) and (ii) are iterated to converge the parameters.

Let $0 < h < \infty$. For given $x_i = ih$, $i = 1, \dots, 2M$, it is known that $t_1, \ldots, t_M$ and $c_1, \ldots, c_M$ satisfying \eqref{eq:best_eq_eq} can be determined by a variant of Prony's method \cite[Sec.~10.2]{Meinardus1967} \cite{Kammler1977}. The identical problem can also be solved by the Gaussian quadrature as follows:
\begin{lemma}
\label{lem:best_init}
Let $0 < a < b < \infty$ and $M \in \intge{1}$. Let $W(t)$ be a bounded nondecreasing function with at least $M + 1$ points of increase for $a \le t \le b$. Let $f(x) \coloneqq \int_a^b \enum^{-xt} \del W(t)$ for $0 \le x < \infty$. For a given $0 < h < \infty$, let us introduce $W_h(y) \coloneqq - W(- \log(y) / h)$ and $\widehat{W}_h(y) \coloneqq \int_{ \enum^{-bh} }^y (1 + Y) \del W_h(Y)$ for $\enum^{-bh} \le y \le \enum^{-ah}$. Then one has
\begin{align}
\label{eq:Wh_bounded}
&\text{$W_h(y)$ and $\widehat{W}_h(y)$ are bounded nondecreasing functions} \\ \nonumber
&\text{with at least $M + 1$ points of increase for $\enum^{-bh} \le y \le \enum^{-ah}$}, \\
\label{eq:Wh_eq}
&\int_a^b F(\enum^{-ht}) (1 + \enum^{-ht}) \del W(t) = \int_{\enum^{-bh}}^{\enum^{-ah}} F(y) (1 + y) \del W_h(y) \\ \nonumber
&= \int_{\enum^{-bh}}^{\enum^{-ah}} F(y) \del \widehat{W}_h(y) \text{ if $F(y)$ is continuous on $[\enum^{-bh}, \enum^{-ah}]$}.
\end{align}
Furthermore, there exists $y_{h, \nu}$ and $d_{h, \nu}$, $\nu = 1, \ldots, M$, satisfying
\begin{align}
\label{eq:yhdh_ineq}
&\text{$\enum^{-bh} < y_{h, 1} < \cdots < y_{h, M} < \enum^{-ah}$ and $d_{h, \nu} > 0, \nu = 1, \ldots, M$}, \\
\label{eq:best_eq_yhdh}
&f(ih) + f((i + 1) h) = \int_{\enum^{-bh}}^{\enum^{-ah}} y^i (1 + y) \del W_h(y) = \int_{\enum^{-bh}}^{\enum^{-ah}} y^i \del \widehat{W}_h(y) = \sum_{\nu = 1}^M d_{h, \nu} y_{h, \nu}^i, i = 0, \ldots, 2 M - 1.
\end{align}
\end{lemma}
\begin{proof}
The statement of \eqref{eq:Wh_bounded} for $W_h(y)$ follows from the definition. The statement of \eqref{eq:Wh_bounded} for $\widehat{W}_h(y)$ follows from Corollary \ref{cor:int_f_da}. The first equality of \eqref{eq:Wh_eq} follows from Theorem \ref{the:stieltjes_neg} and Theorem \ref{the:stieltjes_change}. The second equality of \eqref{eq:Wh_eq} follows from Theorem \ref{the:int_fphi_da}. By substituting $F(y) = y^i$, $i = 0, \ldots, 2 M - 1$, for \eqref{eq:Wh_eq}, we have the first and the second equalities of \eqref{eq:best_eq_yhdh}. By \eqref{eq:Wh_bounded} for $\widehat{W}_h(y)$ and Theorem \ref{the:gauss}, there exists $y_{h, \nu}$ and $d_{h, \nu}$, $\nu = 1, \ldots, M$, satisfying \eqref{eq:yhdh_ineq} and the third equality of \eqref{eq:best_eq_yhdh}.
\end{proof}
\begin{corollary}
\label{cor:best_init}
Let the conditions of Lemma \ref{lem:best_init} be satisfied. Let us introduce
\begin{equation}
\label{eq:thch}
t_{h, \nu} \coloneqq \frac{- \log(y_{h, M - \nu + 1})}{h}, c_{h, \nu} \coloneqq \frac{ d_{h, M - \nu + 1} }{ 1 + y_{h, M - \nu + 1} }, \nu = 1, \ldots, M.
\end{equation}
Then one has
\begin{align}
\label{eq:thch_ineq}
&\text{$a < t_{h, 1} < \cdots < t_{h, M} < b$ and $c_{h, \nu} > 0$, $\nu = 1, \ldots, M$}, \\
\label{eq:best_eq_thch}
&f(ih) + f((i + 1) h) = \sum_{\nu = 1}^M c_{h, \nu} \enum^{- t_{h, \nu} i h } + \sum_{\nu = 1}^M c_{h, \nu} \enum^{- t_{h, \nu} (i + 1) h }, i = 0, \ldots, 2 M - 1.
\end{align}
\end{corollary}

By comparing \eqref{eq:best_eq_eq} and \eqref{eq:best_eq_thch}, \eqref{eq:best_eq_eq} are satisfied by setting $x_i = ih$, $i = 1, \ldots, 2M$, $t_\nu = t_{h, \nu}$, $c_\nu = c_{h, \nu}$, $\nu = 1, \ldots, M$, which can be used for the initial parameters of Remez algorithm. For this purpose, we need to select a value of $0 < h < \infty$. For the initial parameters, by \eqref{eq:best_eq_eq}, we have $\lvert E_M(ih) \rvert = \lvert E_M(0) \rvert = \lvert f(0) - \sum_{\nu = 1}^M c_{h, \nu} \rvert$, $i = 1, \ldots, 2M$. Since the evaluation of $E_M(x)$ in the neighbourhood of $x_i = ih$ is necessary for the initial iteration of the Remez algorithm, a smaller value of $\lvert E_M(ih) \rvert = \lvert f(0) - \sum_{\nu = 1}^M c_{h, \nu} \rvert$ requires a higher precision for the computation of $E_M(x)$. Then we investigate properties of $f(0) - \sum_{\nu = 1}^M c_{h, \nu}$ as follows:
\begin{lemma}
\label{lem:Eh_eq}
Let the conditions of Lemma \ref{lem:best_init} and Corollary \ref{cor:best_init} be satisfied. Let $H_{M, h}$ be an $(M + 1) \times (M + 1)$ matrix, whose elements are given as $(H_{M, h})_{i,j} = \int_{ \enum^{-bh} }^{ \enum^{ -ah } } y^{i + j} \del W_h(y) = f((i + j) h)$, $i = 0, \ldots, M$, $j = 0, \ldots, M$. Let $G_{M, h}$ be an $M \times M$-matrix, whose elements are given as $(G_{M, h})_{i,j} = f((i + j) h) + 2 f((i + j + 1) h) + f((i + j + 2) h)$, $i = 0, \ldots, M - 1$, $j = 0, \ldots, M - 1$.
Let $\V{u}_M \coloneqq \left( (-1)^0 \quad (-1)^1 \quad \cdots \quad (-1)^M \right)^\trans$. Then $H_{M, h}$ and $G_{M, h}$ are positive definite matrices and one has
\begin{subequations}
\label{eq:Eh_eq}
\begin{align}
\label{eq:Eh_int}
&f(0) - \sum_{\nu = 1}^M c_{h, \nu} = \int_{\enum^{-bh}}^{\enum^{-ah}} \frac{1}{1 + y} \del \widehat{W}_h(y) - \sum_{\nu = 1}^M \frac{ d_{h, \nu} }{ 1 + y_{h, \nu} } \\
\label{eq:Eh_variational}
&= \min_{p \in \pset_M, p(-1) = 1} \int_a^b p(\enum^{-ht})^2 \del W(t) \\
&= \frac{ \det(H_{M, h}) }{ \det(G_{M, h}) } = \frac{ 1 }{ \V{u}_M^\trans H_{M, h}^{-1} \V{u}_M } > 0.
\end{align}
\end{subequations}
\end{lemma}
\begin{proof}
Let $\V{0}_n \coloneqq (0, \ldots, 0)^\trans \in \mathbb{R}^n$, $n \in \intge{1}$. Let $(\alpha_0, \ldots, \alpha_M)^\trans \in \mathbb{R}^{M + 1} \setminus \{ \V{0}_{M + 1} \}$. By using $\sum_{i, j = 0}^M (H_{M, h})_{i,j} \alpha_i \alpha_j = \int_{\enum^{-bh}}^{\enum^{-ah}} (\sum_{i = 0}^M \alpha_i y^i)^2 \del W_h(y)$, \eqref{eq:Wh_bounded} for $W_h(y)$, and Theorem \ref{the:stieltjes_pos}, we have $\sum_{i, j = 0}^M (H_{M, h})_{i,j} \alpha_i \alpha_j > 0$. Thus $H_{M, h}$ is a positive definite matrix. Let $(\beta_0, \ldots, \beta_{M - 1})^\trans \in \mathbb{R}^M \setminus \{ \V{0}_M \}$. By \eqref{eq:Wh_eq} with $F(y) = y^{i + j} (1 + y)$, we have $(G_{M, h})_{i,j} = \int_{\enum^{-bh}}^{\enum^{-ah}} y^{i + j} (1 + y)^2 \del W_h(y)$. Then by using $\sum_{i, j = 0}^{M - 1} (G_{M, h})_{i,j} \beta_i \beta_j = \int_{\enum^{-bh}}^{\enum^{-ah}} (\sum_{i = 0}^{M - 1} \beta_i y^i)^2 (1 + y)^2 \del W_h(t)$, \eqref{eq:Wh_bounded} for $W_h(y)$, and Theorem \ref{the:stieltjes_pos}, we have $\sum_{i, j = 0}^{M - 1} (G_{M, h})_{i,j} \beta_i \beta_j > 0$. Thus $G_{M, h}$ is a positive definite matrix.

By substituting $F(y) = 1 / (1 + y)$ for \eqref{eq:Wh_eq}, we have $f(0) = \int_{\enum^{-bh}}^{\enum^{-ah}} \del W_h(y) = \int_{\enum^{-bh}}^{\enum^{-ah}} \frac{1}{1 + y} \del \widehat{W}_h(y)$. With this equality and \eqref{eq:thch}, we have \eqref{eq:Eh_int}.

For a given $0 < h < \infty$, let $\hat{p}_i(y)$, $i = 0, \ldots, M$, be the $i$-th degree orthonormal polynomial satisfying
\begin{equation}
\int_{\enum^{-bh}}^{\enum^{-ah}} \hat{p}_i(y) \hat{p}_j(y) \del \widehat{W}_h(y) = \delta_{i, j}, i = 0, \ldots, M, j = 0, \ldots, M,
\end{equation}
whose existence is guaranteed by \eqref{eq:Wh_bounded} for $\widehat{W}_h(y)$ and Theorem \ref{the:gauss}.
For $z \in \mathbb{C} \setminus [\enum^{-bh}, \enum^{-ah}]$, the Markov function is expressed as \cite[Eq.~1.3.40 and Theorem 1.47]{Gautschi2004}
\begin{align}
\label{eq:markov_function}
&\int_{\enum^{-bh}}^{\enum^{-ah}} \frac{1}{z - y} \del \widehat{W}_h(y) = \sum_{\nu = 1}^M \frac{ d_{h, \nu} }{ z - y_{h, \nu} } + \int_{\enum^{-bh}}^{\enum^{-ah}} \frac{ \hat{p}_M(y) }{ \hat{p}_M(z) } \frac{1}{z - y} \del \widehat{W}_h(y).
\end{align}
By substituting  $z = -1$ for \eqref{eq:markov_function} and using $f(0) = \int_{\enum^{-bh}}^{\enum^{-ah}} \frac{1}{1 + y} \del \widehat{W}_h(y)$, \eqref{eq:thch}, and \eqref{eq:Wh_eq} with $F(y) = \frac{ \hat{p}_M(y) }{ \hat{p}_M(-1) } \frac{1}{1 + y}$, we have
\begin{equation}
\label{eq:Eh_markov}
f(0) - \sum_{\nu = 1}^M c_{h, \nu} = \int_{\enum^{-bh}}^{\enum^{-ah}} \frac{\hat{p}_M(y)}{\hat{p}_M(-1)} \del W_h(y).
\end{equation}
Let $p_i(y)$ be the $i$-th degree orthonormal polynomial satisfying
\begin{equation}
\label{eq:dWh_orthonormal}
\int_{\enum^{-bh}}^{\enum^{-ah}} p_i(y) p_j(y)\del W_h(y) = \delta_{i, j}, i = 0, \ldots, M, j = 0, \ldots, M,
\end{equation}
whose existence is guaranteed by \eqref{eq:Wh_bounded} for $W_h(y)$ and Theorem \ref{the:gauss}. By substituting a formula of Christoffel $\frac{ \hat{p}_M(y) }{ \hat{p}_M(-1) } = \frac{ \sum_{i = 0}^M p_i(-1) p_i(y) }{ \sum_{i = 0}^M p_i(-1)^2 }$ \cite[Sec.~2.5 and Theorem 3.1.4]{Szego1975} for \eqref{eq:Eh_markov} and using \eqref{eq:dWh_orthonormal}, we have
\begin{equation}
\label{eq:Eh_christoffel}
f(0) - \sum_{\nu = 1}^M c_{h, \nu} = \left( \sum_{i = 0}^M p_i(-1)^2 \right)^{-1}.
\end{equation}
The function $\left( \sum_{i = 0}^M p_i(x)^2 \right)^{-1}$, $x \in \mathbb{R},$ is called the (real) Christoffel function,
which has the following extremal property \cite[Proposition 9.12]{Schmudgen2017}
\begin{equation}
\label{eq:christoffel_variational}
\left( \sum_{i = 0}^M p_i(x)^2 \right)^{-1} = \min_{ p \in \pset_M, p(x) = 1 } \int_{ \enum^{-bh} }^{ \enum^{-ah} } p(y)^2 \del W_h(y), x \in \mathbb{R}.
\end{equation}
By substituting $x = -1$ for \eqref{eq:christoffel_variational} and using \eqref{eq:Wh_eq} with $F(y) = p(y)^2 / (1 + y)$ and \eqref{eq:Eh_christoffel}, we have $f(0) - \sum_{\nu = 1}^M c_{h, \nu} = \min\limits_{p \in \pset_M, p(-1) = 1} \int_a^b p(\enum^{-ht})^2 \del W(t)$. The Christoffel function can also be expressed by the moments \cite[Proposition 9.11]{Schmudgen2017} as
\begin{equation}
\label{eq:christoffel_det}
\left( \sum_{i = 0}^M p_i(-1)^2 \right)^{-1} = \left. - \det(H_{M, h}) \middle/
\begin{vmatrix}
0 & \V{u}_M^\trans \\
\V{u}_M & H_{M, h}
\end{vmatrix}
\right..
\end{equation}
Let $\V{v} \in \mathbb{R}^M$, whose $i$-th component is  given as $f((i-1) h) + f(i h)$, $i = 1, \ldots, M$. By using \eqref{eq:Eh_christoffel}, \eqref{eq:christoffel_det}, and 
\begin{equation}
\begin{vmatrix}
0 & \V{u}_M^\trans \\
\V{u}_M & H_{M, h}
\end{vmatrix}
=
\begin{vmatrix}
0 & 1 & \V{0}_M^\trans \\
1 & f(0) & \V{v}^\trans \\
\V{0}_M & \V{v} & G_{M, h}
\end{vmatrix}
= - \det(G_{M, h}),
\end{equation}
we have $f(0) - \sum_{\nu = 1}^M c_{h, \nu} = \frac{ \det(H_{M, h}) }{ \det(G_{M, h}) }$. Since $H_{M, h}$ is a positive definite matrix, $H_{M, h}$ is invertible. Then by using \eqref{eq:Eh_christoffel}, \eqref{eq:christoffel_det}, and the Schur complement determinant formula $
\begin{vmatrix}
0 & \V{u}_M^\trans \\
\V{u}_M & H_{M, h}
\end{vmatrix}
= \det(H_{M, h}) \left( 0 - \V{u}_M^\trans H_{M, h}^{-1} \V{u}_M \right)
$ \cite[Sec.~0.8.5]{Horn1985}, we have $f(0) - \sum_{\nu = 1}^M c_{h, \nu} = \frac{ 1 }{ \V{u}_M^\trans H_{M, h}^{-1} \V{u}_M }$. Since $H_{M, h}$ and $G_{M, h}$ are positive definite matrices, we have $\frac{ \det(H_{M, h}) }{ \det(G_{M, h}) } > 0$.
\end{proof}

\begin{lemma}
\label{lem:hr}
For a given $0 < r < 1$, the equation
\begin{equation}
\label{eq:hr}
(1 - r) \enum^{- h } +  \enum^{- (1 - r) h } - r = 0
\end{equation}
has a unique solution in $0 < h < \infty$. Let $h_r$ be the solution. Let $W_0(\enum^{-1})$ be the principal value of the Lambert W function \cite{Corless1996} at $\enum^{-1}$. Then one has
\begin{align}
&\text{$\max_{0 < h < \infty} \frac{ 1 + \enum^{ - r h } }{ 1 + \enum^{- h} } = r \enum^{ (1 - r) h_r}$ and the maximum is attained iff $h = h_r$}, \\
\label{eq:hr_ineq}
&(1 - r / 2) \frac{ \log((2 - r)/r) }{ 1 - r } < h_r < \frac{ \log((2 - r) / r) }{ 1 - r }, \\
\label{eq:hr_ineq_W}
&1 + W_0(\enum^{-1}) < h_r < (1 + W_0(\enum^{-1})) / r, \\
\label{eq:h1}
&\lim_{r \uparrow 1} h_r = 1 + W_0(\enum^{-1}) \approx 1.278464542761.
\end{align}
\end{lemma}
\begin{proof}
Let $0 < h < \infty$, $F(h) \coloneqq \frac{ 1 + \enum^{- r h} }{ 1 + \enum^{-h} }$, and $G(h) \coloneqq (1 - r) \enum^{-h} + \enum^{- (1 - r) h} - r$. Then we have $F'(h) = \enum^{- r h} \frac{ G(h) }{ (\enum^{-h} + 1)^2 }$. Since $G(h)$ satisfies $G'(h) < 0$, $G(0) > 0$, and $\lim_{h \to \infty} G(h) < 0$, there exists a unique $0 < h_r < \infty$ satisfying $G(h_r) = 0$. Then $F'(h) > 0$ on $(0, h_r)$, $F'(h_r) = 0$, and $F'(h) < 0$ on $(h_r, \infty)$, which indicate $F(h)$ is maximized iff $h = h_r$. By using $G(h_r) = 0$, we have $F(h_r) = r \enum^{ (1 - r) h_r }$. By the strict convexity of the exponential function, we have $G(h) > (2 - r) \enum^{ - \frac{ 1 - r }{1 - r / 2} h } - r \eqqcolon L(h)$. Since $L(h)$ has a root $(1 - r / 2) \frac{ \log((2 - r)/r) }{ 1 - r }$, we have the first inequality in \eqref{eq:hr_ineq}. By using $\enum^{-h} < \enum^{- (1 - r) h}$, we have $G(h) < (2 - r) \enum^{-(1 - r) h} - r \eqqcolon U(h)$. Since $U(h)$ has a root $\frac{ \log((2 - r)/r) }{ 1 - r }$, we have the second inequality in \eqref{eq:hr_ineq}. By using $\enum^{- (1 - r) h} > 1 - (1 - r) h$, we have $G(h) > (1 - r) (\enum^{-h} + 1 - h) \eqqcolon l(h)$. Since $l(h)$ has a root $1 + W_0(\enum^{-1})$, we have the first inequality in \eqref{eq:hr_ineq_W}. By using $\enum^{- (1 - r) h} < \frac{ 1 }{ 1 + (1 - r) h }$, we have $G(h) < (1 - r) \frac{ \enum^{- r h} + 1 - r h }{ 1 + (1 - r) h } \eqqcolon u(h)$. Since $u(h)$ has a root $(1 + W_0(\enum^{-1})) / r$, we have the second inequality in \eqref{eq:hr_ineq_W}. \eqref{eq:h1} follows from \eqref{eq:hr_ineq_W}.
\end{proof}
\begin{theorem}
\label{the:Eh_ineq}
Under the conditions in Lemma \ref{lem:best_init} and Corollary \ref{cor:best_init}, one has
\begin{subequations}
\begin{align}
\label{eq:Eh_ineq_gauss}
&0 < f(0) - \sum_{\nu = 1}^M c_{h, \nu} \le \left( \sqrt{ \frac{ 1 + \enum^{-ah} }{ 1 + \enum^{-bh} } } + 1 \right)^2 \left( \frac{ \sqrt{ \frac{ 1 + \enum^{-ah} }{ 1 + \enum^{-bh} } } + 1 }{ \sqrt{ \frac{ 1 + \enum^{-ah} }{ 1 + \enum^{-bh} } } - 1 } \right)^{-2M}  f(0) \\
\label{eq:Eh_ineq_hr}
&\le \left( \sqrt{a/b} \enum^{ \frac{1}{2} (1 - a/b) h_{a/b} } + 1 \right)^2 \left( \frac{ \sqrt{a/b} \enum^{ \frac{1}{2} (1 - a/b) h_{a/b} } + 1 }{ \sqrt{a/b} \enum^{ \frac{1}{2} (1 - a/b) h_{a/b} } - 1 } \right)^{-2M} f(0) \\
\label{eq:Eh_ineq_h0}
&< (\sqrt{2} + 1)^{-4M + 2} f(0) \approx 5.8284271247462 \times 33.970562748477^{-M} f(0),
\end{align}
\end{subequations}
where $h_{a/b}$ is defined in Lemma \ref{lem:hr} and the equality in \eqref{eq:Eh_ineq_hr} holds iff $h = h_{a/b} / b$.
\end{theorem}
\begin{proof}
By using \eqref{eq:Eh_eq} and \eqref{eq:gauss_error_dist}, we have
\begin{equation}
\label{eq:Eh_gauss_best}
0 < f(0) - \sum_{\nu = 1}^M c_{h, \nu} \le 2 \dist\left( \frac{1}{1 + y}, \pset_{2M-1}, [\enum^{-bh}, \enum^{-ah}] \right) \int_{\enum^{-bh}}^{ \enum^{-ah} } \del \widehat{W}_h(y).
\end{equation}
By using Chebyshev's equality \eqref{eq:dist_inverse}, we have
\begin{equation}
\label{eq:chebyshev_exp}
\dist\left( \frac{1}{1 + y}, \pset_{2M-1}, [\enum^{-bh}, \enum^{-ah}] \right) = \frac{ \enum^{-ah} - \enum^{-bh} }{ 2 (\enum^{-ah} + 1) (\enum^{-bh} + 1) } \left( \frac{ \sqrt{\enum^{-ah} + 1} + \sqrt{\enum^{-bh} + 1} }{ \sqrt{\enum^{-ah} + 1} - \sqrt{\enum^{-bh} + 1} } \right)^{ -(2M-1) }.
\end{equation}
By substituting \eqref{eq:chebyshev_exp} for \eqref{eq:Eh_gauss_best} and using $\int_{\enum^{-bh}}^{ \enum^{-ah} } \del \widehat{W}_h(y) \le (1 + \enum^{-ah}) f(0)$ by \eqref{eq:Wh_eq} with $F(y) = 1$, we have \eqref{eq:Eh_ineq_gauss}. By Lemma \ref{lem:hr}, we have $\max\limits_{0 < h < \infty} \frac{ 1 + \enum^{-ah} }{ 1 + \enum^{-bh} } = \frac{a}{b} \enum^{ (1 - a/b) h_{a/b} }$ and the maximum value is attained iff $bh = h_{a/b}$, which leads to \eqref{eq:Eh_ineq_hr}. By the second inequality in \eqref{eq:hr_ineq} and $0 < \frac{a}{b} < 1$, we have $\frac{a}{b} \enum^{ (1 - a/b) h_{a/b} } < 2 - a/b < 2$, which leads to \eqref{eq:Eh_ineq_h0}. 
\end{proof}

By \eqref{eq:Eh_ineq_gauss}, we have $\lim\limits_{h \downarrow 0} (f(0) - \sum_{\nu = 1}^M c_{h, \nu}) = \lim\limits_{h \uparrow \infty} (f(0) - \sum_{\nu = 1}^M c_{h, \nu}) = 0$. Thus the value of $f(0) - \sum_{\nu = 1}^M c_{h, \nu}$ can be close to zero and we need to be careful of the choice of $0 < h < \infty$. As shown in Theorem \ref{the:Eh_ineq}, the upper bound in \eqref{eq:Eh_ineq_gauss} is maximized iff $h = h_{a/b} / b$. Thus if the upper bound is the reasonable estimate of $f(0) - \sum_{\nu = 1}^M c_{h, \nu}$, the use of $h = h_{a/b} / b$ is expected to be larger value of $f(0) - \sum_{\nu = 1}^M c_{h, \nu}$. For $h = h_{a/b} / b$, the bound \eqref{eq:Eh_ineq_hr} decreases geometrically with respect to $M$ and the decreasing rate is faster than $ (\sqrt{2} + 1)^4 \approx 33.97$ as indicated by \eqref{eq:Eh_ineq_h0} and Table \ref{tab:hr}. It is necessary to use a higher precision than the value predicted by \eqref{eq:Eh_ineq_hr} to obtain initial parameters of Remez algorithm described in this section.

\begin{table}[htbp]
{\footnotesize
\caption{Values of $h_r$ in Lemma \ref{lem:hr} and factors in \eqref{eq:Eh_ineq_hr}. The equation \eqref{eq:hr} was numerically solved by Newton's method, whose initial value was given by the lower bound in \eqref{eq:hr_ineq}. }
\label{tab:hr}
\begin{center}
\begin{tabular}{lcccc}
\hline
 $r$
 & $h_r$
 & $r \enum^{ (1 - r) h_r }$
 & $\left( \sqrt{r} \enum^{ \frac{1}{2} (1 - r) h_r } + 1 \right)^2$
 & $\left( \frac{ \sqrt{r} \enum^{ \frac{1}{2} (1 - r) h_r } + 1 }
 { \sqrt{r} \enum^{ \frac{1}{2} (1 - r) h_r } - 1 } \right)^2$ \\
\hline
$2^{-1}$ & 1.7627472 & 1.2071068 & 4.4044750 & 452.27196 \\
$2^{-2}$ & 2.3162347 & 1.4203192 & 4.8038622 & 130.62386 \\
$2^{-3}$ & 2.9189951 & 1.6074984 & 5.1432416 & 71.677624 \\
$2^{-4}$ & 3.5547945 & 1.7507270 & 5.3970279 & 51.682703 \\
$2^{-5}$ & 4.2121565 & 1.8492696 & 5.5690267 & 42.999899 \\
$2^{-6}$ & 4.8833662 & 1.9120587 & 5.6776029 & 38.751138 \\
$2^{-7}$ & 5.5633166 & 1.9499874 & 5.7428263 & 36.543968 \\
$2^{-8}$ & 6.2486724 & 1.9720746 & 5.7806860 & 35.363856 \\
$2^{-9}$ & 6.9372990 & 1.9846151 & 5.8021425 & 34.725121 \\
$2^{-10}$ & 7.6278639 & 1.9916093 & 5.8140970 & 34.378220 \\
$2^{-11}$ & 8.3195557 & 1.9954597 & 5.8206744 & 34.190015 \\
$2^{-12}$ & 9.0118921 & 1.9975586 & 5.8242589 & 34.088232 \\
$2^{-13}$ & 9.7045920 & 1.9986941 & 5.8261977 & 34.033408 \\
$2^{-14}$ & 10.397495 & 1.9993046 & 5.8272399 & 34.004002 \\
$2^{-15}$ & 11.090509 & 1.9996311 & 5.8277974 & 33.988294 \\
$2^{-16}$ & 11.783584 & 1.9998050 & 5.8280942 & 33.979935 \\
$2^{-17}$ & 12.476693 & 1.9998972 & 5.8282516 & 33.975503 \\
$2^{-18}$ & 13.169820 & 1.9999459 & 5.8283349 & 33.973160 \\
$2^{-19}$ & 13.862956 & 1.9999717 & 5.8283787 & 33.971925 \\
$2^{-20}$ & 14.556097 & 1.9999852 & 5.8284018 & 33.971275 \\
\hline
\end{tabular}
\end{center}
}
\end{table}

\subsection{Finite completely monotonic function associated with the inverse power function}
\label{sec:inverse_power}
Let $0 < a < b < \infty$, $0 \le x < \infty$, and $0 < \eta < \infty$. For numerical experiments, we consider the finite completely monotonic function associated with the inverse power function \eqref{eq:f_inv} as follows:
\begin{align}
\label{eq:fcmf_inv}
&f(x) = \int_a^b \enum^{-xt} \frac{ t^{\eta - 1} }{ \Gamma(\eta) } \del t = \int_a^b \enum^{-xt} W'(t) \del t = \int_a^b \enum^{-xt} \del W(t) \eqqcolon f_{\eta, a, b}(x), \\
&W'(t) = \frac{ t^{\eta - 1} }{ \Gamma(\eta) }, W(t) = \int_a^t \frac{ s^{\eta - 1} }{ \Gamma(\eta) } \del s = \frac{ t^\eta - a^\eta }{ \Gamma(\eta + 1) }, a \le t \le b.
\end{align}
By using the generalized incomplete gamma function $\Gamma(y, A, B) = \int_A^B s^{y - 1} \enum^{-s} \del s$ for $0 < y < \infty$ and $0 \le A < B$ \cite{mpmath}, $f^{(n)}(x)$ for $n \in \intge{0}$ can be computed by
\begin{equation}
f^{(n)}(x) =
\begin{cases}
\frac{ (-1)^n }{ \Gamma(\eta) } \frac{ b^{n + \eta} - a^{n + \eta} }{ n + \eta } = \frac{ (-1)^n a^{n + \eta} \expm1((n + \eta) \log(b/a)) }{ (n + \eta) \Gamma(\eta) }, & x = 0, \\
\frac{ (-1)^n }{ \Gamma(\eta) } \frac{ \Gamma(n + \eta, ax, bx) }{ x^{ n + \eta } }, & 0 < x < \infty,
\end{cases}
\end{equation}
where $\expm1(x) \coloneqq \enum^x - 1$. By the change of variable $t = b \tau$ for \eqref{eq:fcmf_inv}, we have
\begin{equation}
\label{eq:fab_inv_scaling}
f_{\eta, a, b}(x) = \int_{a/b}^1 \enum^{-x b \tau} \frac{(b \tau)^{\eta - 1}}{\Gamma(\eta)} b \del \tau = b^\eta f_{\eta, a/b, 1}(bx),
\end{equation}
which indicates $f_{\eta, a, b}(x)$ can be obtained by linear scaling of $f_{\eta, a/b, 1}(x)$. Then without loss of generality, we fixed $b = 1$ in the following numerical experiments. For the other parameters, we used $\eta \in \{1/2, 1, 2\}$ and $a \in \{2^{-1}, 2^{-10}\}$. We used $\vt_{a/b} \in \{ \Phi_{a/b}, \vt_{\exp, a/b}, \vt_{\pset_2, a/b}, \vt_{\pset_1, a/b}, \vt_{\rset_{0, 1}, a/b} \}$ for the variable transformations. Among these parameter sets, by noting
\begin{align}
&f_{1, a, b}(x) = \int_a^b \enum^{-xt} \del t = \frac{ b - a }{ 2 } \int_{-1}^1 \enum^{-x b \vt_{\pset_1, a/b}(u)} \del u, \\
\label{eq:fab_1_2}
&f_{1/2, a, b}(x) = \int_a^b \enum^{-xt} \frac{ t^{- 1/2} }{ \sqrt{\pi} } \del t = \frac{ 2 }{ \sqrt{\pi} } \int_{\sqrt{a}}^{\sqrt{b}} \enum^{-x s^2} \del s = \frac{ \sqrt{b} - \sqrt{a} }{ \sqrt{\pi} } \int_{-1}^1 \enum^{-x b \vt_{\pset_2, a/b}(u)} \del u,
\end{align}
nodes $u_\nu$ and weights $c_\nu$ of $M$-point Gaussian quadrature in Lemma \ref{lem:vt_gauss} can be expressed by nodes $u_\nu^\mathrm{GL}$ and weights $c_\nu^\mathrm{GL}$ of $M$-point Gauss--Legendre quadrature as
\begin{align}
&\text{$u_\nu = u_\nu^\mathrm{GL}$ and $c_\nu = \frac{ b - a }{ 2 } c_\nu^\mathrm{GL}$ if $\eta = 1$ and $\vt_{a/b} = \vt_{\pset_1, a/b}$}, \\
\label{eq:fab_1_2_P2_gauss}
&\text{$u_\nu = u_\nu^\mathrm{GL}$ and $c_\nu = \frac{ \sqrt{b} - \sqrt{a} }{ \sqrt{\pi} } c_\nu^\mathrm{GL}$ if $\eta = 1/2$ and $\vt_{a/b} = \vt_{\pset_2, a/b}$},
\end{align}
where $\nu = 1, \ldots, M$.

\subsubsection{Computation of the best exponential sum approximation}
\label{sec:numerical_best}
We consider the computation of the best exponential sum approximation for \eqref{eq:fcmf_inv}. We used $M = 1, \dots, 17$ for the numerical experiments. For $M = 17$, based on the bound \eqref{eq:Eh_ineq_hr} and the values in Table \ref{tab:hr}, we need the precision at least $4.404 \times 452.3^{-17} \approx 3.174 \times 10^{-45}$ for $a = 2^{-1}$ and $5.814 \times 34.38^{-17} \approx 4.439 \times 10^{-26}$ for $a = 2^{-10}$. For the consideration of the overestimation of the bound and the accuracy of the parameters $t_{h, \nu}$ and $c_{h, \nu}$, we used additional $10^{-30}$ precision. As a result, we required the precision of $3.174 \times 10^{-(45+30)} \approx 2^{-247.5}$ for $a = 2^{-1}$ and $4.439 \times 10^{-(26+30)} \approx 2^{-183.9}$ for $a = 2^{-10}$. Then we set the precision parameter (mp.prec) of mpmah library as 248 bit for $a = 2^{-1}$ and 184 bit for $a = 2^{-10}$. Let $E_{M, h} \coloneqq f(0) - \sum_{\nu = 1}^M c_{h, \nu}$. In Figure \ref{fig:EMh_f0}, we showed $E_{M, h} / f(0)$ and the corresponding upper bound \eqref{eq:Eh_ineq_gauss}. For $a = 2^{-1}$, we see that the upper bound gives the accurate estimate of $E_{M, h}$. For $a = 2^{-10}$, the upper bound becomes larger by increasing $\eta$ for $h > h_{a/b} / b$. However, the deviation for $h < h_{a/b} / b$ is modest and the selection of $h = h_{a/b} / b$ succeeds to give a larger value of $E_{M, h}$. Then we used $h = h_{a/b} / b$ for all cases.
\begin{figure}[htbp]
\centering
\includegraphics{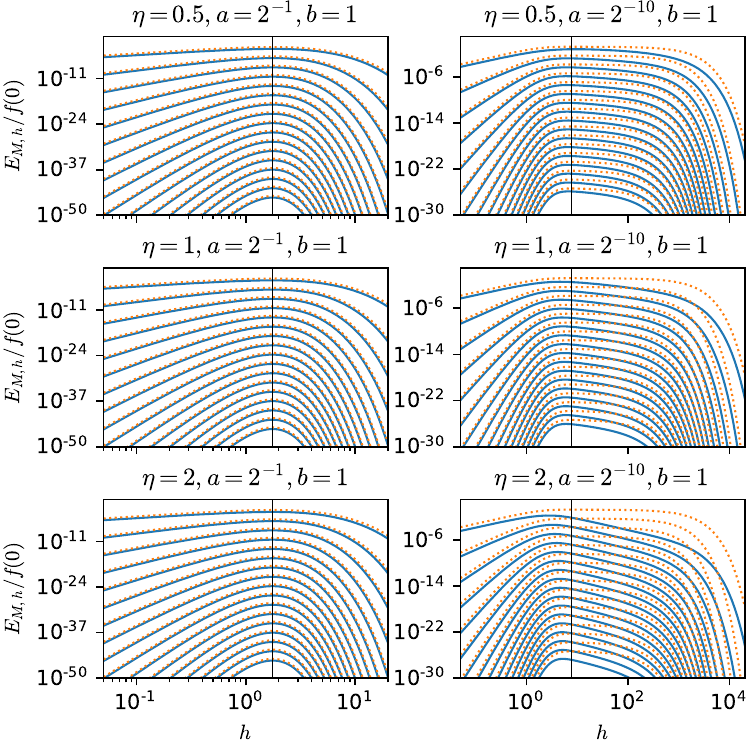}
\caption{ $E_{M, h} / f(0)$ (solid curves), $M = 1$ (top solid curve), $\ldots, 17$ (bottom solid curve), where $E_{M, h} \coloneqq f(0) - \sum_{\nu = 1}^M c_{h, \nu}$ and $f(x)$ is defined in \eqref{eq:fcmf_inv}. The dotted curves represent the upper bound in \eqref{eq:Eh_ineq_gauss}, $M = 1$ (top dotted curve), $\ldots, 17$ (bottom dotted curve). The vertical solid line represents $h = h_{a/b} / b$ in Theorem \ref{the:Eh_ineq}. We used $E_{M, h} = (\V{u}_M^\trans H_{M, h}^{-1} \V{u}_M)^{-1}$ in Lemma \ref{lem:Eh_eq} for the computation.
}
\label{fig:EMh_f0}
\end{figure}

Next, we consider the computation of nodes $y_{h, \nu}$ and weights $d_{h, \nu}$, $\nu = 1, \ldots, M$, of the Gaussian quadrature in Lemma \ref{lem:best_init}. For the computation, we used the discretized Stieltjes procedure by Gautschi \cite{Gautschi1982} and Golub-Welsch algorithm \cite{Golub1969}. For the discretized Stieltjes procedure, it is necessary to approximate the integral $\int_{ \enum^{ -bh } }^{ \enum^{ -ah } } P(y) \del \widehat{W}_h(y)$ by some $M_\mathrm{DS}$-point quadrature method with nodes $y_{h, \nu}^\mathrm{DS}$ and weights $d_{h, \nu}^\mathrm{DS}$, $\nu = 1, \ldots, M_\mathrm{DS}$, as
\begin{equation}
\int_{ \enum^{ -bh } }^{ \enum^{ -ah } } P(y) \del \widehat{W}_h(y) \approx \sum_{\nu = 1}^{M_\mathrm{DS}} d_{h, \nu}^\mathrm{DS} P(y_{h, \nu}^\mathrm{DS}), P \in \pset_{2M-1}.
\end{equation}
Let $\vt_{a/b} \in \vtset_{a/b}$ and $\rhoana[\vt_{a/b}] > 1$. Then we have
\begin{align}
\label{eq:init_remez_vt}
&\int_{ \enum^{ -bh } }^{ \enum^{ -ah } } P(y) \del \widehat{W}_h(y) = \int_a^b P(\enum^{-ht}) (1 + \enum^{-ht}) \del W(t) = \int_{-1}^1 P(\enum^{-h b \vt_{a/b}(u)}) (1 + \enum^{-h b \vt_{a/b}(u)}) \del W(b \vt_{a/b}(u)) \\ \nonumber
&= \int_{-1}^1 P(\enum^{-h b \vt_{a/b}(u)}) (1 + \enum^{-h b \vt_{a/b}(u)}) b \vt_{a/b}'(u) W'(b \vt_{a/b}(u)) \del u.
\end{align}
By applying $M_\mathrm{DS}$-point Gauss--Legendre quadrature with nodes $u_\nu^\mathrm{GL}$ and weights $c_\nu^\mathrm{GL}$, $\nu = 1, \ldots, M_\mathrm{DS}$, to the right-hand side of \eqref{eq:init_remez_vt}, we used
\begin{subequations}
\label{eq:ydDS}
\begin{align}
&y_{h, \nu}^\mathrm{DS} = \enum^{-h b \vt_{a/b}(u_\nu^\mathrm{GL})}, \\
&d_{h, \nu}^\mathrm{DS} = c_\nu^\mathrm{GL} (1 + \enum^{-h b \vt_{a/b}(u_\nu^\mathrm{GL})}) b \vt'_{a/b}(u_\nu^\mathrm{GL}) W'(b \vt_{a/b}(u_\nu^\mathrm{GL})),
\end{align}
\end{subequations}
where $\nu = 1, \ldots, M_\mathrm{DS}$. Let $t_{h, \nu}^{\vt_{a/b}, M_\mathrm{DS}}$ (resp. $c_{h, \nu}^{\vt_{a/b}, M_\mathrm{DS}}$), $1 \le \nu \le M$, be the estimate of $t_{h, \nu}$ (resp. $c_{h, \nu}$) in Corollary \ref{cor:best_init} computed by the discretized Stieltjes procedure with \eqref{eq:ydDS} and Golub-Welsch algorithm. To measure the convergence by increasing $M_\mathrm{DS}$, we define the maximum relative error as
\begin{equation}
\label{eq:MREh}
\mathrm{MRE}_h \coloneqq \max_{X \in \{ t, c \}, 1 \le \nu \le M} \left\lvert \frac{ X_{h, \nu}^{\vt_{a/b}, M_\mathrm{DS}} - X_{h, \nu}^{\vt_{a/b}, 2  M_\mathrm{DS}} }{ X_{h, \nu}^{\vt_{a/b}, 2 M_\mathrm{DS}} } \right\rvert.
\end{equation}
In Figure \ref{fig:best_init_MDS}, we showed $\mathrm{MRE}_h$ with $h = h_{a/b} / b$. Since $\vt_{\pset_2, a/b}$ gives the good convergence for all conditions, we used $\vt_{a/b} = \vt_{\pset_2, a/b}$. Based on the convergence, we used $M_\mathrm{DS} = 96$ and $M_\mathrm{DS} = 192$ for $a = 2^{-1}$ and $a = 2^{-10}$, respectively.
\begin{figure}[htbp]
\centering
\includegraphics{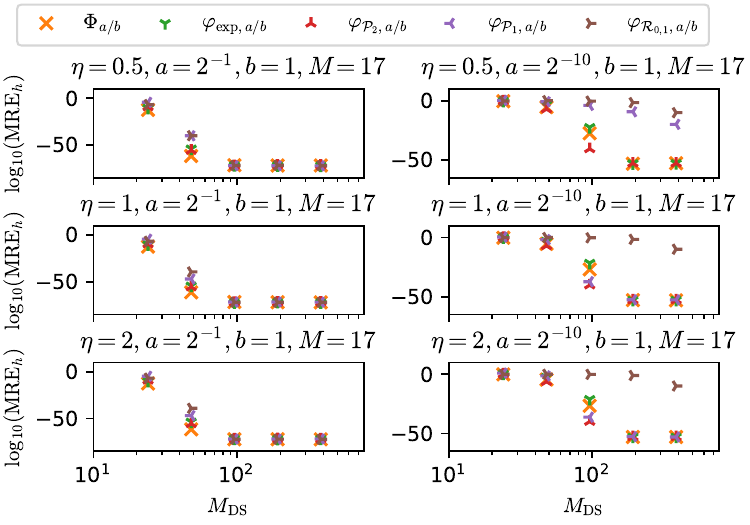}
\caption{$\mathrm{MRE}_h$ defined in \eqref{eq:MREh} with $h = h_{a/b} / b$. $\vt_{a/b}$ is specified in the legend. $M_{\mathrm{DS}} = 24, 48, 96, 192, 384$.
}
\label{fig:best_init_MDS}
\end{figure}

Then we set the initial parameters of the Remez algorithm as $t_\nu = t_{h, \nu}^{\vt_{a/b}, M_\mathrm{DS}}$, $c_\nu = c_{h, \nu}^{\vt_{a/b}, M_\mathrm{DS}}$, $\nu = 1, \ldots, M$, and $x_i = ih$, $i = 1, \ldots, 2M$, with $h = h_{a/b}/b$. We iterated the Remez algorithm to satisfy $\max\limits_{1 \le i \le 2M} \lvert \frac{ E_M(x_{i-1}) + E_M(x_i) }{ E_M(x_i)} \rvert \le \varepsilon$ and $\max\limits_{1 \le i \le 2M} \lvert x_i \frac{ E_M'(x_i) }{ E_M(x_i) } \rvert \le \varepsilon$, where $E_M(x)$ is defined in \eqref{eq:best_eq} and we set $\varepsilon = 10^{-10}$. For further refinement, we conducted 5 iterations of Newton's method for \eqref{eq:best_eq}. We confirmed the obtained error curves have the equioscillation property \cite{KoyamaZenodo2023}.

\subsubsection{Computation of nodes and weights in Lemma \ref{lem:vt_gauss}}
\label{sec:numerical_vt_gauss}
We consider the computation of nodes $u_\nu$ and weights $c_\nu$ of $M$-point Gaussian quadrature in Lemma \ref{lem:vt_gauss}. For the discretized Stieltjes procedure, we need to approximate the integral $\int_{-1}^1 P(u) \del W(b \vt_{a / b}(u))$ by some $M_\mathrm{DS}$-point quadrature method with nodes $u_\nu^\mathrm{DS}$ and weights $c_\nu^\mathrm{DS}$ as
\begin{equation}
\int_{-1}^1 P(u) \del W(b \vt_{a / b}(u)) \approx \sum_{\nu = 1}^{M_\mathrm{DS}} c_\nu^\mathrm{DS} P(u_\nu^\mathrm{DS}), P \in \pset_{2M-1}.
\end{equation}
Let $\vt_{a/b} \in \vtset_{a/b}$ and $\rhoana[\vt_{a/b}] > 1$. By using
\begin{equation}
\int_{-1}^1 P(u) \del W(b \vt_{a / b}(u)) = \int_{-1}^1 P(u) b \vt'_{a / b}(u) W'(b \vt_{a / b}(u)) \del u
\end{equation}
and applying $M_\mathrm{DS}$-point Gauss--Legendre quadrature, we used
\begin{equation}
\label{eq:ucDS}
u_\nu^\mathrm{DS} = u_\nu^\mathrm{GL}, c_\nu^\mathrm{DS} = c_\nu^\mathrm{GL} b \vt'_{a / b}(u_\nu^\mathrm{GL}) W'(b \vt_{a / b}(u_\nu^\mathrm{GL})), \nu = 1, \ldots, M_\mathrm{DS}.
\end{equation}

Let $u_\nu^{\vt_{a/b}, M_\mathrm{DS}}$ (resp. $c_\nu^{\vt_{a/b}, M_\mathrm{DS}}$), $1 \le \nu \le M$, be the estimate of $u_\nu$ (resp. $c_\nu$) in Lemma \ref{lem:vt_gauss} computed by the discretized Stieltjes procedure with \eqref{eq:ucDS} and Golub-Welsch algorithm. Let $t_\nu^{\vt_{a/b}, M_\mathrm{DS}} \coloneqq b \vt_{a/b}(u_\nu^{\vt_{a/b}, M_\mathrm{DS}})$. To measure the convergence by increasing $M_\mathrm{DS}$, we define the maximum relative error as
\begin{equation}
\label{eq:MRE}
\mathrm{MRE} \coloneqq \max_{X \in \{ t, c \}, 1 \le \nu \le M} \left\lvert \frac{ X_\nu^{\vt_{a/b}, M_\mathrm{DS}} - X_\nu^{\vt_{a/b}, 2  M_\mathrm{DS}} }{ X_\nu^{\vt_{a/b}, 2 M_\mathrm{DS}} } \right\rvert.
\end{equation}
Based on MRE in Figure \ref{fig:gauss_MDS}, we used $M_\mathrm{DS} = 96$ and $M_\mathrm{DS} = 1536$ for $a = 2^{-1}$ and $a = 2^{-10}$, respectively.
\begin{figure}[htbp]
\centering
\includegraphics{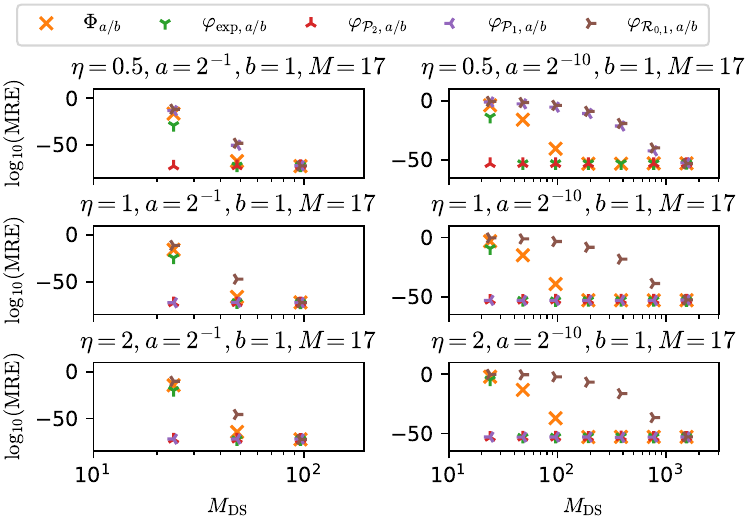}
\caption{MRE defined in \eqref{eq:MRE}. $\vt_{a/b}$ is specified in the legend. For $a = 2^{-1}$ (left figures), $M_{\mathrm{DS}} = 24, 48, 96$. For $a = 2^{-10}$ (right figures), $M_{\mathrm{DS}} = 24, 48, 96, 192, 384, 768, 1536$. The symbol for $\vt_{\pset_2, a/b}$ at $M_\mathrm{DS} = 384$ in the upper right figure is not displayed because $\mathrm{MRE} = 0$.
}
\label{fig:gauss_MDS}
\end{figure}

\subsubsection{Errors of exponential sum approximations}
Let $E_M$ be the maximum absolute error of the $M$-term exponential sum approximation on $[0, \infty)$. To measure the decreasing rate of the maximum absolute error by increasing $M$, we showed $E_M / E_{M+1}$ in Figure \ref{fig:EM}. For $a = 2^{-1}$ or $a = 2^{-10}$ and $\eta = 1/2$, the decreasing rate of $E_M$ with $\Phi_{a/b}$ is comparable to that of the best exponential sum approximation and is fastest among the variable transformations. For $a = 2^{-10}$ and $\eta = 1$, the decreasing rate of $E_M$ with $\vt_{\pset_2, a/b}$ is comparable to that of $\Phi_{a/b}$ for $M = 1$. For $a = 2^{-10}$ and $\eta = 2$, the decreasing rate of $E_M$ with $\vt_{\pset_2, a/b}$ is faster than that of $\Phi_{a/b}$ for $M = 1,2$. However, by increasing $M$, the decreasing rate of $E_M$ with $\Phi_{a/b}$ becomes fastest among the variable transformations and is comparable to that of the best exponential sum approximation. Thus the decreasing rate of $E_M$ with $\Phi_{a/b}$ is not necessary fastest among the variable transformations for smaller $M$ but it becomes fastest and is comparable to the best exponential sum approximation by increasing $M$ in the current examples.
\begin{figure}[htbp]
\centering
\includegraphics{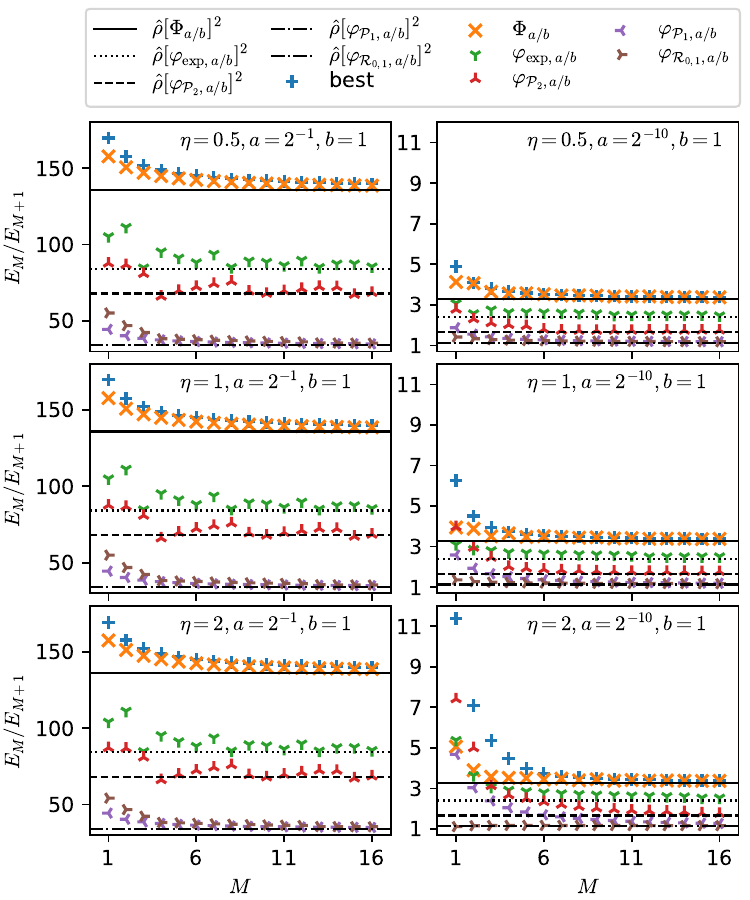}
\caption{The ratio of $E_M$ to $E_{M+1}$, where $E_M$ is the maximum absolute error of the $M$-term exponential sum approximation of \eqref{eq:fcmf_inv} on $[0, \infty)$. The ``+'' symbol represents the best exponential sum approximation and the other symbols represent the exponential sum approximation by the Gaussian quadrature with the variable transformation specified in the legend. The corresponding $x$-dependent error curves with the numerically identified maximum absolute errors have been uploaded to Zenodo \cite{KoyamaZenodo2023}.
}
\label{fig:EM}
\end{figure}

In Figure \ref{fig:Ex}, we showed $x$-dependent error curves of the exponential sum approximations with $\eta = 1$ and $M = 8$. For $a = 2^{-1}$, we see that the error curve of the Gaussian quadrature with $\Phi_{a/b}$ is well approximated by the leading term (i.e., $2 \varepsilon_{M, 2M} \basis{\Phi}_{a/b, 2M}(bx)$) of the basis expansion in Theorem \ref{the:f_gauss_basis}. We note that the absolute error of the Gaussian quadrature with $\Phi_{a/b}$ is smaller in the wide range of $x$ compared with that of the best exponential sum approximation. For $a = 2^{-10}$, the error curve of the Gaussian quadrature with $\Phi_{a/b}$ is well approximated by the leading 5 terms (i.e., $2 \sum_{n = 2M}^{2M+4} \varepsilon_{M, n} \basis{\Phi}_{a/b, n}(bx)$) of the basis expansion.
\begin{figure}[htbp]
\centering
\includegraphics{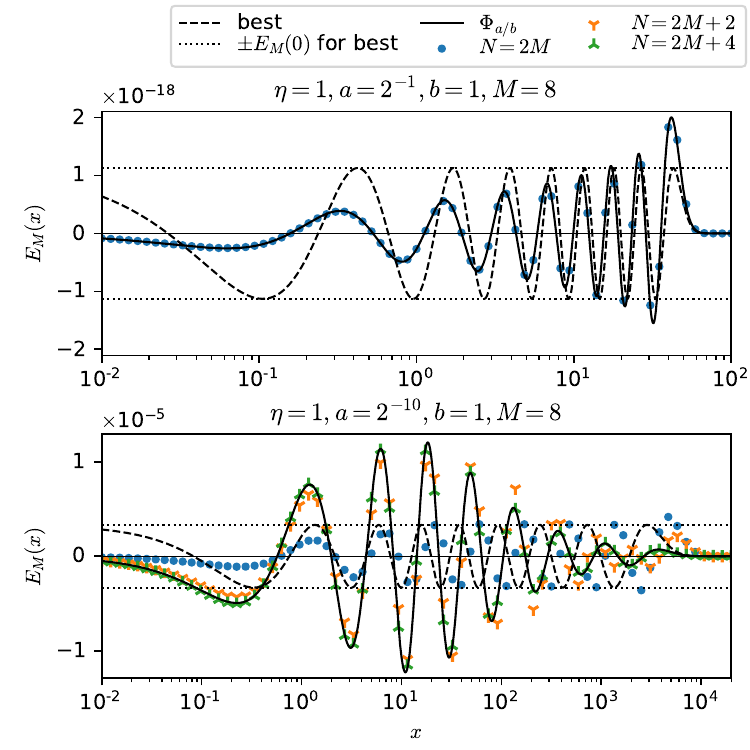}
\caption{$E_M(x) \coloneqq f(x) - \sum_{\nu = 1}^M c_\nu \enum^{-t_\nu x}$ for the best exponential sum approximation (dashed curve) and $E_M(x) \coloneqq f(x) - \sum_{ \nu = 1 }^M c_\nu \enum^{ - b \Phi_{a/b}(u_\nu) x }$ for the Gaussian quadrature with $\vt_{a/b} = \Phi_{a/b}$ in Lemma \ref{lem:vt_gauss} (solid curve). $f(x)$ is defined in \eqref{eq:fcmf_inv}. Dotted horizontal lines represent $\pm E_M(0)$ for the best exponential sum approximation. The symbols represent $2 \sum_{n = 2M}^N \varepsilon_{M, n} \basis{\Phi}_{a/b, n}(bx)$ in Theorem \ref{the:f_gauss_basis}.
}
\label{fig:Ex}
\end{figure}

\section{Discussion}
\label{sec:discussion}
We applied the Achieser--Stenger bound (Theorem \ref{the:achieser_stenger}) to obtain the approximation error bound of a finite completely monotonic function by the Gaussian quadrature with an analytic variable transformation $\vt_r(u)$ (Lemma \ref{lem:Evtx}). We introduced $\hat{\rho}[\vt_r]$ in Lemma \ref{lem:hatrho} to obtain the best possible maximum absolute error bound (Theorem \ref{the:Evt}). Then we obtained $\Phi_r(u)$ in Lemma \ref{lem:Phi} as the maximizer of $\hat{\rho}[\vt_r]$ (Theorem \ref{the:Phi}). The function satisfies $\Phi_r(\cos z) = \dn(\EK(k) z / \pi , k)$ on $\mathbb{C}$ by \eqref{eq:Phi_cos}, which indicates the composition of $\Phi_r$ and cosine leads to the elliptic function $\dn$. Gross considered the conditions such that an elliptic function $h(z)$ can be expressed as $h(z) = f(g(z))$ \cite{Gross1968}. Let $0 < k < 1$, $k' \coloneqq \sqrt{1 - k^2}$, and $q \coloneqq \enum^{-\pi \frac{ \EK(k') }{ \EK(k) }}$. Based on Jacobi's infinite product representations \eqref{eq:sn_ip} and \eqref{eq:cn_ip}, the functions
\begin{align}
\phi_k(u) &\coloneqq \sn(2 \EK(k) \arcsin(u) / \pi, k) = \frac{ 1 }{ \sqrt{k}  } \frac{ 2 q^\frac{1}{4} u \prod_{n = 1}^\infty (1 + 2 q^{2n} (2 u^2 - 1) + q^{4n}) }{ \prod_{n = 1}^\infty (1 + 2 q^{2n - 1} (2 u^2 - 1) + q^{4n - 2}) }, u \in \mathbb{C}, \\
\psi_k(u) &\coloneqq \cn( 2 \EK(k) \arccos(u) / \pi, k ) = \sqrt{ \frac{ k' }{ k } } \frac{ 2 q^\frac{1}{4} u \prod_{n = 1}^\infty (1 + 2 q^{2n} (2 u^2 - 1) + q^{4n}) }{ \prod_{n = 1}^\infty (1 - 2 q^{2n - 1} (2 u^2 -1) + q^{4n - 2}) }, u \in \mathbb{C},
\end{align}
are single-valued meromorphic functions on $\mathbb{C}$, one-to-one mappings of $[-1, 1]$ onto $[-1, 1]$, and satisfy $\phi_k(\sin z) = \sn(2 \EK(k) z / \pi, k)$ and $\psi_k(\cos z) = \cn(2 \EK(k) z / \pi, k)$ on $\mathbb{C}$. To evaluate an integral $I = \int_{-1}^1 \frac{ f(x) \del x }{ \sqrt{(1 - x^2) (1 - k^2 x^2)} }$, Goodrich and Stenger used the variable transformation $x = \phi_k(u)$ and obtained $I = \frac{2 \EK(k)}{\pi} \int_{-1}^1 \frac{ f(\phi_k(u)) }{ \sqrt{1 - u^2} } \del u$ \cite{Goodrich1970}. Then they applied the Gauss--Chebyshev quadrature and used the Achieser--Stenger bound by noting that $\phi_k(u)$ is the one-to-one conformal mapping of $\ellipse_{q^{-1/4}}$ onto the interior of the circle with radius $k^{-1/2}$~\cite{Schwarz1869}\cite[Sec.~13.4]{Kober1957}. To approximate a Markov function $f(z) \coloneqq \int_{-1}^1 \frac{w(t)}{z - t} \del t$ by a rational function $\sum_{\nu = 1}^M \frac{ c_\nu }{ z - t_\nu }$ on $(-\infty, -1/k] \cup [1/k, \infty)$ (resp. $\inum (-\infty, -k'/k] \cup \inum [k'/k, \infty)$), Barrett used the Gaussian quadrature with the variable transformation $t = \phi_k(u)$ (resp. $t = \psi_k(u))$, which is the one-to-one conformal mapping of $\ellipse_{q^{-1/2}}$ onto $\mathbb{C} \setminus ( (- \infty, - 1 / k] \cup [1/k, \infty) )$ (resp. $\mathbb{C} \setminus ( \inum (- \infty, - k' / k] \cup \inum [k'/k, \infty) )$) \cite{Barrett1971}. Applications of conformal mappings in quadrature methods are also found in \cite{Bakhvalov1967,Osipenko1989,Petras1998,Hale2008}\cite[Sec.~4.9]{Brass2011}. We also note the dn function played an important role for the derivation of Achieser's bound for the best uniform polynomial approximation error of an analytic function \cite{Achyeser1938}{\cite[Sec.~94 and 95]{Achieser1992}.

We introduced basis functions $\basis{\vt}_{r, n}(z)$ in Lemma \ref{lem:basis} to expand the approximation error of a finite completely monotonic function by the Gaussian quadrature with an analytic variable transformation $\vt_r(u)$ (Theorem \ref{the:f_gauss_basis}). As a byproduct, we obtained the basis expansion of a finite Laplace--Stieltjes transform (Lemma \ref{lem:f_basis} and Corollary \ref{cor:f_basis_error}). We showed the basis functions $\basis{\Phi}_{r, n}(z)$ associated with $\Phi_r(u)$ defined in Lemma \ref{lem:Phi_basis} are eigenfunctions of the fourth order differential operator $\mathcal{D}_r$ (Theorem \ref{the:L}). Although the operator $\mathcal{D}_r$ has a simple form among fourth order differential operators, we could not find the result in the literature. If we introduce
\begin{align}
&k \coloneqq \sqrt{1 - (a/b)^2}, q \coloneqq \enum^{- \pi \frac{\EK(a/b)}{\EK(k)}}, \\
&\basis{\Phi}_{a, b, n}(z) \coloneqq \basis{\Phi}_{a/b, n}(bz) = \frac{1}{\pi} \int_0^\pi \enum^{- bz \Phi_{a/b}(\cos\theta)} \cos(n \theta) \del \theta = \frac{1}{\pi} \int_0^\pi \enum^{- bz \dn(\EK(k) \theta / \pi, k) } \cos(n \theta) \del \theta, z \in \mathbb{C}, n \in \intge{0}, \\
&\mathcal{D}_{a, b} \coloneqq z^2 \frac{ \del^4 }{ \del z^4 } + 2 z \frac{ \del^3 }{ \del z^3 } - (a^2 + b^2) z^2 \frac{ \del^2 }{ \del z^2 } - (a^2 + b^2) z \frac{ \del }{ \del z } + a^2 b^2 z^2 = z \left( \frac{ \del^2 }{ \del z^2 } z \frac{ \del^2 }{ \del z^2 } - \frac{ \del }{ \del z } (a^2 + b^2) z \frac{ \del }{ \del z } + a^2 b^2 z \right),
\end{align}
by using Theorem \ref{the:L}, Theorem \ref{the:basis_Phi_orthogonal}, and Corollary \ref{cor:basis_Phi_integral_eq}, we have
\begin{align}
&\mathcal{D}_{a, b} \basis{\Phi}_{a, b, n}(z) = \frac{ b^2 n^2 \pi^2 }{ \EK(k)^2 } \basis{\Phi}_{a, b, n}(z), z \in \mathbb{C}, n \in \intge{0}, \\
&\int_0^\infty \basis{\Phi}_{a, b, m}(x) \basis{\Phi}_{a, b, n}(x) \frac{1}{x} \del x =
\begin{cases}
0, & m \neq n, \\
\frac{1}{n} \frac{ 1 }{ q^{-2n} - q^{2n} }, & m = n,
\end{cases}
m \in \intge{1}, n \in \intge{1}, \\
&\basis{\Phi}_{a, b, n}(z) = n \frac{ q^{-2n} - q^{2n} }{ 2 } \int_0^\infty \left( \basis{\Phi}_{a, b, 0}(z + y) - \basis{\Phi}_{a, b, 0}(z)  \basis{\Phi}_{a, b, 0}(y) \right) \basis{\Phi}_{a, b, n}(y) \frac{1}{y} \del y, z \in \mathbb{C}, n \in \intge{1}.
\end{align}
For a finite Laplace transform, Bertero and Gr\"unbaum \cite{Bertero1985} considered the singular system $\int_a^b \enum^{-xs} u_i(s) \del s = \alpha_i v_i(x)$ for $x_1 \le x \le x_2$ and $\int_{x_1}^{x_2} \enum^{-ty} v_i(y) \del y = \alpha_i u_i(t)$ for $a \le t \le b$, $i \in \intge{0}$. Functions $u_i(t)$ and $v_i(x)$ are called right and left singular functions, respectively, and satisfy orthogonality conditions $\int_a^b u_i(t) u_j(t) \del t = 0$ and $\int_{x_1}^{x_2} v_i(x) v_j(x) \del x = 0$ for $i \ne j$.  For $x_1 = 0$ and $x_2 = \infty$, they showed left singular functions $v_i(x)$ are eigenfunctions of the following fourth order differential operator
\begin{align}
\label{eq:Lfl}
&\mathcal{L}_{a, b} \coloneqq \frac{\del^2}{\del x^2} x^2 \frac{\del^2}{\del x^2} - (a^2 + b^2) \frac{\del}{\del x} x^2 \frac{\del}{\del x} + a^2 b^2 x^2 - 2 a^2 \\ \nonumber
&= x^2 \frac{\del^4}{\del x^4} + 4 x \frac{\del^3}{\del x^3} - ((a^2 + b^2) x^2 - 2) \frac{\del^2}{\del x^2} - 2 (a^2 + b^2) x \frac{\del}{\del x} + a^2 b^2 x^2 - 2 a^2.
\end{align}
Lederman and Rokhlin constructed efficient algorithms for right \cite{Lederman2015} and left \cite{Lederman2016} singular functions. de Villiers and Pike discussed the interlacing property of zeros of right singular functions \cite[Sec.~4.8.2]{Villiers2016}. Since $\mathcal{L}_{a, b}$ and $\mathcal{D}_{a, b}$ are similar, it may be interesting to consider the relationship between them.

Beylkin and Monz\'on investigated the relative error of the exponential (or Gaussian) sum approximation of the inverse power function \cite{Beylkin2005,Beylkin2010}. The approximation was used to compute the Coulomb interactions \cite{Shaw2014,Predescu2020,Shaw2021}. We developed a method named the tensor-structured multilevel Ewald summation method (TME) to compute Coulomb interactions \cite{Morimoto2021}. A splitting of the Coulomb potential $\frac{1}{r} = \frac{2}{\sqrt{\pi}} \int_0^\infty \enum^{-r^2 s^2} \del s = \frac{2}{\sqrt{\pi}} \int_\alpha^\infty \enum^{-r^2 s^2} \del s + \frac{2}{\sqrt{\pi}} \int_0^\alpha \enum^{-r^2 s^2} \del s$ \cite{Ewald1921} is called the Ewald splitting and $0 < \alpha < \infty$ is called the Ewald splitting parameter. In TME method, based on the Ewald splitting, the Coulomb potential was divided into multilevel potentials as follows:
\begin{equation}
\frac{1}{r} = \frac{2}{\sqrt{\pi}} \int_0^\infty \enum^{-r^2 s^2} \del s = \frac{2}{\sqrt{\pi}} \int_\alpha^\infty \enum^{-r^2 s^2} \del s + \sum_{l = 1}^L \frac{2}{\sqrt{\pi}} \int_{\alpha / 2^l}^{\alpha / 2^{l-1}} \enum^{-r^2 s^2} \del s + \frac{2}{\sqrt{\pi}} \int_0^{\alpha / 2^L} \enum^{-r^2 s^2} \del s, 0 < r < \infty,
\end{equation}
where $L \in \intge{1}$. By using Gauss--Legendre quadrature, the level $l$ potential $\frac{2}{\sqrt{\pi}} \int_{\alpha / 2^l}^{\alpha / 2^{l-1}} \enum^{-r^2 s^2} \del s$ was approximated by the $M$-term Gaussian sum. By using \eqref{eq:fab_1_2}, the level $l$ potential is expressed as
\begin{equation}
\frac{2}{\sqrt{\pi}} \int_{\alpha / 2^l}^{\alpha / 2^{l-1}} \enum^{-r^2 s^2} \del s = f_{1/2, \alpha^2 / 4^l, \alpha^2 / 4^{l-1}}(r^2), 0 \le r < \infty.
\end{equation}
Then the Gaussian sum approximation is interpreted as the exponential sum approximation of $f_{1/2, \alpha^2 / 4^l, \alpha^2 / 4^{l-1}}(x)$ with the variable transformation $\vt_{\pset_2, 1/4}(u)$. By using Theorem \ref{the:Evt} and \eqref{eq:fab_1_2_P2_gauss}, we have the error bound
\begin{equation}
0 < \max_{0 \le r < \infty} \left\lvert f_{1/2, \alpha^2 / 4^l, \alpha^2 / 4^{l-1}}(r^2) - \sum_{\nu = 1}^M \frac{ \alpha c_\nu^\mathrm{GL} }{ 2^l \sqrt{\pi} } \enum^{- \frac{\alpha^2}{4^{l-1}} \vt_{\pset_2, 1/4}(u_\nu^\mathrm{GL}) r^2} \right\rvert < \frac{16}{\pi} \hat{\rho}[\vt_{\pset_2, 1/4}]^{- 2M} \frac{\alpha}{2^{l-1} \sqrt{\pi}}, l = 1, \ldots, L.
\end{equation}
The fast decreasing rate $\hat{\rho}[\vt_{\pset_2, 1/4}]^{- 2M} \approx 17.9443^{-M}$ (Table \ref{tab:hatrho}) explains our previous numerical result, which implied the geometrical convergence with respect to the number of Gaussian functions (Fig.3 in \cite{Morimoto2021}).

Based on the current study, the splitting can be generalized to a wider class of potential energy functions and variable transformations as follows. Let $f(x)$ be completely monotonic in $0 < x < \infty$ and the integral representation is given as $f(x) = \int_0^\infty \enum^{-xt} \del W(t)$. Let $g(r)$ be a potential energy function expressed as $g(r) = f(r^2) = \int_0^\infty \enum^{-r^2 t} \del W(t)$ for $0 < r < \infty$. It is known that various potential energy functions or kernel functions can be represented in this form (e.g., $r^{-\kappa} = \int_0^\infty \enum^{-r^2 t} \frac{ t^{\kappa/2 - 1} }{\Gamma(\kappa/2)} \del t$, $\enum^{- \kappa r} = \int_0^\infty \enum^{-r^2 t} \frac{\kappa}{2 \sqrt{\pi} t^{3/2}} \enum^{- \frac{\kappa^2}{4t}} \del t$, and $\frac{\enum^{- \kappa r}}{r} = \int_0^\infty \enum^{-r^2 t} \frac{1}{\sqrt{\pi t}} \enum^{- \frac{\kappa^2}{4t}} \del t$ for a given $0 < \kappa < \infty$) \cite[Anhang. 2.]{Doetsch1937} \cite{Kutzelnigg1994,Harrison2003,Hackbusch2006,Mazars2011}. Then $g(r)$ can be split  as
\begin{equation}
g(r) = f(r^2) = \int_0^\infty \enum^{-r^2 t} \del W(t) = \int_{\alpha^2}^\infty \enum^{-r^2 t} \del W(t) + \sum_{l = 1}^L \int_{\alpha^2 / 4^l}^{\alpha^2 / 4^{l-1}} \enum^{-r^2 t} \del W(t) + \int_0^{\alpha^2 / 4^L} \enum^{-r^2 t} \del W(t), 0 < r < \infty.
\end{equation}
Let $M \in \intge{1}$, $\vt_{1/4} \in \vtset_{1/4}$, and $\rhoana[\vt_{1/4}] > 1$. If $W(t)$ has at least $M+1$ points of increase for $\alpha^2 / 4^l \le t \le \alpha^2 / 4^{l-1}$, by using Theorem \ref{the:Evt}, we have
\begin{equation}
0 < \max_{0 \le r < \infty} \left\lvert \int_{\alpha^2 / 4^l}^{\alpha^2 / 4^{l-1}} \enum^{-r^2 t} \del W(t) - \sum_{\nu = 1}^M c_{l, \nu} \enum^{- \frac{\alpha^2}{4^{l-1}} \vt_{1/4}(u_{l, \nu}) r^2} \right\rvert < \frac{16}{\pi} \hat{\rho}[\vt_{1/4}]^{- 2M} \int_{\alpha^2 / 4^l}^{\alpha^2 / 4^{l-1}} \del W(t), l = 1, \ldots, L,
\end{equation}
where we used $u_{l, \nu}$ and $c_{l, \nu}$ instead of $u_\nu$ and $c_\nu$ to represent $l$ dependency. If we use $\vt_{1/4}(u) = \Phi_{1/4}(u)$, we have $\hat{\rho}[\Phi_{1/4}]^{- 2M} \approx 35.8885^{-M}$ (Table \ref{tab:hatrho}). The fast decreasing rate will make it possible to use a few number of Gaussian functions to approximate the level $l$ potential function $\int_{\alpha^2 / 4^l}^{\alpha^2 / 4^{l-1}} \enum^{-r^2 t} \del W(t)$.

\section*{Acknowledgments}
The author would like to thank Makoto Taiji for providing the opportunity to conduct the current study. This work used computational resources of the HOKUSAI BigWaterfall system provided by RIKEN. This work was supported by JSPS KAKENHI Grant Number JP19H01107.

\appendix
\section{Properties of Stieltjes integral}
We listed the properties of Stieltjes integral used in the current study. If not specified, we assumed $-\infty < a < b < \infty$.
\begin{lemma}[{\cite[Lemma II.5]{Widder1941}}]
\label{lem:fl_entire}
Let $0 \le a \le b < \infty$ and $z \in \mathbb{C}$. Let $\alpha(t)$ be of bounded variation in $a \le t \le b$. Then $f(z) = \int_a^b \enum^{-zt} \del \alpha(t)$ is an entire function and satisfies $f^{(n)}(z) = (-1)^n \int_a^b \enum^{-zt} t^n \del \alpha(t)$, $n \in \intge{1}$.
\end{lemma}
\begin{theorem}[{\cite[Theorem I.7c]{Widder1941}}]
\label{the:stieltjes_pos}
Let $n \in \intge{0}$. If $\alpha(x)$ is nondecreasing with at least $n + 1$ points of increase, if $f(x)$ is continuous and nonnegative with at most $n$ zeros in $a \le x \le b$, then $\int_a^b f(x) \del \alpha(x) > 0$.
\end{theorem}
\begin{theorem}
\label{the:stieltjes_neg}
If $f(x)$ is continuous, $\alpha(x)$ is of bounded variation in $a \le x \le b$, and $\beta(y) = - \alpha(-y)$ in $-b \le y \le -a$, then $\int_a^b f(x) \del \alpha(x) = \int_{-b}^{-a} f(-y) \del \beta(y)$.
\end{theorem}
\begin{theorem}[{\cite[Theorem I.11a with the correction of typo]{Widder1941}}]
\label{the:stieltjes_change}
If $f(x)$ is continuous, $\alpha(x)$ is of bounded variation in $a \le x \le b$, and $\beta(y)$ is continuous increasing with no points of invariability in $c \le y \le d$, then $\int_a^b f(x) \del \alpha(x) = \int_c^d f(\beta(y)) \del \alpha(\beta(y))$, where $a = \beta(c), b = \beta(d)$.
\end{theorem}
\begin{theorem}[{\cite[Theorem I.6a]{Widder1941}}]
\label{the:stieltjes_riemann}
If $f(x)$ is continuous and $\varphi(x)$ is measurable in $a \le x \le b$ and satisfies $\int_a^b \lvert \varphi(x) \rvert \del x < \infty$, and if $\alpha(x) = \int_c^x \varphi(t) \del t, a \le c \le b, a \le x \le b,$ then $\int_a^b f(x) \del \alpha(x) = \int_a^b f(x) \varphi(x) \del x = \int_a^b f(x) \alpha'(x) \del x$.
\end{theorem}
\begin{theorem}[{\cite[Theorem I.6b]{Widder1941}}]
\label{the:int_fphi_da}
If $f(x)$ and $\vt(x)$ are continuous and $\alpha(x)$ is of bounded variation in $a \le x \le b$, and if $\beta(x) = \int_c^x \vt(t) \del \alpha(t), a \le x \le b, a \le c \le b$, then $\int_a^b f(x) \del \beta(x) = \int_a^b f(x) \vt(x) \del \alpha(x)$.
\end{theorem}
\begin{theorem}[{\cite[Theorem I.5b]{Widder1941}}]
\label{the:stieltjes_abs_ineq}
If $f(x)$ is continuous and $\alpha(x)$ is of bounded variation in $a \le x \le b$, then $\lvert \int_a^b f(x) \del \alpha(x) \rvert \le \max\limits_{a \le x \le b} \lvert f(x) \rvert \TV[\alpha(x)]_a^b$.
\end{theorem}
\begin{theorem}[{\cite[Theorem I.4b]{Widder1941}}]
\label{the:stieltjes_parts}
If $f(x)$ is of bounded variation and $\alpha(x)$ is continuous in $a \le x \le b$, then the Stieltjes integral of $f(x)$ with respect to $\alpha(x)$ from $a$ to $b$ exists and $\int_a^b f(x) \del \alpha(x) = f(b) \alpha(b) - f(a) \alpha(a) - \int_a^b \alpha(x) \del f(x)$.
\end{theorem}
\begin{theorem}[{\cite[Theorem I.5c]{Widder1941}}]
If $f(x)$ is continuous and $\alpha(x)$ is of bounded variation in $a \le x \le b$, then $F(x) = \int_a^x f(t) \del \alpha(t)$ is also of bounded variation in $a \le x \le b$ and satisfies $F(x+) - F(x) = f(x) (\alpha(x+) - \alpha(x))$ in $a \le x < b$ and $F(x) - F(x-) = f(x) (\alpha(x) - \alpha(x-))$ in $a < x  \le b$.
\end{theorem}
\begin{corollary}
\label{cor:int_f_da}
Let $n \in \intge{0}$. Let $f(x)$ be a positive continuous function in $a \le x \le b$. Let $\alpha(x)$ be a bounded nondecreasing function with at least $n$ points of increase in $a \le x \le b$. Then $F(x) = \int_a^x f(t) \del \alpha(t)$ is a bounded nondecreasing function with at least $n$ points of increase in $a \le x \le b$.
\end{corollary}
\begin{theorem}[{\cite[Theorem I.5d]{Widder1941}}]
\label{the:stieltjes_series}
If the functions $f_n(x)$, $n \in \intge{0}$, are continuous, $\alpha(x)$ is of bounded variation in $a \le x \le b$, and if the series $\sum_{n = 0}^\infty f_n(x)$ converges uniformly to $f(x)$ on that interval, then $\int_a^b f(x) \del \alpha(x) = \sum_{n = 0}^\infty \int_a^b f_n(x) \del \alpha(x)$.
\end{theorem}

\section{Properties and nomenclature of Gaussian quadrature}
\label{app:gauss}
We summarize properties and nomenclature of Gaussian quadrature used in the current study. We represent the set of real polynomials of degree at most $n \in \intge{0}$ by $\pset_n$. The following theorem is the basis of the Gaussian quadrature.
\begin{theorem}[{\cite[Sec.~2.2, 3.3, and 3.4]{Szego1975}}]
\label{the:gauss}
Let $M \in \intge{1}$ and $-\infty < a < b < \infty$. Let $\alpha(u)$ be a bounded nondecreasing function with at least $M + 1$ points of increase for $a \le u \le b$. Then there exists orthonormal polynomials $p_i(u)$, $i = 0, \ldots, M$, such that (i) each degree of $p_i(u)$ is $i$, (ii) $\int_a^b p_i(u) p_j(u) \del \alpha(u) = \delta_{i, j}$, $i = 0, \ldots, M$, $j = 0, \ldots, M$, and (iii) each polynomial $p_i(u)$ has $i$ simple zeros in $(a, b)$. Let $u_\nu$, $\nu = 1, \ldots, M$, be zeros of $p_M(u)$. Let us introduce Christoffel numbers $c_\nu \coloneqq \left( \sum_{i = 0}^{M-1} p_i(u_\nu)^2 \right)^{-1} > 0$, $\nu = 1, \ldots, M$. Then one has $\int_a^b p(u) \del \alpha(u) = \sum_{\nu = 1}^M c_\nu p(u_\nu)$, $p \in \pset_{2M-1}$.
\end{theorem}
The quadrature method to approximate an integral $\int_a^b F(u) \del \alpha(u)$ by $\sum_{\nu = 1}^M c_\nu F(u_\nu)$ with $u_\nu$ (called node) and $c_\nu$ (called weight), $\nu = 1, \ldots, M$, defined in Theorem \ref{the:gauss} is called the $M$-point Gaussian quadrature, which is characterized by the exactness of the polynomial integrand with at most $2M-1$ degree.

Next, we summarize error bounds of Gaussian quadrature used in the current study. We assume the conditions in Theorem \ref{the:gauss}. The bound
\begin{equation}
\label{eq:gauss_error_dist}
\left\lvert \int_a^b F(u) \del \alpha(u) - \sum_{\nu = 1}^M c_\nu F(u_\nu) \right\rvert \le 2 \dist(F, \pset_{2M-1}, [a, b]) \int_a^b \del \alpha(u)
\end{equation}
enables to estimate the error bound of the Gaussian quadrature based on the best uniform polynomial approximation error of the integrand $F(u)$ \cite{Bernstein1918}\cite[Eq.~4.8]{Gautschi1981}.
Several bounds of $\dist(F, \pset_n, [-1, 1])$ are listed in \cite[p.~102]{Brass2011}. Chebyshev obtained the following equality \cite[Sec.~37]{Achieser1992}
\begin{equation}
\label{eq:dist_inverse}
\dist\left( \frac{1}{ u - v }, \pset_n, [-1, 1] \right) = \frac{ (v + \sqrt{v^2 - 1})^{-n} }{ v^2 - 1 }, 1 < v < \infty, n \in \intge{0}.
\end{equation}
Based on the bound by Achieser \cite{Achyeser1938}{\cite[Sec.~94 and 95]{Achieser1992}} and \eqref{eq:gauss_error_dist}, Stenger obtained the following error bound of the Gaussian quadrature: 
\begin{theorem}[Achieser \cite{Achyeser1938}{\cite[Sec.~94 and 95]{Achieser1992}} and Stenger \cite{Stenger1966}]
\label{the:achieser_stenger}
Let $F(u)$ satisfy
\begin{subequations}
\label{eq:achieser_condition}
\begin{align}
\label{eq:achieser_analytic}
&\text{$F(u)$ is analytic on $\ellipse_\rho$ for some $1 < \rho < \infty$,} \\
\label{eq:achieser_real}
&\text{$F(u)$ is real on $[-1, 1]$,} \\
&0 < \sup_{u \in \ellipse_\rho} \lvert \RE F(u) \rvert < \infty.
\end{align}
\end{subequations}
Let us use the conditions in Theorem \ref{the:gauss} with $a = -1$ and $b = 1$. Then one has
\begin{equation}
\label{eq:achieser_stenger}
\left\lvert \int_{-1}^1 F(u) \del \alpha(u) - \sum_{\nu = 1}^M c_\nu F(u_\nu) \right\rvert < \frac{16}{\pi} \rho^{-2M} \sup_{u \in \ellipse_\rho} \lvert \RE F(u) \rvert \int_{-1}^1 \del \alpha(u).
\end{equation}
\end{theorem}
We refer the bound \eqref{eq:achieser_stenger} as the Achieser--Stenger bound. We used the condition \eqref{eq:achieser_real} instead of ``$F(u)$ is real on the real axis of $\ellipse_\rho$'' based on the following result:
\begin{lemma}
\label{lem:ellipse_real}
Let $1 < \rho \le \infty$. If $F(u)$ is analytic on $\ellipse_\rho$ and real on $[-1, 1]$, then $F(u)$ is real on $(-(\rho + \rho^{-1}) / 2, (\rho + \rho^{-1}) / 2)$.
\end{lemma}
\begin{proof}
By Theorem \ref{the:laurent}, we have $F((s + s^{-1})/2) = a_0 +  \sum_{n =1}^\infty a_n (s^n + s^{-n})$ for $s \in \aset_\rho$, where $a_n \coloneqq \frac{1}{\pi} \int_0^\pi F(\cos\theta) \cos(n \theta) \del \theta$ for $n \in \intge{0}$. Since $F(\cos\theta)$ is real for $0 \le \theta \le \pi$, $a_n$ is real. Then $F(\pm (x+x^{-1})/2)$ is real for $\rho^{-1} < x < \rho$.
\end{proof}

\section{Expressions of $\hat{\rho}[\vt_r]$ for some variable transformations}
\label{app:hatrho}

\subsection{Linear polynomial}
We consider the linear polynomial $\vt_r(u) = a_1 u + a_0$. The condition \eqref{eq:vt_eq} determines the parameters as
\begin{equation}
\label{eq:vtP1}
\vt_r(u) = \frac{1 - r}{2} u + \frac{1 + r}{2} \eqqcolon \vt_{\pset_1, r}(u), u \in \mathbb{C},
\end{equation}
which satisfies \eqref{eq:vt_increase}, $\rhoana[\vt_r] = \infty$, and $\inf\limits_{u \in \ellipse_{\rhoana[\vt_r]}} \RE \vt_r(u) = -\infty$. We have
\begin{equation}
\label{eq:vtP1_infre}
\inf_{u \in \ellipse_\rho} \RE \vt_{\pset_1, r}(u) = - \frac{1 - r}{2} \frac{\rho + \rho^{-1}}{2} + \frac{1 + r}{2}, 1 < \rho < \infty,
\end{equation}
and by solving $\inf\limits_{u \in \ellipse_{\rho_0[\vt_r]}} \RE \vt_r(u) = 0$ for $1 < \rho_0[\vt_r] < \infty$, we have
\begin{equation}
\label{eq:hatrho_P1}
\hat{\rho}[\vt_{\pset_1, r}] = \rho_0[\vt_{\pset_1, r}] = \frac{1 + \sqrt{r}}{1 - \sqrt{r}}.
\end{equation}
Let $\minfre{\vt}_{\pset_1, r}(X)$ be a function $\minfre{\vt}_r(X)$ defined in \eqref{eq:minfre} with $\vt_r = \vt_{\pset_1, r}$. By using \eqref{eq:vtP1_infre}, we have
\begin{equation}
\label{eq:minfre_P1}
\minfre{\vt}_{\pset_1, r}(X) = \frac{1 - r}{2} \frac{\enum^X + \enum^{-X}}{2} - \frac{1 + r}{2}, 0 < X < \infty.
\end{equation}

Let $n \in \intge{0}$ and $z \in \mathbb{C}$. By using the modified Bessel function $I_n(z) = \frac{1}{\pi} \int_0^\pi \enum^{ z \cos\theta } \cos(n \theta) \del \theta$ \cite[Eq.~10.32.3]{DLMF}, \eqref{eq:vtP1}, and $I_n(-z) = (-1)^n I_n(z)$, the basis functions associated with $\vt_{\pset_1, r}(u)$ can be expressed as
\begin{equation}
\label{eq:basis_P1}
\basis{\vt}_{\pset_1, r, n}(z) \coloneqq \frac{1}{\pi} \int_0^\pi \enum^{ -z \vt_{\pset_1, r}(\cos\theta)  } \cos(n \theta) \del \theta = (-1)^n \enum^{ - \frac{1 + r}{2} z } I_n\left( \frac{1 - r}{2} z \right).
\end{equation}
Since the inequality $I_n(x) > 0$ holds for $0 < x < \infty$ \cite[Sec.~10.37]{DLMF}, we have
\begin{equation}
\label{eq:basis_P1_ineq}
(-1)^n \basis{\vt}_{\pset_1, r, n}(x) > 0, n \in \intge{0}, 0 < x < \infty.
\end{equation}
We can show the following equalities:
\begin{align}
\label{eq:W0_P1}
&W_{\vt_{\pset_1, r}, 0}(\tau) = \frac{ \arccos\left( \frac{ - 2 \tau + 1 + r }{ 1 - r } \right) }{\pi} = \int_r^\tau \frac{\del t}{ \pi \sqrt{ (1 - t) (t - r) } }, r \le \tau \le 1, \\
&\basis{\vt}_{\pset_1, r, n}(z)= \int_r^1 \enum^{ -z \tau } \frac{ T_n\left( \frac{ 2 \tau }{ 1 - r } - \frac{ 1 + r }{ 1 - r } \right) }{ \pi \sqrt{ (1 - \tau) (\tau - r) } } \del \tau, n \in \intge{0}, z \in \mathbb{C}, \\
&\basis{\vt}_{\pset_1, r, 0}'(0) = - (1 + r) / 2, \basis{\vt}_{\pset_1, r, 1}'(0) = - (1 - r) / 4, \basis{\vt}_{\pset_1, r, n}'(0) = 0, n \in \intge{2}, \\
&\basis{\vt}_{\pset_1, r, 0}''(0) = (3 + 2 r + 3 r^2) / 8, \basis{\vt}_{\pset_1, r, 1}''(0) = (1 - r^2) / 4, \\ \nonumber
&\basis{\vt}_{\pset_1, r, 2}''(0) = (1 - r)^2 / 16, \basis{\vt}_{\pset_1, r, n}''(0) = 0, n \in \intge{3}, \\
\label{eq:V_P1}
&V_{\vt_{\pset_1, r}} = \basis{\vt}_{\pset_1, r, 0}''(0) - \basis{\vt}_{\pset_1, r, 0}'(0)^2 = (1 - r)^2 / 8.
\end{align}

\subsection{Quadratic polynomial maximizing $\hat{\rho}[\vt_r]$}
We consider the quadratic polynomial $\vt_r(u) = a_2 u^2 + a_1 u + a_0$. The condition \eqref{eq:vt_eq} leads to
\begin{equation}
\label{eq:vtquadratic}
\vt_r(u) = a_2 u^2 + \frac{1 - r}{2} u + \frac{1 + r}{2} - a_2, u \in \mathbb{C}, a_2 \in \mathbb{R}.
\end{equation}
The condition \eqref{eq:vt_increase} is satisfied for
\begin{equation}
\label{eq:quadraticineq}
- (1 - r) / 4 \le a_2 \le (1 - r) / 4.
\end{equation}
Although $\vt_r(u)$ is a linear polynomial for $a_2 = 0$, we do not exclude it in the following analysis. Under these conditions, we have $\rhoana[\vt_r] = \infty$ and $\inf\limits_{u \in \ellipse_{\rhoana[\vt_r]}} \RE \vt_r(u) = -\infty$. Then we find the maximum value of $\rho_0[\vt_r]$ satisfying \eqref{eq:vtquadratic} and \eqref{eq:quadraticineq}. We note that $\rho_0[\vt_r]$ in Lemma \ref{lem:rho0} can be found by following two steps: (i) To find the level set of $\RE \vt_r(x + \inum y) = 0$. (ii) To find $\rho$ such that $\bd{\ellipse_\rho}$ contacts with the level set for the first time by increasing $\rho > 1$. For \eqref{eq:vtquadratic}, we have
\begin{equation}
\label{eq:Revtquadratic}
\RE \vt_r(x + \inum y) = a_2 (x^2 - y^2) + \frac{1 - r}{2} x + \frac{1 + r}{2} - a_2.
\end{equation}
Equalities $\RE \vt_r(x_r \pm \inum y_r) = 0$ hold for $x_r \coloneqq - \frac{1 + r}{1 - r} < -1$, $y_r \coloneqq \frac{2 \sqrt{r}}{1 - r} > 0$, which are independent of $a_2$. We introduce $\rho_r > 1$ such that $\bd{\ellipse_{\rho_r}}$ passes through two points $(x_r, \pm y_r)$. The condition leads to the equation $x_r^2 / A_r^2 + y_r^2 / B_r^2 = 1$, where we introduced $A_r \coloneqq (\rho_r + \rho_r^{-1}) / 2 > 1$ and $B_r \coloneqq (\rho_r - \rho_r^{-1}) / 2 > 0$. By solving the equation with respect to $A_r > 1$ or $B_r > 0$, we have $A_r = \sqrt{1 + r} / (1 - \sqrt{r})$ and $B_r = \sqrt{2 \sqrt{r}} / (1 - \sqrt{r})$. By solving
$A_r = (\rho_r + \rho_r^{-1}) / 2$ with respect to $\rho_r > 1$, we have $\rho_r = (\sqrt{1 + r} + \sqrt{2 \sqrt{r}}) / (1 - \sqrt{r})$. Since $\RE \vt_r(x_r \pm \inum y_r) = 0$ and $x_r \pm \inum y_r \in \bd{\ellipse_{\rho_r}}$ hold regardless of the value of $a_2$, $\rho_r$ gives an upper bound as $\rho_0[\vt_r] \le \rho_r$.

We can show that the equality $\rho_0[\vt_r] = \rho_r$ holds for some $a_2$ as follows. The condition, the level set of $\RE \vt_r(x + \inum y) = 0$ and $\bd{\ellipse_{\rho_r}}$ contact at $(x_r, y_r)$, leads to the equation
\begin{equation}
\label{eq:contact}
\nabla \RE \vt_r(x + \inum y) \vert_{x = x_r, y = y_r} = - \alpha \nabla (x^2 / A_r^2 + y^2 / B_r^2 - 1) \vert_{x = x_r, y = y_r}, \alpha > 0.
\end{equation}
The equation \eqref{eq:contact} has a unique solution $\alpha = \sqrt{r} / 2$ and $a_2 = ((1 - \sqrt{r}) / 2)^2$, which satisfies the condition \eqref{eq:quadraticineq}. In this case, \eqref{eq:vtquadratic} is expressed as
\begin{equation}
\label{eq:vtP2}
\vt_r(u) = \left( \frac{1 - \sqrt{r}}{2} u + \frac{1 + \sqrt{r}}{2} \right)^2 \eqqcolon \vt_{\pset_2, r}(u), u \in \mathbb{C}.
\end{equation}
The level set of $\RE \vt_{\pset_2, r}(x + \inum y) = 0$ is two lines $y = \pm (x + (1 + \sqrt{r}) / (1 - \sqrt{r}))$. We can confirm these two lines contact with $\bd{\ellipse_{\rho_r}}$ at $(x_r, \pm y_r)$. In summary, \eqref{eq:vtP2} gives the maximum value of $\rho_0[\vt_r]$  satisfying \eqref{eq:vtquadratic} and \eqref{eq:quadraticineq} and we have
\begin{equation}
\label{eq:hatrho_P2}
\hat{\rho}[\vt_{\pset_2, r}] = \rho_0[\vt_{\pset_2, r}] = \rho_r = \frac{\sqrt{1 + r} + \sqrt{2 \sqrt{r}}}{1 - \sqrt{r}}.
\end{equation}

\subsection{Exponential function}
We consider the exponential function $\vt_r(u) = \enum^{a_1 u + a_0}$. The condition \eqref{eq:vt_eq} determines the parameters as
\begin{equation}
\label{eq:vtexp}
\vt_r(u) = \sqrt{r} \enum^{\frac{1}{2} \log(1 / r) u} = r^\frac{ 1 - u }{ 2 } \eqqcolon \vt_{\exp, r}(u), u \in \mathbb{C},
\end{equation}
which satisfies the condition \eqref{eq:vt_increase}, $\rhoana[\vt_r] = \infty$, $\inf\limits_{u \in \ellipse_{\rhoana[\vt_r]}} \RE \vt_r(u) = -\infty$, and
\begin{subequations}
\label{eq:vtexp_property}
\begin{align}
&\vt_{\exp, r}^{(n)}(u) > 0, n \in \intge{0}, u \in \mathbb{R}, \\
&\vt_{\exp, r}(-u) = r / \vt_{\exp, r}(u), u \in \mathbb{C}.
\end{align}
\end{subequations}
By using $\RE \vt_r(x + \inum y) = \sqrt{r} \exp\left(\frac{1}{2} \log(1 / r) x \right) \cos\left( \frac{1}{2} \log(1 / r) y \right)$, the level set of $\RE \vt_r(x + \inum y) = 0$ is horizontal lines $y = (2 \pi n - \pi) / \log(1 / r), n \in \mathbb{Z}$. Thus, by increasing $\rho > 1$, $\bd{\ellipse_\rho}$ contacts with the level set for the first time at $x = 0$ and $y = \pm \pi / \log(1 / r)$. Since $y = \pm (\rho_0[\vt_r] - \rho_0[\vt_r]^{-1}) / 2$ for $x = 0$ on $\bd{\ellipse_{\rho_0[\vt_r]}}$, by solving $(\rho_0[\vt_r] - \rho_0[\vt_r]^{-1}) / 2 = \pi / \log(1 / r)$ with respect to $\rho_0[\vt_r] > 1$, we have
\begin{equation}
\label{eq:hatrho_exp}
\hat{\rho}[\vt_{\exp, r}] = \rho_0[\vt_{\exp, r}] = \frac{\pi}{\log(1 / r)} + \sqrt{\left( \frac{\pi}{\log(1 / r)} \right)^2 + 1}.
\end{equation}

\subsection{A rational function}
We consider a rational function $\vt_r(u) = A / (u + B)$. The condition \eqref{eq:vt_eq} determines the parameters as
\begin{equation}
\label{eq:vtR01}
\vt_r(u) = \frac{ 2 r }{ 1 - r } \frac{ 1 }{ \frac{ 1 + r }{ 1 - r } - u } \eqqcolon \vt_{\rset_{0, 1}, r}(u), u \in \mathbb{C},
\end{equation}
which satisfies the condition \eqref{eq:vt_increase} and
\begin{equation}
\vt_{\rset_{0, 1}, r}^{(n)}(u) > 0, n \in \intge{0}, - \infty < u < \frac{ 1 + r }{ 1 - r }.
\end{equation}
Since $\vt_r(u)$ has a simple pole at $u = \frac{ 1 + r }{ 1 - r } > 1$, by solving $\frac{ \rhoana[\vt_r] + \rhoana[\vt_r]^{-1} }{ 2 } = \frac{ 1 + r }{ 1 - r }$ for $1 < \rhoana[\vt_r] < \infty$, we have $\rhoana[\vt_r] = \frac{ 1 + \sqrt{r} }{ 1 - \sqrt{r} }$ and
\begin{align}
\label{eq:infRe_R01}
&\inf_{ u \in \ellipse_{\rhoana[\vt_r]} } \RE \vt_r(u) = \inf_{ u \in \ellipse_{\rhoana[\vt_r]} } \RE \frac{ 2 r }{ 1 - r } \frac{ 1 }{ \frac{ 1 + r }{ 1 - r } - u } \\ \nonumber
&= \frac{ 2 r }{ 1 - r } \min_{ 0 < \theta \le \pi } \frac{ \frac{ 1 + r }{ 1 - r } - \frac{ \rhoana[\vt_r] + \rhoana[\vt_r]^{-1} }{ 2 } \cos\theta }{ \left( \frac{ 1 + r }{ 1 - r } - \frac{ \rhoana[\vt_r] + \rhoana[\vt_r]^{-1} }{ 2 } \cos\theta \right)^2 + \left( \frac{ \rhoana[\vt_r] - \rhoana[\vt_r]^{-1} }{ 2 } \sin\theta \right)^2 } = \frac{ 2 r }{ 1 - r } \min_{ 0 < \theta \le \pi } \frac{ \frac{ 1 + r }{ 1 - r } }{ \left( \frac{ 1 + r }{ 1 - r } \right)^2 + \left( \frac{ 2 \sqrt{r} }{ 1 - r } \right)^2 - \cos\theta } = \frac{ r }{ 1 + r } > 0.
\end{align}
Then we have
\begin{equation}
\label{eq:hatrho_R01}
\hat{\rho}[\vt_{\rset_{0, 1}, r}] = \rhoana[\vt_{\rset_{0, 1}, r}] = \frac{ 1 + \sqrt{r} }{ 1 - \sqrt{r} }.
\end{equation}

\section{Some equalities and inequalities related to Jacobi's elliptic functions}
\begin{lemma}
Let $0 < k < 1$ and $k' \coloneqq \sqrt{1 - k^2}$. Then one has
\begin{equation}
\label{eq:dn_K_2_iKrl}
\dn(\EK(k)/2 + \inum \EK(k') \lambda, k) = \sqrt{k'} \enum^{ - \inum \left( \am\left( \EK\left( \frac{ 2 \sqrt{k'} }{ 1 + k' } \right) (\lambda + 1), \frac{ 2 \sqrt{k'} }{ 1 + k' } \right) - \frac{\pi}{2} \right) }, \lambda \in \mathbb{R}.
\end{equation}
\end{lemma}
\begin{proof}
Let $x \in \mathbb{R}$, $y \in \mathbb{R}$, and $z \in \mathbb{C}$.
By using equalities \cite[Eq.~16.21.4]{MilneThomson1972}\cite[Table 22.5.2]{DLMF}
\begin{align}
\label{eq:dn_x_iy}
&\dn(x + \inum y, k) = \frac{ \dn(x, k) \cn(y, k') \dn(y, k') - \inum k^2 \sn(x, k) \cn(x, k) \sn(y, k') }{ \cn(y, k')^2 + k^2 \sn(x, k)^2 \sn(y, k')^2 }, \\
&\sn(\EK(k)/2, k) = \frac{1}{\sqrt{1+k'}}, \cn(\EK(k)/2, k) = \sqrt{ \frac{ k' }{ 1 + k' } }, \dn(\EK(k)/2, k) = \sqrt{k'},
\end{align}
we have
\begin{equation}
\label{eq:dn_K_2_iy}
\dn(\EK(k)/2 + \inum y, k) = \sqrt{k'} \frac{ \cn(y, k') \dn(y, k') - \inum (1 - k') \sn(y, k') }{ 1 - k' \sn(y, k')^2 }.
\end{equation}
By using descending Landen transformations from $k'_1$ to $k'$ \cite[Sec.~19.8(ii), 22.7(i)]{DLMF}
\begin{subequations}
\begin{align}
\label{eq:k1}
&k'_1 \coloneqq \frac{ 2 \sqrt{k'} }{ 1 + k' }, k_1 \coloneqq \sqrt{1 - (k'_1)^2} = \frac{ 1 - k' }{ 1 + k' }, \EK(k'_1) = (1 + k') \EK(k'), \\
&\sn(z, k'_1) = \frac{ (1 + k') \sn(z/ (1 + k'), k') }{ 1 + k' \sn(z / (1 + k'), k')^2 }, \\
&\cn(z, k'_1) = \frac{ \cn(z / (1 + k'), k') \dn(z / (1 + k'), k') }{ 1 + k' \sn(z / (1 + k'), k')^2 }, \\
&\dn(z, k'_1) = \frac{ \dn(z / (1 + k'), k')^2 - (1 - k') }{ 1 + k' - \dn(z / (1 + k'), k')^2 }
 = \frac{ 1 - k' \sn(z / (1 + k'), k')^2 }{ 1 + k' \sn(z / (1 + k'), k')^2 },
\end{align}
\end{subequations}
\eqref{eq:dn_K_2_iy} is expressed as
\begin{equation}
\label{eq:dn_K_2_iy_cd}
\dn(\EK(k)/2 + \inum y, k) = \sqrt{k'} ( \cd((1 + k') y, k'_1) - \inum k_1 \sd((1 + k') y, k'_1) ).
\end{equation}
By using equalities \cite[Table 22.4.3]{DLMF}
\begin{subequations}
\begin{align}
&\cd(z, k'_1) = \sn(z + \EK(k'_1), k'_1) = \sin(\am(z + \EK(k'_1), k'_1)), \\
&- k_1 \sd(z, k'_1) = \cn(z + \EK(k'_1), k'_1) = \cos(\am(z + \EK(k'_1), k'_1)),
\end{align}
\end{subequations}
\eqref{eq:dn_K_2_iy_cd} is expressed as
\begin{equation}
\label{eq:dn_K_2_iy_am}
\dn(\EK(k)/2 + \inum y, k) = \sqrt{k'} \enum^{ - \inum \left( \am( (1 + k') y + \EK(k'_1), k'_1 ) - \pi/2 \right) }.
\end{equation}
By substituting $y = \EK(k') \lambda, \lambda \in \mathbb{R}$, for \eqref{eq:dn_K_2_iy_am} and using \eqref{eq:k1}, we have \eqref{eq:dn_K_2_iKrl}.
\end{proof}

\begin{lemma}
Let $0 < k < 1$ and $k' \coloneqq \sqrt{1 - k^2}$. Let $\sqrt{z}$ be the principal value of the square root function on $\mathbb{C} \setminus (-\infty, 0]$. Then one has
\begin{equation}
\label{eq:dnxmiKr_2}
\dn( \EK(k) \theta / \pi - \inum \EK(k')/2, k) = \sqrt{ 1 + k \enum^{ \inum 2 \am\left( \EK\left( \frac{ 2 \sqrt{k} }{ 1 + k } \right) \frac{ \theta }{ \pi }, \frac{ 2 \sqrt{k} }{ 1 + k } \right) } }, \theta \in \mathbb{R}.
\end{equation}
\end{lemma}
\begin{proof}
Let $\lambda \in \mathbb{R}$, $\theta \in \mathbb{R}$, and $z \in \mathbb{C}$. By interchanging $k$ and $k'$ in \eqref{eq:dn_K_2_iKrl} and using $\dn(\inum z, k') = \dc(z, k)$ \cite[Table 22.6.1]{DLMF}, we have
\begin{equation}
\label{eq:dcKlmiK_2}
\dc(\EK(k) \lambda - \inum \EK(k')/2, k) = \inum \sqrt{k} \enum^{ - \inum \am\left( \EK\left( \frac{ 2 \sqrt{k} }{ 1 + k } \right) (\lambda + 1), \frac{ 2 \sqrt{k} }{ 1 + k } \right) }.
\end{equation}
By substituting $\lambda = \theta / \pi - 1$ for \eqref{eq:dcKlmiK_2} and using $\dc(z - \EK(k), k) = 1 / \sn(z, k)$ \cite[Table 22.4.3]{DLMF}, we have
\begin{equation}
\label{eq:sksn}
\sqrt{k} \sn( \EK(k) \theta / \pi - \inum \EK(k')/2, k) = - \inum \enum^{ \inum \am\left( \EK\left( \frac{ 2 \sqrt{k} }{ 1 + k } \right) \frac{ \theta }{ \pi }, \frac{ 2 \sqrt{k} }{ 1 + k } \right) }.
\end{equation}
By squaring \eqref{eq:sksn} and using \eqref{eq:dn2}, we have
\begin{equation}
\label{eq:dn2_1ke}
\dn( \EK(k) \theta / \pi - \inum \EK(k')/2, k)^2 = 1 + k \enum^{ \inum 2 \am\left( \EK\left( \frac{ 2 \sqrt{k} }{ 1 + k } \right) \frac{ \theta }{ \pi }, \frac{ 2 \sqrt{k} }{ 1 + k } \right) }.
\end{equation}
By using \eqref{eq:dneven}, \eqref{eq:dnK}, and Theorem \ref{the:dn_conformal}, we have $\RE( \dn( \EK(k) \theta / \pi - \inum \EK(k')/2, k) ) > 0$. With this inequality and \eqref{eq:dn2_1ke}, we have \eqref{eq:dnxmiKr_2}.
\end{proof}

\begin{theorem}[Jacobi {\cite[Sec.~36]{Jacobi1829}}]
Let $0 < k < 1$, $k' \coloneqq \sqrt{1 - k^2}$, and $q \coloneqq \enum^{- \pi \frac{\EK(k')}{\EK(k)} }$. For $z \in \mathbb{C}$, one has
\begin{align}
\label{eq:sn_ip}
\sn(2 \EK(k) z / \pi, k) &= \frac{1}{\sqrt{k}} \frac{ 2 q^\frac{1}{4} \sin z \prod_{n = 1}^\infty (1 - 2 q^{2n} \cos(2 z) + q^{4n}) }{ \prod_{n = 1}^\infty (1 - 2 q^{2n - 1} \cos(2 z) + q^{4n - 2}) }, \\
\label{eq:cn_ip}
\cn(2 \EK(k) z / \pi, k) &= \sqrt{ \frac{k'}{k} } \frac{ 2 q^\frac{1}{4} \cos z \prod_{n = 1}^\infty (1 + 2 q^{2n} \cos(2 z) + q^{4n}) }{ \prod_{n = 1}^\infty (1 - 2 q^{2n - 1} \cos(2 z) + q^{4n - 2}) }, \\
\label{eq:dn_ip}
\dn(2 \EK(k) z / \pi, k) &= \sqrt{k'} \frac{\prod_{n = 1}^\infty (1 + 2 q^{2n-1} \cos(2z) + q^{4n-2})} {\prod_{n = 1}^\infty (1 - 2 q^{2n-1} \cos(2z) + q^{4n-2})}.
\end{align}
\end{theorem}
\begin{corollary}
Let $0 < k < 1$, $k' \coloneqq \sqrt{1 - k^2}$, and $q \coloneqq \enum^{- \pi \frac{\EK(k')}{\EK(k)} }$. For $z \in \mathbb{C}$, one has
\begin{equation}
\label{eq:sncn_prod}
\sn(\EK(k) z / \pi, k) \cn(\EK(k) z / \pi, k) = \frac{ 2 \sqrt{k' q} }{ k } \frac{ \sin z \prod_{n = 1}^\infty (1 - 2 q^{2n} \cos z + q^{4n}) (1 + 2 q^{2n} \cos z + q^{4n}) } { \prod_{n = 1}^\infty (1 - 2 q^{2n-1} \cos z + q^{4n-2})^2 }.
\end{equation}
\end{corollary}

\begin{theorem}[Jacobi {\cite[Sec.~36]{Jacobi1829}}]
Let $0 < k < 1$, $k' \coloneqq \sqrt{1 - k^2}$, and $q \coloneqq \enum^{- \pi \frac{\EK(k')}{\EK(k)} }$. Then one has
\begin{align}
\label{eq:prod_q2nm1}
&\prod_{n = 1}^\infty (1 - q^{2n - 1})^6 = \frac{ 2 q^\frac{1}{4} k' }{ \sqrt{k} }, \prod_{n = 1}^\infty (1 + q^{2n - 1})^6 = \frac{ 2 q^\frac{1}{4} }{ \sqrt{k k'} }, \\
\label{eq:prod_q2n}
&\prod_{n = 1}^\infty (1 - q^{2n})^6 = \frac{ 2 k k' \EK(k)^3 }{ \pi^3 \sqrt{q} }, \prod_{n = 1}^\infty (1 + q^{2n})^6 = \frac{ k }{ 4 \sqrt{k' q} }.
\end{align}
\end{theorem}

\begin{theorem}[Jacobi {\cite[Sec.~64]{Jacobi1829}}]
For $0 < q < 1$ and $z \in \mathbb{C}$, one has
\begin{align}
\label{eq:prod_q2nm1_cos}
&\prod_{n = 1}^\infty (1 - 2 q^{2n - 1} \cos z + q^{4n - 2}) = \frac{ 1 + 2 \sum_{n = 1}^\infty (-1)^n q^{n^2} \cos(nz)  }{ \prod_{n = 1}^\infty (1 - q^{2n}) }, \\
\label{eq:prod_q2n_cos}
&\sin(z / 2) \prod_{n = 1}^\infty (1 - 2 q^{2n} \cos z + q^{4n}) = \frac{ \sum_{n = 0}^\infty (-1)^n q^{ n^2 + n } \sin( (n + 1 / 2) z ) }{ \prod_{n = 1}^\infty (1 - q^{2n}) }.
\end{align}
\end{theorem}
\begin{corollary}
Let $0 < k < 1$, $k' \coloneqq \sqrt{1 - k^2}$, and $q \coloneqq \enum^{- \pi \frac{\EK(k')}{\EK(k)} }$. Then one has
\begin{align}
\label{eq:prod_Tn}
&\prod_{n = 1}^\infty (1 + 2 q^{2n - 1} u + q^{4n - 2}) = \frac{ 1 + 2 \sum_{n = 1}^\infty q^{n^2} T_n(u)  }{ \prod_{n = 1}^\infty (1 - q^{2n}) }, u \in \mathbb{C}, \\
\label{eq:prod_Vn}
&\prod_{n = 1}^\infty (1 + 2 q^{2n} u + q^{4n}) = \frac{ \sum_{n = 0}^\infty q^{ n^2 + n } V_n(u) }{ \prod_{n = 1}^\infty (1 - q^{2n}) }, u \in \mathbb{C}, \\
\label{eq:prod_Vn2}
&\prod_{n = 1}^\infty (1 - 2 q^{2n} u + q^{4n}) (1 + 2 q^{2n} u + q^{4n}) = \frac{ \sum_{n = 0}^\infty q^{ 2 n^2 + 2 n } V_n(1 - 2 u^2) }{ \prod_{n = 1}^\infty (1 - q^{4n}) } \\ \nonumber
&= (k')^\frac{1}{4} \sqrt{\frac{ 2 \EK(k) }{ \pi }} \frac{ \sum_{n = 0}^\infty q^{ 2 n^2 + 2 n } V_n(1 - 2 u^2) }{ \prod_{n = 1}^\infty (1 - q^{2n})^2 }, u \in \mathbb{C}, \\
\label{eq:qn2Tn_ineq}
&1 + 2 \sum_{n = 1}^\infty q^{n^2} T_n(u) \ge \sqrt{2 k' \EK(k) / \pi}, -1 \le u \le 1, \\
\label{eq:q2n22nVn_ineq}
&\sum_{n = 0}^\infty q^{ 2 n^2 + 2 n } V_n(1 - 2 u^2) \ge \frac{ k (k')^\frac{1}{4} (2 \EK(k) / \pi)^\frac{3}{2} }{ 4 \sqrt{q} }, -1 \le u \le 1.
\end{align}
\end{corollary}
\begin{proof}
Let $z \in \mathbb{C}$. \eqref{eq:prod_Tn} follows from \eqref{eq:prod_q2nm1_cos}, $(-1)^n \cos(n z) = T_n(-\cos z)$, and $u = - \cos z$. By using \eqref{eq:prod_q2n_cos} and $(-1)^n \sin((n + 1/2) z) = \sin(z / 2) V_n(- \cos z)$, we have
\begin{align}
\label{eq:q2n_cos}
&\sin(z / 2) \prod_{n = 1}^\infty (1 - 2 q^{2n} \cos z + q^{4n}) = \sin(z / 2) \frac{ \sum_{n = 0}^\infty q^{ n^2 + n } V_n(- \cos z) }{ \prod_{n = 1}^\infty (1 - q^{2n}) }.
\end{align}
If $\sin(z / 2) \ne 0$, i.e., $z \in \mathbb{C} \setminus 2 \pi \mathbb{Z}$, by dividing both sides of \eqref{eq:q2n_cos} by $\sin(z / 2)$ and substituting $u = - \cos(z) \in \mathbb{C} \setminus \{ -1 \}$, we have \eqref{eq:prod_Vn} for $u \in \mathbb{C} \setminus \{ -1 \}$. Since both sides in \eqref{eq:prod_Vn} are entire functions and continuous at $u = -1$, the equality also holds at $u = -1$. The first equality in \eqref{eq:prod_Vn2} follows from $(1 - 2 q^{2n} u + q^{4n}) (1 + 2 q^{2n} u + q^{4n}) = 1 + 2 q^{4n} (1 - 2 u ^2) + q^{8n}$ and \eqref{eq:prod_Vn}. The second equality in \eqref{eq:prod_Vn2} follows from \eqref{eq:prod_q2n}.

Let $-1 \le u \le 1$. Let $F(u) \coloneqq 1 + 2 \sum_{n = 1}^\infty q^{n^2} T_n(u)$. Since $1 + 2 q^{2n - 1} u + q^{4n - 2}$ is positive and minimized at $u = -1$, by \eqref{eq:prod_Tn}, we have $F(u) \ge F(-1) = \prod_{n = 1}^\infty (1 - q^{2n - 1})^2 \prod_{n = 1}^\infty (1 - q^{2n})$. Then by using \eqref{eq:prod_q2nm1} and \eqref{eq:prod_q2n}, we have \eqref{eq:qn2Tn_ineq}. Let $G(u) \coloneqq \sum_{n = 0}^\infty q^{ 2 n^2 + 2 n } V_n(1 - 2 u^2)$. Since $(1 - 2 q^{2n} u + q^{4n}) (1 + 2 q^{2n} u + q^{4n}) = (1 + q^{4n})^2 - 4 q^{4n} u^2$ is positive and  minimized at $u = \pm 1$, by \eqref{eq:prod_Vn2}, we have $G(u) \ge G(\pm 1) = (k')^{- \frac{1}{4}} \sqrt{\frac{ \pi }{ 2 \EK(k) }} \prod_{n = 1}^\infty (1 - q^{2n})^4 (1 + q^{2n})^2$. Then by using \eqref{eq:prod_q2n}, we have \eqref{eq:q2n22nVn_ineq}.
\end{proof}

\begin{theorem}[Jacobi {\cite[THEOREMA 37.II]{Jacobi1829}}]
\label{the:q2}
Let $k(x)$ be the modulus corresponding to the nome $0 < x < 1$. Let $k'(x) \coloneqq \sqrt{1 - k(x)^2}$. Then one has
\begin{equation}
k(x^2) = \frac{ 1 - k'(x) }{ 1 + k'(x) } = \frac{ 1 - \sqrt{1 - k(x)^2} }{ 1 + \sqrt{1 - k(x)^2} }, k'(x^2) = \frac{ 2 \sqrt{k'(x)} }{ 1 + k'(x) }.
\end{equation}
\end{theorem}

\begin{lemma}
Let $0 < k < 1$, $k' \coloneqq \sqrt{1 - k^2}$, and $q \coloneqq \enum^{- \pi \frac{ \EK(k') }{ \EK(k) }  }$. Then one has
\begin{equation}
\label{eq:inv_1_dn_fourier}
\frac{ 1 }{ 1 + \dn(\EK(k) \theta / \pi, k) } = \frac{ \EK(k) - \EE(k) }{ k^2 \EK(k) } - \frac{  4 \pi^2   }{ k^2 \EK(k)^2 } \sum_{n = 1}^\infty \frac{ n \cos(n \theta) }{ q^{-2n} - q^{2n} }, \theta \in \imset_{\log(q^{-2})}.
\end{equation}
\end{lemma}
\begin{proof}
By using descending Landen transformations from $k$ to $k_1$ \cite[Sec.~19.8(ii), 22.7(i)]{DLMF}\cite[Sec.~37]{Jacobi1829}
\begin{align}
&k_1 \coloneqq \frac{ 1 - k' }{ 1 + k' }, k'_1 \coloneqq \sqrt{1 - k_1^2} = \frac{ 2 \sqrt{k'} }{ 1 + k' }, \EK(k_1) = \frac{ \EK(k) }{ 1 + k_1 } = \frac{ 1 + k' }{ 2 } \EK(k), \EE(k_1) = \frac{ k' \EK(k) + \EE(k) }{ 1 + k' }, q_1 \coloneqq \enum^{ - \pi \frac{ \EK(k'_1) }{ \EK(k_1) } } = q^2, \\
&\frac{ 1 }{ 1 + \dn(\EK(k) \theta / \pi, k) } = \frac{ 1 + k_1 - \dn( \EK(k_1) \theta / \pi, k_1 )^2 }{ 2 k_1 }, \theta \in \mathbb{C},
\end{align}
and \eqref{eq:dn2_fourier}, we have \eqref{eq:inv_1_dn_fourier}.
\end{proof}



\bibliographystyle{elsarticle-num}
\bibliography{ymkoyama.bbl}







\end{document}